\documentclass[a4paper,11pt]{amsart}
\usepackage[active]{srcltx}

\usepackage{graphicx}
\usepackage{euscript}
\usepackage{amsmath}
\usepackage{amssymb}
\usepackage{hyperref}

\newtheorem{thm}{Theorem}[section]
\newtheorem{lem}[thm]{Lemma}%
\newtheorem{prop}[thm]{Proposition}%
\newtheorem{cor}[thm]{Corollary}%
\theoremstyle{remark}
\newtheorem{remark}{Remark}[section] %

\theoremstyle{plain}

\numberwithin{equation}{section}

\def\RR{{\mathbb R}}

\def\veca{{\text{\boldmath$a$}}}
\def\vecb{{\text{\boldmath$b$}}}

\def\vecc{{\text{\boldmath$c$}}}

\def\vece{{\text{\boldmath$e$}}}
\def\vech{{\text{\boldmath$h$}}}
\def\veck{{\text{\boldmath$k$}}}

\def\vecell{{\text{\boldmath$\ell$}}}
\def\vecm{{\text{\boldmath$m$}}}

\def\vecq{{\text{\boldmath$q$}}}
\def\vecQ{{\text{\boldmath$Q$}}}
\def\vecp{{\text{\boldmath$p$}}}

\def\vecs{{\text{\boldmath$s$}}}
\def\vecS{{\text{\boldmath$S$}}}
\def\vect{{\text{\boldmath$t$}}}
\def\vecu{{\text{\boldmath$u$}}}
\def\vecv{{\text{\boldmath$v$}}}
\def\vecV{{\text{\boldmath$V$}}}
\def\vecw{{\text{\boldmath$w$}}}
\def\vecx{{\text{\boldmath$x$}}}

\def\vecy{{\text{\boldmath$y$}}}
\def\vecz{{\text{\boldmath$z$}}}
\def\vecalf{{\text{\boldmath$\alpha$}}}

\def\vecnull{{\text{\boldmath$0$}}}

\def\scrB{{\mathcal B}}

\def\scrD{{\mathcal D}}
\def\scrE{{\mathcal E}}
\def\scrF{{\mathcal F}}

\def\scrK{{\mathcal K}}
\def\scrL{{\mathcal L}}

\def\scrN{{\mathcal N}}

\def\fC{{\mathfrak C}}

\def\fZ{{\mathfrak Z}}

\def\tvecv{\widetilde{\vecv}}

\def\diag{\operatorname{diag}}

\def\id{\operatorname{id}}

\def\C{\operatorname{C{}}}

\def\GL{\operatorname{GL}}

\def\S{\operatorname{S{}}}

\def\SL{\operatorname{SL}}

\def\SO{\operatorname{SO}}

\def\T{\operatorname{T{}}}

\def\vol{\operatorname{vol}}

\def\trans{\,^\mathrm{t}\!}

\def\Onder#1#2#3#4#5{#1 \setbox0=\hbox{$#1$}\setbox1=\hbox{$#2$}
       \dimen0=.5\wd0 \dimen1=\dimen0 \dimen2=\dp0 \dimen3=\dimen2
       \advance\dimen0 by .5\wd1 \advance\dimen0 by -#4
       \advance\dimen1 by -.5\wd1 \advance\dimen1 by -#4
       \advance\dimen2 by -#3 \advance\dimen2 by \ht1
       \advance\dimen2 by 0.3ex \advance\dimen3 by #5
        \kern-\dimen0\raisebox{-\dimen2}[0ex][\dimen3]{\box1}
       \kern\dimen1}
\def\utilde#1{\Onder{#1}{\mbox{\char'176}}{0pt}{0ex}{0.5ex}} %
\def\myutilde#1{\Onder{#1}{\mbox{$\sim$}}{0pt}{0ex}{0.5ex}} %

\newcommand{\mytA}{{t}}
\newcommand{\mytB}{{t'}}

\newcommand{\Q}{\mathbb{Q}}
\newcommand{\R}{\mathbb{R}}
\newcommand{\Z}{\mathbb{Z}}

\renewcommand{\aa}{\mathsf{a}}
\newcommand{\kk}{\mathsf{k}}
\newcommand{\nn}{\mathsf{n}}
\newcommand{\sfrac}[2]{{\textstyle \frac {#1}{#2}}}
\newcommand{\col}{\: : \:}
\newcommand{\Si}{\mathcal{S}}
\newcommand{\bn}{\mathbf{0}}
\newcommand{\ta}{\utilde{a}}
\newcommand{\tu}{\utilde{u}}

\newcommand{\tM}{\myutilde{M}}
\newcommand{\tkk}{\utilde{\mathsf{k}}}

\newcommand{\ve}{\varepsilon}
\newcommand{\F}{\mathcal{F}}

\newcommand{\matr}[4]{\left( \begin{matrix} #1 & #2 \\ #3 & #4 \end{matrix} \right) }

\title{On the probability of a random lattice avoiding a large convex set}
\author{Andreas Str\"ombergsson}
\address{Department of Mathematics, Box 480, Uppsala University,
SE-75106 Uppsala, Sweden\newline
\rule[0ex]{0ex}{0ex} \hspace{8pt}{\tt astrombe@math.uu.se}}
\date{\today}


\thanks{The author is a Royal Swedish Academy of Sciences Research Fellow supported by a grant from the Knut and Alice Wallenberg Foundation, and is also supported by the Swedish Research Council, Research Grant 621-2007-6352.}

\begin{document}

\begin{abstract}
Given a set $\fC\subset\R^d$, let $p(\fC)$ be the probability
that a random $d$-dimensional
unimodular lattice, chosen according to Haar measure on 
$\SL(d,\Z)\backslash\SL(d,\R)$, is disjoint from $\fC\setminus\{\bn\}$.
For special convex sets $\fC$ we prove bounds on $p(\fC)$ 
which are sharp up to a scaling of $\fC$ by a constant.
We also prove bounds on a %
variant 
of $p(\fC)$ where the probability is conditioned %
on the random lattice containing a fixed given point $\vecp\neq\bn$.
Our bounds have applications, among other things,
to the asymptotic properties of the
collision kernel of the periodic Lorentz gas in the Boltzmann-Grad limit,
in arbitrary dimension $d$.
\end{abstract}

\maketitle

\section{Introduction}
\subsection{General introduction}

Let $X_1$ be the space of $d$-dimensional lattices $L\subset\R^d$
of covolume one,
equipped with its invariant probability measure $\mu$.
Let $p(\fC)=p^{(d)}(\fC)$ be the probability that a random lattice
$L\in X_1$ is disjoint from
a given subset $\fC\subset\R^d$ excluding the origin, i.e.
\begin{align}
p(\fC)=p^{(d)}(\fC)
:=\mu\bigl(\bigl\{L\in X_1\col L\cap\fC\setminus\{\bn\}=\emptyset
\bigr\}\bigr).
\end{align}
In recent years this probability $p(\fC)$
for certain specific choices of $\fC$,
as well as a conditional variant $p_\vecp(\fC)$ which we discuss below,
have appeared
as limit functions in a number of asymptotic problems in
number theory and mathematical physics, %
cf.\ 
\cite[Sec.\ 4]{Marklof00}, 
\cite{Elkies04}, %
\cite[Thm.\ 2]{SV}, \cite{partI}, \cite{partII}.

Our aim in the present paper is to give bounds on $p(\fC)$ and $p_\vecp(\fC)$
for special choices of convex %
sets $\fC$ of large volume,
and to point out some applications.
For a general measurable set $\fC$ the following
fundamental bound was recently
proved by Athreya and Margulis (\cite[Thm.\ 2.2]{athreyamargulis}):
\begin{align}\label{ATHREYAMARGULISRES}
p(\fC)\ll|\fC|^{-1},  %
\end{align}
where $|\fC|$ denotes the volume of $\fC$.
Here and throughout the paper we keep the convention that the implied constant
in any ``$\ll$'', ``$\asymp$'', or ``big-$O$'' depends only on $d$.
As Athreya and Margulis point out, \eqref{ATHREYAMARGULISRES} 
can be seen as a 'random' analogue to the classical Minkowski
theorem in the geometry of numbers.
We will be interested in giving stronger bounds than
\eqref{ATHREYAMARGULISRES} for special sets $\fC$.
(The bound \eqref{ATHREYAMARGULISRES} itself is %
easy for convex $\fC$; cf.\ Lemma \ref{TRIVPCBOUNDLEM2} below.)

The probability $p(\fC)$ for arbitrary sets $\fC$
was also studied in the late 1950's by Rogers \cite[part II]{Rogers58} and
Schmidt \cite{Schmidt58}, \cite{Schmidt59},
from a different point of view.
They obtained precise results on the size of $p(\fC)$
in the case of large dimension $d$ and not too large volume $|\fC|$.

\vspace{5pt}

The basic principle which we will use to obtain bounds on $p(\fC)$
is the following result, which 
as we will see
is an easy consequence of
classical reduction theory in $\SL_d(\R)$.
Let $\S_1^{d-1}$ be the unit sphere in $\R^d$ and let $\vol_{\S_1^{d-1}}$ be
the %
$(d-1)$-dimensional volume measure on $\S_1^{d-1}$.
\begin{prop}\label{GENPRINCBOUNDPROP}
Given $d\geq2$ there exist constants $k_1,k_2>0$ such that
for every measurable set $\fC\subset\R^d$ we have
\begin{align}\label{GENPRINCBOUNDPROPRES1}
p^{(d)}(\fC)\leq\min\biggl\{1,k_1\int_{\S_1^{d-1}}\int_{k_2r}^\infty
p^{(d-1)}\Bigl(a_1^{\frac1{d-1}}\fC\cap\vecv^\perp\Bigr)
\,\frac{da_1}{a_1^{d+1}}\, %
d\!\vol_{\S_1^{d-1}}(\vecv)
\biggr\},
\end{align}
where $r$ is the supremum of the radii of all 
$d$-dimensional balls contained in $\fC$,
and $a_1^{\frac1{d-1}}\fC\cap\vecv^\perp$ is viewed as a subset
of $\R^{d-1}$ via any 
volume preserving linear space isomorphism $\vecv^\perp\cong\R^{d-1}$.
\end{prop}
To bound $p(\fC)$ for a given ``nice'' set $\fC$, 
a reasonable strategy %
seems to be to first use the invariance relation
\begin{align}\label{PDINV}
p(\fC)=p(\fC M), \qquad\forall M\in\SL_d(\R),
\end{align}
so as to make the radius $r$ maximal or nearly maximal,
and then apply Proposition \ref{GENPRINCBOUNDPROP}.
In fact, if $\fC$ is convex, then the resulting bound is
\textit{sharp}, up to a scaling of $\fC$ %
by a constant factor only depending on $d$;
cf.\ Remark \ref{GENCONVEXPLREM} below.
(We remark that for $\fC$ convex we have 
$p(k_2\fC)\leq p(k_1\fC)$ for all $0<k_1<k_2$, 
cf.\ Lemma \ref{CONVPDMONOTONELEM}.
From now on, when we say that a bound 
$p(\fC)\leq B$ is ``sharp'',
we mean that there is a constant $k>0$ which only depends on $d$
such that $p(k\fC)\geq\min(\frac12,B)$.)

We will see that for an arbitrary convex set $\fC$ 
of large volume, if $\bn$ lies outside $\fC$ and not too near $\fC$,
then the Athreya-Margulis bound 
\eqref{ATHREYAMARGULISRES} is sharp, viz.\
$p(\fC)\gg |\fC|^{-1}$;
cf.\ Corollary~\ref{GENCONVAMSHARPCOR} below.
On the other hand we trivially have 
$p(\fC)=0$ whenever $\bn\in\fC$ and $\bn$ has
distance $\gg1$
to $\partial\fC$.  %
Hence the question about the order of magnitude of $p(\fC)$
for a general convex set $\fC$ of large volume
is interesting primarily when
$\bn$ lies fairly near $\partial\fC$. 

\subsection{Bounds on $p(\fC)$ for $\fC$ a ball, a cut ball,
a cone or a cylinder}

Our first main result is a sharp bound on $p(\fC)$ for any
$d$-dimensional \textit{ball} $\fC\subset\R^d$. 
Set %
\begin{align}\label{FBALLDEF}
F_{ball}^{(d)}(\tau;v):=
\begin{cases}
0&\text{if }\: \tau\leq -v^{-\frac2{d-1}}
\\[3pt]
v^{-2}&\text{if }\:|\tau|<v^{-\frac2{d-1}}
\\[3pt]
\tau^{\frac{d-1}2}v^{-1}&\text{if }\:\tau\geq v^{-\frac2{d-1}}.
\end{cases}
\end{align}
\begin{thm}\label{GENBALLTHM}
Given $d\geq2$ there exist constants $0<k_1<k_2$ 
such that for any $d$-dimensional ball $\fC\subset\R^d$
of volume $|\fC|\geq\frac12$,
\begin{align}\label{GENBALLTHMRES}
F_{ball}^{(d)}(\tau;k_2|\fC|)\leq 
p(\fC)\leq 
F_{ball}^{(d)}(\tau;k_1|\fC|)
\qquad\text{with }\:
\tau=\frac{\|\vecq\|-r}{\max(\|\vecq\|,r)},
\end{align}
where $r$ and $\vecq$ are the radius and center of $\fC$.
\end{thm}
Of course, by \eqref{PDINV}, Theorem \ref{GENBALLTHM} immediately 
implies a sharp bound on $p(\fC)$ in the
more general case of $\fC$ an arbitrary \textit{ellipsoid.}
Theorem \ref{GENBALLTHM} shows in particular that
for $\fC$ a ball or an ellipsoid, the Athreya-Margulis bound in
\eqref{ATHREYAMARGULISRES} can be improved to $|\fC|^{-2}$
whenever $\bn$ lies sufficiently near $\partial\fC$.
Regarding the restriction $|\fC|\geq\frac12$
in Theorem~\ref{GENBALLTHM}, 
note that if $\fC\subset\R^d$ is \textit{any} 
measurable set of volume $|\fC|<\frac12$ then
$\frac12<p(\fC)\leq1$ (cf.\ Lemma \ref{TRIVPLOWBOUNDLEM} below).
Note also that in Theorem~\ref{GENBALLTHM} we make no assumption on
$\fC$ being open or closed; in fact we have
$p(\fC)=p(\fC^\circ)=p(\overline\fC)$ for any set $\fC$ with
$|\partial\fC|=0$ (cf.\ Lemma \ref{PCONTLEM} below).

We remark that
Theorem \ref{GENBALLTHM} leads to 
good bounds on $p(\fC)$ also for %
many sets $\fC$ which are not ellipsoids, 
using the obvious fact that 
$p(\fC)\leq p(\fC')$ whenever $\fC'\subset\fC$.
For example Theorem~\ref{GENBALLTHM} implies a simple explicit sharp 
bound on $p(\fC)$ for any convex body $\fC$
such that $\partial\fC$ has pinched positive curvature;
cf.\ Corollary~\ref{GENCURVCOR} below.

\vspace{5pt}

Our second main result concerns a special situation
with $\bn\in\partial\fC$, tailored to suit our applications:
We take $\fC$ to be a ``cut ball'', by which we mean an intersection of
a $d$-dimensional ball and a half space,
and we assume that $\bn$ belongs to the flat part of $\partial\fC$.
Set 
\begin{align*}
F_{\substack{cut\\ball}}^{(d)}(e;t;v)=\begin{cases}
v^{-2}(1+v^{\frac2d}t^{1-\frac1d})
&\text{if }\: t=0\:\text{ or }\:e\leq v^{-\frac2d}t^{\frac1d-1}
\\
0&\text{if }\: e>v^{-\frac2d}t^{\frac1d-1}.
\end{cases} 
\end{align*}

\begin{thm}\label{MAINTHM2}
Given $d\geq2$ there exist constants $0<k_1<k_2$ such that the following
holds. 
Let $B$ be a $d$-dimensional ball 
containing $\bn$ in its closure, let $\vecw$ be a unit vector,
and assume that the intersection
\begin{align*}
\fC:=%
B\cap\{\vecx\in\R^d\col \vecw\cdot\vecx>0\}
\end{align*}
has volume $|\fC|\geq\frac12$. 
Let $r$ and $\vecp$ be the radius and center of $B$,
let $r'$ and $\vecq$ be the radius and center of the 
$(d-1)$-dimensional ball $B\cap\vecw^\perp$, and set
\begin{align*}
t:=1-\frac{\vecw\cdot\vecp}r\in[0,2)\quad\text{and}\quad
e=\frac{r'-\|\vecq\|}{r'}\in[0,1]
\end{align*}
(we leave $e$ undefined when $r'=0$).
Then
\begin{align}\label{MAINTHM2RES}
F_{\substack{cut\\ball}}^{(d)}\Bigl(e;t;k_2|\fC|\Bigr)\leq
p(\fC)\leq F_{\substack{cut\\ball}}^{(d)}\Bigl(e;t;k_1|\fC|\Bigr).
\end{align}
\end{thm}

\begin{figure}
\begin{center}
\framebox{
\begin{minipage}{0.4\textwidth}
\unitlength0.1\textwidth
\begin{picture}(9,5.8)(0,0)
\put(2,0.5){\includegraphics[width=0.6\textwidth]{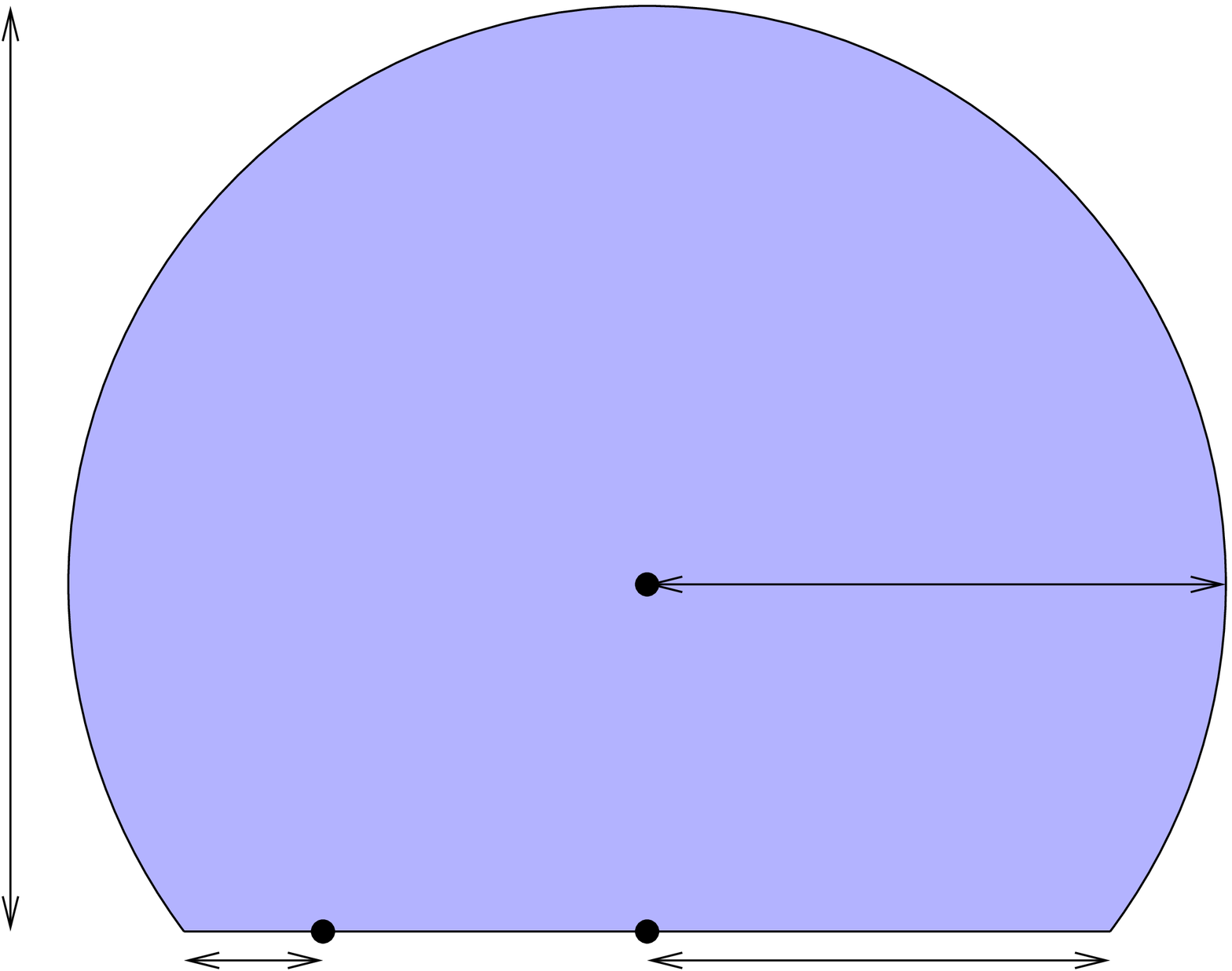}}
\put(-0.1,2.2){$(2-t)r$}
\put(6.3,2.6){$r$}
\put(3,0.0){$er'$}
\put(6,0.0){$r'$}
\put(4.9,2.6){$\vecp$}
\put(4.9,1.0){$\vecq$}
\put(3.5,0.9){$\bn$}
\end{picture}
\end{minipage}
}
\framebox{
\begin{minipage}{0.4\textwidth}
\unitlength0.1\textwidth
\begin{picture}(9,5.8)(0,0)
\put(2,0.5){\includegraphics[width=0.6\textwidth]{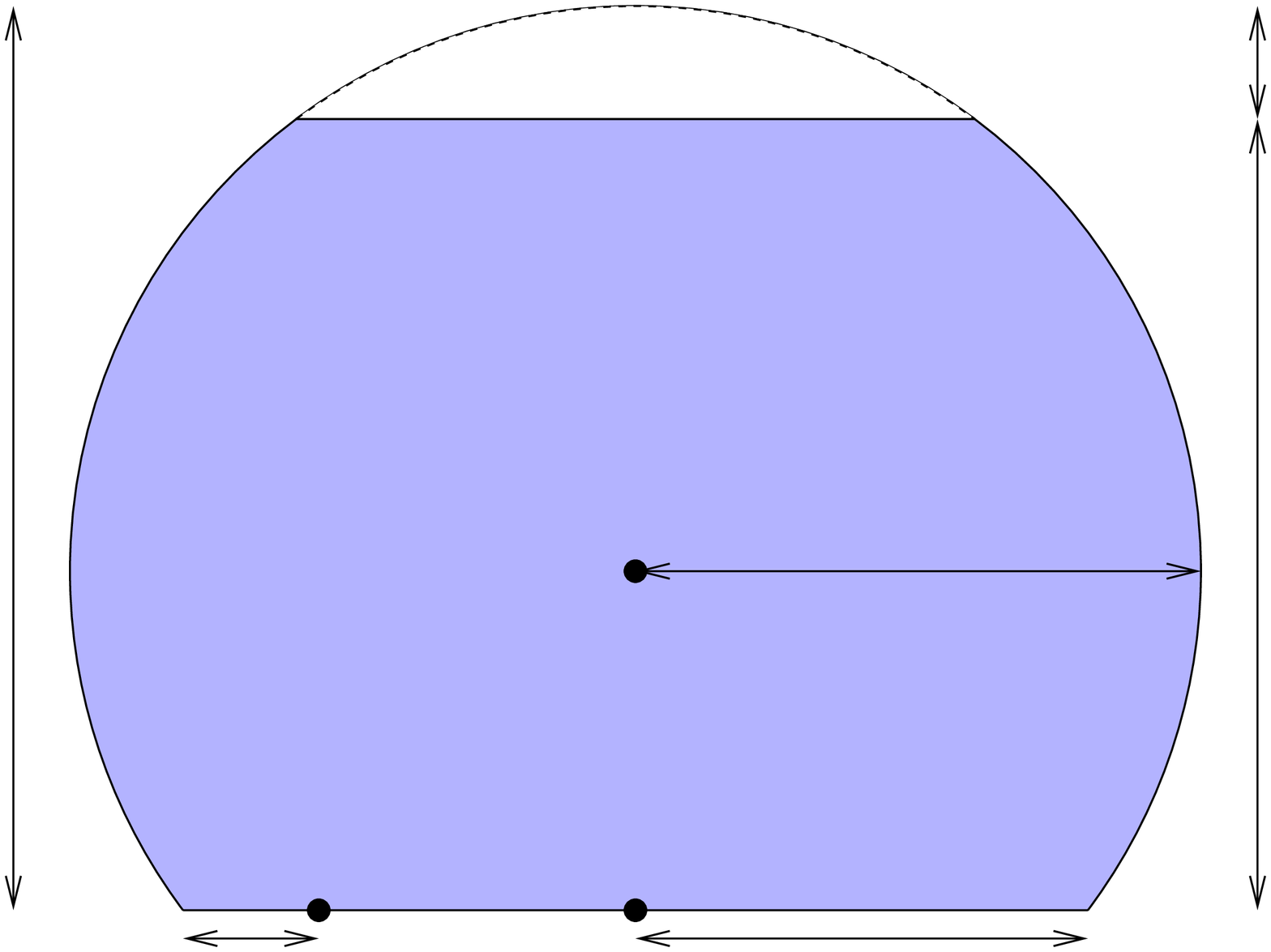}}
\put(-0.1,2.2){$(2-t)r$}
\put(6.3,2.6){$r$}
\put(3,0.0){$er'$}
\put(6,0.0){$r'$}
\put(4.9,2.6){$\vecp$}
\put(4.9,1.0){$\vecq$}
\put(3.5,0.9){$\bn$}
\put(8.1,2.2){$(t'-t)r$}
\put(8.1,4.6){$(2-t')r$}
\end{picture}
\end{minipage}
}
\end{center}
\caption{Left: The cut ball in Theorem \ref{MAINTHM2}, in dimension $d=2$. Right: The doubly cut ball in Corollary \ref{DCBCOR}.} 
\label{FIGCUTBALL}
\end{figure}

Note that the special case $t=0$ in Theorem \ref{MAINTHM2}
is the same as the case $\tau=0$ in Theorem~\ref{GENBALLTHM},
saying that $p(\fC)\asymp|\fC|^{-2}$ when $\fC$ is a large ball
with $\bn\in\partial\fC$.
On the other hand if we keep $t$ bounded away from zero then
$F_{\substack{cut\\ball}}^{(d)}(e;t;v)\asymp v^{-2+\frac2d}$
for $ev^{\frac2d}$ small, and
$F_{\substack{cut\\ball}}^{(d)}(e;t;v)=0$
for $ev^{\frac2d}$ large.
Using this case of Theorem \ref{MAINTHM2} together with
the monotonicity %
$\fC'\subset\fC\Rightarrow p(\fC)\leq p(\fC')$,
leads to sharp bounds on $p(\fC)$ for many other sets $\fC$
such that $\bn$ belongs to a large flat part of $\partial\fC$.
We state this as a corollary for the useful special cases of %
a cone or a cylinder.
Set
\begin{align*}
F_{cone}^{(d)}(e;v)=\begin{cases}
v^{-2+\frac2d}&\text{if }\: e\leq v^{-\frac2d}
\\
0&\text{if }\: e>v^{-\frac2d}.
\end{cases}
\end{align*}

\begin{cor}\label{CONECYLCOR}
Given $d\geq2$ there exist constants $0<k_1<k_2$ such that the following
holds. 
Let $B\subset\R^d$ be a $(d-1)$-dimensional ball containing $\bn$ in its
closure, let $\vecp\in\R^d$ be a point, and let $\fC$ be the
cone which is the convex hull of $B$ and $\vecp$.
Set $e=\frac{r-\|\vecq\|}r$ where $r$ and $\vecq$ are the
radius and center of $B$. Assume $|\fC|\geq\frac12$. Then
\begin{align}\label{CONECYLCORRES}
F_{cone}^{(d)}(e;k_2|\fC|)\leq p(\fC)\leq F_{cone}^{(d)}(e;k_1|\fC|).
\end{align}
Exactly the same bound holds (with new $k_1,k_2$)
if we instead take $\fC$ to be the cylinder 
which is the convex hull of 
$B$ and some translate $B'$ of $B$,
and we again assume $|\fC|\geq\frac12$.
\end{cor}
Another observation which will be useful for us is that, 
again using %
$\fC'\subset\fC\Rightarrow p(\fC)\leq p(\fC')$,
Theorem \ref{MAINTHM2} may be generalized to the case of a
``doubly cut ball'', where the two cuts are parallel
(see Figure \ref{FIGCUTBALL}).
\begin{cor}\label{DCBCOR}
Given $d\geq2$ there exist constants $0<k_1<k_2$ such that the following
holds. 
Given any $B$, $\vecw$, $r$, $t$, $e$ as in Theorem \ref{MAINTHM2}
and any $t'\in(t,2]$, we set
\begin{align*}
\fC:=%
\{\vecx\in B\col0<\vecw\cdot\vecx<(t'-t)r\}.
\end{align*}
Assume that $\fC$ has volume $|\fC|\geq\frac12$. Then
\begin{align}\label{DCBCORRES}
F_{\substack{cut\\ball}}^{(d)}\Bigl(e;\frac t{t'};k_2|\fC|\Bigr)\leq
p(\fC)\leq F_{\substack{cut\\ball}}^{(d)}\Bigl(e;\frac t{t'};k_1|\fC|\Bigr).
\end{align}
\end{cor}

\begin{figure}
\begin{center}
\framebox{
\begin{minipage}{0.4\textwidth}
\unitlength0.1\textwidth
\begin{picture}(10,5.0)(0,0)
\put(2,0.5){\includegraphics[width=0.6\textwidth]{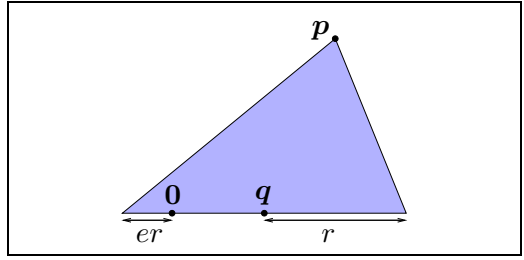}}
\put(2.3,0.1){$er$}
\put(6.2,0.1){$r$}
\put(6.0,4.5){$\vecp$}
\put(4.8,1.0){$\vecq$}
\put(2.9,0.9){$\bn$}
\end{picture}
\end{minipage}
}
\end{center}
\caption{The cone in Corollary \ref{CONECYLCOR} and Theorem \ref{CONEAPEXTHM}, in dimension $d=2$.}\label{FIGCONE}
\end{figure}


\subsection{The conditional probability $p_\vecp(\fC)$; the case of $\fC$ a cone and applications to statistics of directions to lattice points}

We next turn
to a function related to $p(\fC)$: %
the \textit{conditional} probability of 
$L\cap\fC\setminus\{\bn,\vecp\}=\emptyset$
given that $L$ contains 
a fixed point $\vecp\neq\bn$.
We denote this probability by $p_\vecp(\fC)$,
cf.\ Section~\ref{X1RECOLLECTSEC} below for the precise definition.
We will prove non-trivial bounds on $p_\vecp(\fC)$
in two special cases %
both of which have important applications.
In both cases we will have $\bn,\vecp\in\partial\fC$.

The first case is that of an open cone with $\bn$
in its base and \textit{apex} $\vecp$.
In this case it turns out that $p_\vecp(\fC)$ satisfies %
the same kind of upper and lower bounds as $p(\fC)$ 
(cf.\ Corollary \ref{CONECYLCOR}):
\begin{thm}\label{CONEAPEXTHM}
Given $d\geq2$ there exist constants $0<k_1<k_2$ such that the following
holds. 
Let $B\subset\R^d$ be a relatively open $(d-1)$-dimensional ball 
with $\bn\in B$, let $\vecp$ be a point $\neq\bn$, and let $\fC$ be the
open cone which is the interior of the convex hull of $B$ and $\vecp$.
Set $e=\frac{r-\|\vecq\|}r$ where $r$ and $\vecq$ are the
radius and center of $B$. Assume $|\fC|\geq\frac12$. Then
\begin{align}\label{CONEAPEXTHMRES}
F_{cone}^{(d)}(e;k_2|\fC|)  %
\leq p_\vecp(\fC)\leq F_{cone}^{(d)}(e;k_1|\fC|).
\end{align}
\end{thm}
As an application, Theorem \ref{CONEAPEXTHM}
yields information on the
tail behavior of a certain limit density
related to the fine-scale statistics
of directions to lattice points in a fixed $d$-dimensional lattice
(cf.\ Marklof and Str\"ombergsson \cite[Sections 1.2 and 2.2-4]{partI}).
To describe the problem,
fix a lattice $\scrL\subset\R^d$ of covolume one.
Let us write $\scrB_T^d$ for the open ball in $\R^d$ with center $\bn$
and radius $T$, and consider, for large $T$, 
the set of non-zero lattice points in $\scrB_T^d$.
We are interested in the corresponding \textit{directions},
\begin{align}\label{DIRSET}
\|\vecm\|^{-1}\vecm\in\S_1^{d-1},\qquad
\text{for }\:\vecm\in\scrL\cap\scrB_T^d\setminus\{\bn\}. 
\end{align}
It is well known that, as $T\to\infty$, these points become uniformly
distributed on $\S_1^{d-1}$ with respect to the volume measure
$\vol_{\S_1^{d-1}}$.
We are interested in the fine-scale statistics of these directions,
i.e.\ we wish to study the behavior of the point set in \eqref{DIRSET}
when rescaled in such a way that we have on average a constant number 
of points per unit 
volume.
This question was studied for $d=2$ by
Boca, Cobeli and Zaharescu \cite{Boca00}.
Later a general result on the limit statistics in arbitrary dimension
$d$ was proved by
Marklof and 
Str\"ombergsson in 
\cite[Thm.\ 2.1]{partI} (cf.\ also \cite[Sec.\ 2.4]{partI});
we will recall this result here.

Given $\vecv\in\S_1^{d-1}$ and $\sigma>0$ we
let $\scrD_T(\sigma,\vecv)$ be the open disc inside $\S_1^{d-1}$ with
center $\vecv$ and volume $\vol_{\S_1^{d-1}}(\scrD)=\sigma d T^{-d}$;
thus the radius of $\scrD_T(\sigma,\vecv)$ is $\asymp T^{-\frac d{d-1}}$.
We denote by $\scrN_T(\sigma,\vecv)$ the number of directions to
lattice points 
which lie in $\scrD_T(\sigma,\vecv)$:
\begin{align}\label{NTDEF}
\scrN_T(\sigma,\vecv)=\#\bigl\{\vecm\in\scrL\cap\scrB_T^d\setminus\{\bn\}
\col \|\vecm\|^{-1}\vecm\in\scrD_T(\sigma,\vecv)\bigr\}.
\end{align}
The motivation for the definition of $\scrD_T(\sigma,\vecv)$ 
is that it implies that the expectation value of $\scrN_T(\sigma,\vecv)$
for random $\vecv$ is asymptotically equal to $\sigma$ as $T\to\infty$,
cf.\ \cite[(2.11)]{partI}.
In particular this means that the distance (viz., the angle) 
from a random direction to the nearest lattice direction
is typically on the order of $T^{-\frac d{d-1}}$.
Now, as a special case of \cite[Thm.\ 2.1]{partI},
for any Borel probability measure $\lambda$ on $\S_1^{d-1}$,
and for any $\sigma\geq0$ and $r\in\Z_{\geq0}$, 
we have that the limit
\begin{align}\label{E0RSLIMIT}
E_\bn(r,\sigma):=
\lim_{T\to\infty}\lambda\bigl(\bigl\{\vecv\in\S_1^{d-1}\col
\scrN_T(\sigma,\vecv)
=r\bigr\}\bigr)
\end{align}
exists. 
In other words, if $\vecv$ is picked at random according to $\lambda$,
then the random variable $\scrN_T(\sigma,\cdot)$
has a limit distribution as $T\to\infty$, which 
is independent of $\scrL$!
The limit probability in \eqref{E0RSLIMIT} is given by
\begin{align}\label{E0RSDEF}
E_\bn(r,\sigma)=\mu\bigl(\bigl\{L\in X_1\col\#(L\cap\fC\setminus\{\bn\})=r
\bigr\}\bigr),
\end{align}
where $\fC\subset\R^d$ is any $d$-dimensional 
cone with apex $\bn$ and volume $\sigma$.

The case $r=0$ in \eqref{E0RSLIMIT} is of particular interest since it
corresponds to the ``spherical contact'', or ``empty space'', 
distribution function for our set of directions;
cf., e.g., \cite[p.\ 105]{dSwKjM95}.
To make this explicit,
let us write $\varphi_T(\vecv)$ for the smallest angle 
from the point $\vecv\in\S_1^{d-1}$ 
to a point in our set \eqref{DIRSET},
\begin{align}\label{PHITVDEF}
\varphi_T(\vecv)=\min\bigl\{\varphi(\vecv,\vecm)\col
\vecm\in\scrL\cap\scrB_T^d\setminus\{\bn\}
\bigr\}.
\end{align}
Here and from now on $\varphi(\vecv,\vecw)$ denotes the angle between
any two non-zero vectors $\vecv,\vecw$. 
It follows from \eqref{E0RSLIMIT} that
the properly scaled random variable
$T^{\frac d{d-1}}\varphi_T(\vecv)$ has a limit distribution
as $T\to\infty$:
For any $x\geq0$ we have
\begin{align}\label{EMPTYSPACELIMITDISTR}
F_\bn(x):=\lim_{T\to\infty}\lambda\bigl(\bigl\{\vecv\in\S_1^{d-1}\col
T^{\frac d{d-1}}\varphi_T(\vecv)\leq x\bigr\}\bigr)
=1-E_\bn\bigl(0,\kappa_dx^{d-1}\bigr),
\end{align}
where $\kappa_d:=d^{-1}\vol(\scrB_1^{d-1})$.

Note that $E_\bn(0,\sigma)=p(\fC)$ in our notation,
with $\fC$ a cone with apex $\bn$ and volume $\sigma$,
and it is an easy consequence of the theory which we will develop in
Section \ref{GENCONVEXSEC} 
that $E_\bn(0,\sigma)\asymp\sigma^{-1}$ as $\sigma\to\infty$
(cf.\ in particular Corollary \ref{ATHREYAMFROMBELOWPROP1}).
Hence the spherical contact limit distribution function has the
tail asymptotics 
\begin{align}\label{EMPTYSPACEDISTRTAIL}
1-F_\bn(x)\asymp x^{1-d}\qquad\text{as }\:x\to\infty.
\end{align}
In particular this large tail asymptotics implies that there are many
large ``deserts'' in the
set of directions to the points of $\scrL$, to an extent that 
the $(d-1)$th moment of the random variable $T^{\frac d{d-1}}\varphi_T(\vecv)$
tends to $\infty$ as $T\to\infty$.
The main point we wish to make here, however, is that using
Theorem \ref{CONEAPEXTHM} we are even able to give sharp bounds as
$x\to\infty$ on the \textit{density} corresponding to $F_\bn(x)$.
Indeed, it follows from 
\cite[Remark 2.2]{partI} that $F_\bn(x)\in\C^1(\R_{>0})$,
and by \cite[(8.48)]{partI} we have
\begin{align*}
f_\bn(x):=F_\bn'(x)
=\kappa_d(d-1)x^{d-2}\frac1{\vol_{d-1}(B)}\int_B p_\vecp(\fC)
\,d\vecp,
\end{align*}
where $\fC$ is a cone with apex $\bn$ and volume $\kappa_d x^{d-1}$,
and $B$ is the base of $\fC$,
i.e.\ the $(d-1)$-dimensional ball with the property that $\fC$ is the 
convex hull of $B$ and $\bn$;
also $d\vecp$ denotes the standard $(d-1)$-dimensional Lebesgue measure.
Using now $p_\vecp(\fC)=p_\vecp(\vecp-\fC)$ (cf.\ \eqref{PPFSYMM} below)
and Theorem \ref{CONEAPEXTHM}, we conclude that
$\vol_{d-1}(B)^{-1}\int_B p_\vecp(\fC)\,d\vecp\asymp|\fC|^{-2}$
when $|\fC|$ is large, and hence we obtain:
\begin{cor}\label{LATDIRCOR}
$f_\bn(x)\asymp x^{-d}$ as $x\to\infty$.
\end{cor}

Using similar arguments as in \cite{partIV},
building on the present paper,
it should even be possible to obtain an asymptotic formula
for $f_\bn(x)$ as $x\to\infty$,
in arbitrary dimension $d$.
We hope to carry this out in a later paper.
Note that for $d=2$ one knows a completely explicit formula for
$f_\bn(x)$; %
cf.\ Boca, Cobeli and Zaharescu \cite[Cor.\ 0.4]{Boca00}. 

We stress that \cite[Thm.\ 2.1]{partI} is 
more general than \eqref{E0RSLIMIT}--\eqref{E0RSDEF}; 
in particular it applies also in the 
case when we consider the set of directions to any fixed 
\textit{shifted} lattice,
i.e.\ we replace $\scrL$ by $\vecq+\scrL$
in \eqref{DIRSET}--\eqref{NTDEF}, for any fixed $\vecq\in\R^d$.
Here, if $\vecq\in\R^d$ is not a rational
linear combination of points in $\scrL$ then the limit distribution
given by \cite[Thm.\ 2.1]{partI} is in fact 
universal
in the sense that it is independent of both $\scrL$ \textit{and} $\vecq$.
In this case, the limit spherical contact density function is,  
by \cite[(8.48)]{partI},
\begin{align*}
f(x)=\kappa_d(d-1)x^{d-2}\frac1{\vol_{d-1}(B)}\int_B p(\vecp-\fC)
\,d\vecp.
\end{align*}
Hence also in this case we have the asymptotic relation
$f(x)\asymp x^{-d}$ as $x\to\infty$,
now as a consequence of Corollary \ref{CONECYLCOR}.

\subsection{The case of $\fC$ a cylinder, and applications to the periodic Lorentz gas in the Boltzmann-Grad limit}

The second case in which we prove a non-trivial bound on $p_\vecp(\fC)$
is that of an open cylinder, with $\bn$ and $\vecp$ lying
on its opposite bases. %
The function $p_\vecp(\fC)$ in this case occurs as the collision
kernel between consecutive collisions, $\Phi_\bn(\xi,\vecw,\vecz)$,
in the periodic Lorentz gas in the Boltzmann-Grad limit,
cf.\ \cite{partI}, \cite{partII}, \cite{partIII}. 
It was the task of understanding
the asymptotics of this kernel which %
led us to undertake the
present work; in fact we %
make crucial use of 
both Theorem \ref{MAINTHM2} and Theorem~\ref{CONEAPEXTHM}
in the proof of Theorem~\ref{CYLINDER2PTSMAINTHM} below.

\begin{figure}
\begin{center}
\framebox{
\begin{minipage}{0.55\textwidth}
\unitlength0.1\textwidth
\begin{picture}(10,4.8)(0,0)
\put(2,0.5){\includegraphics[width=0.6\textwidth]{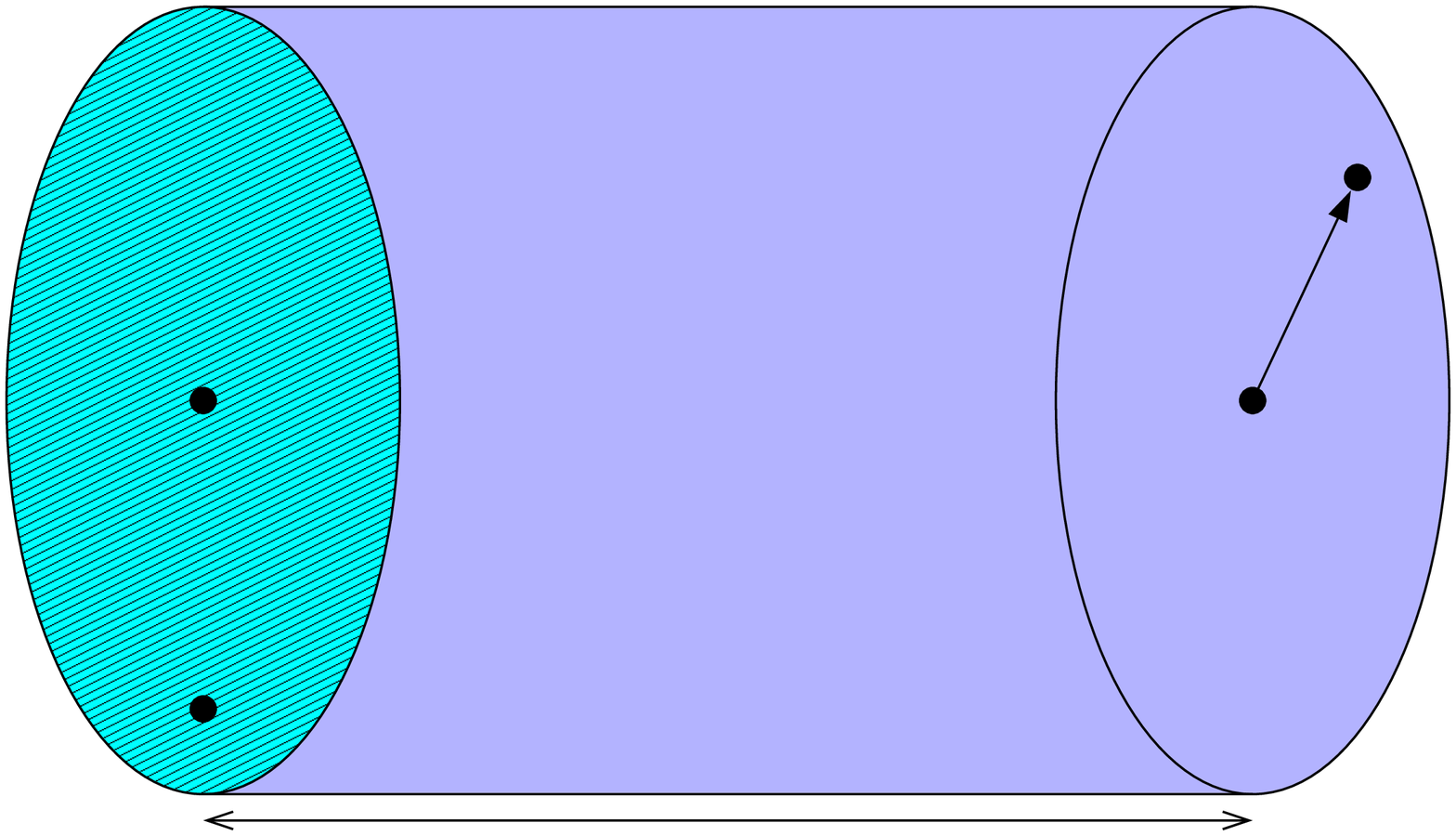}}
\put(5.0,0.2){$\xi$}
\put(2.9,1.1){$\bn$}
\put(2.9,2.45){$\vecz$}
\put(7.25,3.25){$\vecp$}
\put(7.43,2.53){$\vecw$}
\end{picture}
\end{minipage}
}
\end{center}
\caption{The cylinder $\fC$ in the definition of $\Phi_\bn(\xi,\vecw,\vecz)$, in dimension $d=3$.}\label{FIGPHI0}
\end{figure}

For any $\xi>0$ and $\vecw,\vecz\in\scrB_1^{d-1}$
(i.e.\ $\vecw,\vecz\in\R^{d-1}$, $\|\vecw\|,\|\vecz\|<1$),
$\Phi_\bn(\xi,\vecw,\vecz)$ is defined as
\begin{align}\label{PHI0DEF1}
\Phi_\bn(\xi,\vecw,\vecz)=p_\vecp(\fC), 
\end{align}
where $\fC$ is the cylinder (cf.\ Figure \ref{FIGPHI0})
\begin{align}\label{PHI0DEF}
\begin{cases}\fC=\bigl\{(x_1,\ldots,x_d)\in\R^d\col
0<x_1<\xi,\:\bigl\|(x_2,\ldots,x_d)-\vecz\bigr\|<1\bigr\};
\\
\vecp=(\xi,\vecz+\vecw).
\end{cases}
\end{align}
Note that $\Phi_\bn(\xi,\vecw,\vecz)$ only depends on the
four scalars $\xi,\|\vecw\|,\|\vecz\|,\varphi(\vecw,\vecz)$.
Note also that if $\fC$ is an \textit{arbitrary} open cylinder 
with ellipsoidal cross section and 
with $\bn$ and $\vecp$ lying on the opposite bases of $\fC$,
then $p_\vecp(\fC)$ can be expressed in terms of $\Phi_\bn(\xi,\vecw,\vecz)$
(cf.\ \eqref{PPDINV} below).

Before stating our main results on 
$\Phi_\bn(\xi,\vecw,\vecz)$, we briefly explain the connection
with the periodic Lorentz gas, borrowing from the presentation in
\cite{partIV}.
For more details see \cite{partI}, \cite{partII}, \cite{partIII}.

The periodic Lorentz gas describes an ensemble of non-interacting
point particles in an infinite periodic array of spherical scatterers.
Specifically, for a fixed lattice $\scrL\subset\R^d$ of covolume one
and given $\rho>0$ (small), we take the scatterers to be
all the open balls $\scrB_\rho^d+\vecell$ with $\vecell\in\scrL$.
We denote by $\scrK_\rho\subset\RR^d$ the complement of the union of these balls
(the ``billiard domain''), and $\T^1(\scrK_\rho)=\scrK_\rho\times\S_1^{d-1}$ 
its unit tangent bundle (the ``phase space''), with $\vecq(t)\in\scrK_\rho$ 
the position and $\vecv(t)\in\S_1^{d-1}$ the velocity of the particle 
at time $t$. The dynamics of a particle in the Lorentz gas is defined as the motion with unit speed along straight lines, and specular reflection at the 
balls $\scrB_\rho^d+\vecell$ ($\vecell\in\scrL$).
We may in fact also permit other scattering processes, such as the scattering map of a Muffin-tin Coulomb potential (cf.~\cite{partII}). 
A dimensional argument shows that in the Boltzmann-Grad limit $\rho\to 0$ the free path length scales like $\rho^{-(d-1)}$, i.e., the inverse of the total scattering cross section of an individual scatterer. It is therefore natural to rescale space and time by introducing the macroscopic coordinates (see Figure \ref{figLorentz})
\begin{equation}\label{macQV}
	\big(\vecQ(t),\vecV(t)\big) = \big(\rho^{d-1} \vecq(\rho^{-(d-1)} t),\vecv(\rho^{-(d-1)} t)\big) .
\end{equation}
The time evolution of a particle with initial data $(\vecQ,\vecV)$ is then described by the billiard flow
\begin{equation}\label{LP}	(\vecQ(t),\vecV(t))=F_{t,\rho}(\vecQ,\vecV) .
\end{equation}
We extend the dynamics to the inside of each scatterer 
trivially, i.e., set $F_{t,\rho}=\id$ whenever $\vecQ\in\scrB_\rho^d+\scrL$;
thus the relevant phase space is now $\T^1(\RR^d)$,
the unit tangent bundle of $\RR^d$.

\begin{figure}
\begin{center}
\framebox{
\begin{minipage}{0.3\textwidth}
\unitlength0.1\textwidth
\begin{picture}(10,10)(0,0)
\put(0.5,1){\includegraphics[width=0.9\textwidth]{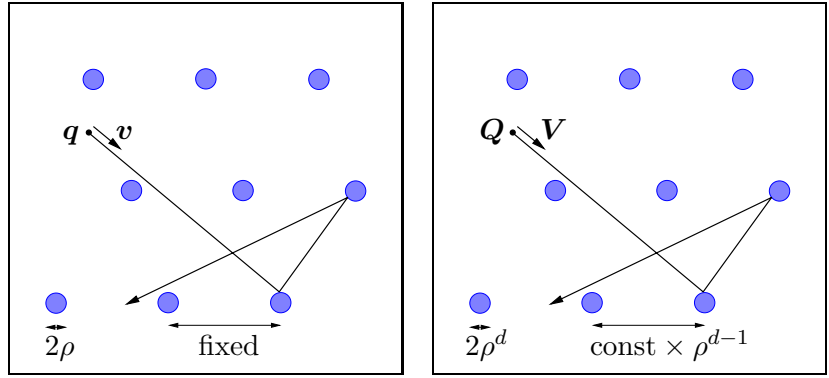}}
\put(1,6.4){$\vecq$} \put(2.5,6.4){$\vecv$}
\put(0.5,0.3){$2\rho$} \put(4.8,0.3){fixed}
\end{picture}
\end{minipage}
}
\begin{minipage}{0.1\textwidth}\end{minipage}
\begin{minipage}{0.1\textwidth}\end{minipage}
\framebox{
\begin{minipage}{0.3\textwidth}
\unitlength0.1\textwidth
\begin{picture}(10,10)(0,0)
\put(0.5,1){\includegraphics[width=0.9\textwidth]{lorentzgas.eps}}
\put(0.8,6.4){$\vecQ$} \put(2.5,6.4){$\vecV$}
\put(0.4,0.3){$2\rho^d$} \put(3.9,0.3){$\text{const}\times\rho^{d-1}$}
\end{picture}
\end{minipage}
}
\end{center}
\caption{Left: The periodic Lorentz gas in ``microscopic'' coordinates---the lattice $\scrL$ remains fixed as the radius $\rho$ of the scatterer tends to zero. Right: The periodic Lorentz gas in ``macroscopic'' coordinates ---both the lattice constant and the radius of each scatter tend to zero, in such a way that the mean free path length remains finite.} \label{figLorentz}
\end{figure}

Let us fix an arbitrary probability measure $\Lambda$ on $\T^1(\RR^d)$. For random initial data $(\vecQ_0,\vecV_0)$ with respect to $\Lambda$, we can then view the billiard flow $\{F_{t,\rho} : t> 0\}$ as a stochastic process. The central result of \cite{partI}, \cite{partII} is that, if $\Lambda$ is absolutely continuous with respect to Lebesque measure, the billiard flow converges in the Boltzmann-Grad limit $\rho\to0$ to a random flight process $\{\Xi(t):t > 0\}$, which is defined as the flow with unit speed along a random piecewise linear curve, whose path segments $\vecS_1,\vecS_2,\vecS_2,\ldots\in\RR^d$ are generated by a Markov process with memory two. Specifically, if we set $\xi_j=\|\vecS_j\|$ and $\vecV_{j-1} =\frac{\vecS_j}{\|\vecS_j\|}$ for $j=1,2,3,\ldots$,
then the distribution of the first $n$ path segments is given by the probability density
\begin{multline}\label{MAINRESULTPARTII}
	\Lambda'(\vecQ_0,\vecV_0) p(\vecV_0,\xi_1,\vecV_1) p_\vecnull(\vecV_0,\vecV_1,\xi_2,\vecV_2)\cdots \\ \cdots p_\vecnull(\vecV_{n-3},\vecV_{n-2},\xi_{n-1},\vecV_{n-1}) \int_{\S_1^{d-1}} p_\vecnull(\vecV_{n-2},\vecV_{n-1},\xi_n,\vecV_n)\, d\!\vol_{\S_1^{d-1}}(\vecV_n) ,
\end{multline}
see Theorem 1.3 and Section 4 in \cite{partII}.
The transition kernels $p$ and $p_\bn$ in \eqref{MAINRESULTPARTII}
are given by
\begin{align}
&	p(\vecV,\xi,\vecV_+) 
	=\sigma(\vecV,\vecV_+)\, \Phi\big(\xi,\vecb(\vecV,\vecV_+)\big),
\\\label{P0V0VXIVPLDEF}
&	p_{\vecnull}(\vecV_0,\vecV,\xi,\vecV_+) 
	=\sigma(\vecV,\vecV_+)\, \Phi_\vecnull\big(\xi,\vecb(\vecV,\vecV_+),
	-\vecs(\vecV,\vecV_0)\big),
\end{align}
where $\sigma(\vecV,\vecV_+)$ is the differential cross section,
$\Phi_\bn$ is the function defined in \eqref{PHI0DEF1}--\eqref{PHI0DEF},
\begin{align}\label{PHIXIWDEF}
\Phi(\xi,\vecw)=\int_\xi^\infty\int_{\scrB_1^{d-1}}\Phi_\bn(\eta,\vecw,\vecz)
\,d\vecz\,d\eta,
\end{align}
and $\vecb(\vecV,\vecV_+)$ and $\vecs(\vecV,\vecV_0)$ are the
impact and exit parameters (cf.~Figure \ref{figTP}),
both measured in units of the scattering radius, and considered as
vectors in $\R^{d-1}$ via a fixed Euclidean space isomorphism
$\vecV^\perp\cong\R^{d-1}$
(thus $\vecb(\vecV,\vecV_+),\vecs(\vecV,\vecV_0)\in\scrB_1^{d-1}$).

\begin{remark}
The function $\Phi$ may alternatively be defined by
$\Phi(\xi,\vecz)=p(\fC)$ where $\fC$ is the cylinder in \eqref{PHI0DEF}.
Cf.\ \cite[Thm.\ 4.4 ($\vecalf\notin\Q$), (4.16), (8.32)]{partI},
and \cite[Remark 6.2]{partII}.
\end{remark}

\begin{remark} 
If the scattering map is given by specular reflection (as in the original Lorentz gas), we have explicitly $\sigma(\vecV,\vecV_+)= \frac 14\|\vecV-\vecV_+\|^{3-d}$
for the scattering cross section, and 
\begin{equation}
\vecs(\vecV,\vecV_0)= -\frac{(\vecV_0 K(\vecV))_\perp}{\|\vecV_0 K(\vecV)-\vece_1\|}, \qquad \vecb(\vecV,\vecV_+)= \frac{(\vecV_+ K(\vecV))_\perp}{\|\vecV_+ K(\vecV)-\vece_1\|},
\end{equation}
for the exit and impact parameters. 
Here $\vecx_\perp$
denotes the orthogonal projection of $\vecx\in\RR^d$ onto $\vece_1^{\perp}=\{0\}\times\RR^{d-1}$,
and for each $\vecV\in\S_1^{d-1}$ we have fixed a rotation
$K(\vecV)\in\SO(d)$ with $\vecV K(\vecV)=\vece_1$.
\end{remark}

It can be seen from \eqref{MAINRESULTPARTII} -- \eqref{P0V0VXIVPLDEF} that
$\Phi_\bn(\xi,\vecb,-\vecs)$ is the limiting probability density 
(in the limit $\rho\to 0$, and with respect to the reference measure
$d\xi\,d\vecb$) of hitting, from a given scatterer
with exit parameter $\vecs$, the next scatterer at time
$\rho^{-(d-1)}\xi$ 
with impact parameter $\vecb$.
Again, cf.~Figure \ref{figTP}.
Similarly
$\Phi\big(\xi,\vecb\big)$ is the limiting probability density of hitting, from a generic point in $\T^1(\RR^d)$, the first scatterer at time $\rho^{-(d-1)}\xi$ with impact parameter $\vecb$.
These results were proved in \cite[Thm.\ 4.4]{partI},
and they form the first steps of the proof of the convergence
of the billiard flow $\{F_{t,\rho} : t> 0\}$ given in \cite{partII}.

\begin{figure}
\begin{center}
\begin{minipage}{0.49\textwidth}
\unitlength0.1\textwidth
\begin{picture}(10,8)(0,0)
\put(0.5,1){\includegraphics[width=0.9\textwidth]{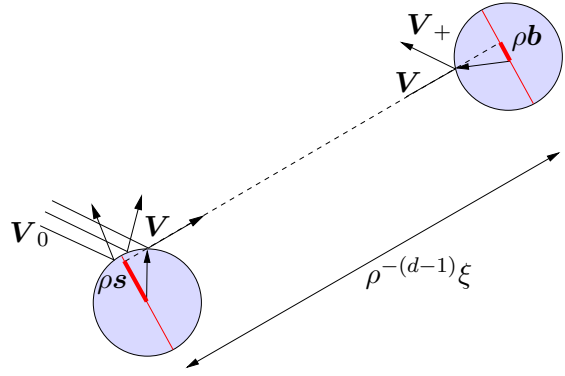}}
\put(0.0,3.2){$\vecV_0$}
\put(2.3,3.3){$\vecV$} 
\put(6.6,5.8){$\vecV$} 
\put(1.5,2.4){$\rho \vecs$}
\put(6.1,2.5){$\rho^{-(d-1)}\xi$}
\put(8.6,6.6){$\rho \vecb$}
\put(6.8,6.8){$\vecV_+$} 
\end{picture}
\end{minipage}
\end{center}
\caption{Two consecutive collisions in the Lorentz gas. \label{figTP}}
\end{figure}

\vspace{5pt}

Our main result on $\Phi_\bn(\xi,\vecw,\vecz)$ is the following.

\begin{thm}\label{CYLINDER2PTSMAINTHM}
Let $d\geq3$. We then have, for all
$\xi>0$, $\vecw,\vecz\in\scrB_1^{d-1}$,
writing $\varphi=\varphi(\vecw,\vecz)$,
\begin{align}\label{CYLINDER2PTSMAINTHMRES}
\Phi_\bn(\xi,\vecw,\vecz)\ll\begin{cases}
\xi^{-2+\frac2d}\min\Bigl\{1,
(\xi\varphi^{d})^{-1+\frac2{d(d-1)}}\Bigr\}
&\text{if }\:\varphi\leq\frac\pi2
\\
\xi^{-2}\min\Bigl\{1,
(\xi(\pi-\varphi)^{d-2})^{-1+\frac2{d-1}}\Bigr\}
&\text{if }\:\varphi\geq\frac\pi2.
\end{cases}
\end{align}
\end{thm}
(If $\vecw=\bn$ or $\vecz=\bn$ then $\varphi$ is undefined,
but in these cases we have 
$\Phi_\bn(\xi,\vecw,\vecz)=0$ whenever $\xi$ is sufficiently large,
cf.\ Proposition \ref{CYLSUPPPROP} below.)

For completeness we recall from \cite{partIII}
that when $d=2$ it is possible to
give an %
explicit formula for $\Phi_\bn(\xi,\vecw,\vecz)$
(now $\vecw=w$ and $\vecz=z$ are real numbers in the 
interval $-1<w,z<1$):
\begin{equation}\label{PF2DIMGENERIC}
\Phi_\vecnull(\xi,w,z)
=\frac{6}{\pi^2}
\begin{cases}
\Upsilon\Bigl(1+\frac{\xi^{-1}-\max(|w|,|z|)-1}{|w+z|}\Bigr)
& \text{if }\: w+z\neq 0
\\
0 & \text{if }\: w+z=0, \: \xi^{-1}<1+|w| \\
1 &  \text{if }\: w+z=0, \: \xi^{-1}\geq 1+|w|,
\end{cases}
\end{equation}
where $\Upsilon(x)=\max(0,\min(1,x))$.

Returning to the case $d\geq3$, 
as a complement to the bound in Theorem \ref{CYLINDER2PTSMAINTHM}
we are able to give a sharp bound on the \textit{support} of the function
$\Phi_\bn$. %
Set
\begin{align*}
s_d(\xi,\varphi):=
\begin{cases}\min(\xi^{-\frac2d},
(\varphi \xi)^{-\frac2{d-1}})&\text{if }\:\varphi\leq\frac\pi2
\\
\max(\xi^{-\frac2{d-2}},(\frac{\xi}{\pi-\varphi})^{-\frac2{d-1}})
&\text{if }\:\varphi>\frac\pi2
\end{cases}
\end{align*}
(and, say, $s_d(\xi,\varphi):=\xi^{-\frac2{d-2}}$ 
when $\varphi$ is undefined).
Then:
\begin{prop}\label{CYLSUPPPROP}
There exist constants $0<c_1<c_2$ which only depend on $d$ such 
that if $\Phi_\bn(\xi,\vecw,\vecz)>0$ then
$\|\vecw\|,\|\vecz\|>1-c_2 s_d(\xi,\varphi)$,
and on the other hand $\Phi_\bn(\xi,\vecw,\vecz)>0$ does hold
for any $\xi>0$, $\vecw,\vecz\in\scrB_1^{d-1}$
satisfying %
$\|\vecw\|,\|\vecz\|>1-c_1 s_d(\xi,\varphi)$.
\end{prop}

We stress that, in contrast to our bound on the support, 
we do not expect the bound 
in Theorem \ref{CYLINDER2PTSMAINTHM} to be sharp in general.
We will prove (cf.\ Propositions \ref{CYLD3OPTPROP} and
\ref{CYLINDER2PTSLOWBDPROP2}) 
that the bound in Theorem \ref{CYLINDER2PTSMAINTHM} is sharp
in a natural sense when $d=3$,
and also for general $d$
if either $\varphi\ll\xi^{-\frac1d}$ or $\pi-\varphi\ll\xi^{-\frac1{d-2}}$.
In particular, for $d\geq4$, while %
we have $\Phi_\bn(\xi,\vecw,\vecz)\ll\xi^{-3+\frac2{d-1}}$ for 
$\ve<\varphi<\pi-\ve$ (any fixed $\ve>0$) and $\xi$ sufficiently large,
$\Phi_\bn(\xi,\vecw,\vecz)$ takes \textit{significantly larger}
values than $\xi^{-3+\frac2{d-1}}$ both when $\varphi\approx0$
and $\varphi\approx\pi$. 

An important consequence of the bound in Theorem \ref{CYLINDER2PTSMAINTHM}
is that it implies a sharp upper bound on the
integral $\int_{\scrB_1^{d-1}}\Phi_\bn(\xi,\vecw,\vecz)\,d\vecz$,
and also implies that the main contribution to this integral comes from
$\vecz$ with $\varphi=\varphi(\vecz,\vecw)$ small.
This integral %
is important since from it we can recover, by further integration, 
both the collision kernel for a generic initial point, $\Phi(\xi,\vecb)$;
cf.\ \eqref{PHIXIWDEF}, 
and the limit density functions %
for the free path length between consecutive collisions
and the free path length from a generic initial point
(\cite[Remark 4.6]{partI}):
\begin{align*}
\overline{\Phi}_\bn(\xi)=\frac1{\vol_{d-1}(\scrB_1^{d-1})}
\int_{\scrB_1^{d-1}}\int_{\scrB_1^{d-1}}\Phi_\bn(\xi,\vecw,\vecz)\,d\vecw\,
d\vecz;
\qquad
\Phi(\xi)=\int_{\scrB_1^{d-1}}\Phi(\xi,\vecw)\,d\vecw.
\end{align*}
\begin{cor}\label{PHI0INTCOR}
Let $d\geq3$. We have 
\begin{align}\label{PHI0INT}
\int_{\scrB_1^{d-1}}\Phi_\bn(\xi,\vecw,\vecz)\,d\vecz\ll\xi^{-3+\frac2d}
\end{align}
for all $\xi\geq1$ and $\vecw\in\scrB_1^{d-1}$.
For any fixed $\ve>0$ the contribution from all $\vecz\in\scrB_1^{d-1}$
with $\varphi(\vecz,\vecw)\geq\ve$ in \eqref{PHI0INT}
is $\ll_{d,\ve}\xi^{-3}$.
Furthermore the integral in \eqref{PHI0INT} vanishes unless
$1-\|\vecw\|\ll\xi^{-\frac2d}$.
On the other hand, there is a constant $c>0$ which only depends on $d$
such that 
\begin{align}\label{PHI0INTLB}
\int_{\scrB_1^{d-1}}\Phi_\bn(\xi,\vecw,\vecz)\,d\vecz\gg\xi^{-3+\frac2d}
\end{align}
holds whenever $\xi\geq1$ and $\|\vecw\|>1-c\xi^{-\frac2d}$.
\end{cor}
Indeed, the two upper bounds follow from %
Theorem \ref{CYLINDER2PTSMAINTHM} combined with
Proposition \ref{CYLSUPPPROP};
the statement about vanishing follows from
Proposition \ref{CYLSUPPPROP} since $s_d(\xi,\varphi)\ll\xi^{-\frac2d}$
uniformly over $\varphi$; and finally
\eqref{PHI0INTLB} follows from the first lower bound in 
Proposition \ref{CYLINDER2PTSLOWBDPROP2} below.

Corollary \ref{PHI0INTCOR} immediately implies:
\begin{cor}
\begin{align*}
\overline{\Phi}_\bn(\xi)\asymp\xi^{-3}\qquad\text{and}\qquad
\Phi(\xi)\asymp\xi^{-2}\qquad\text{as }\:\xi\to\infty.
\end{align*}
\end{cor}
Hence Corollary \ref{PHI0INTCOR} can be viewed as a refinement
of the upper and lower bounds
obtained by Bourgain, Golse and Wennberg \cite[Thm.\ A]{Bourgain98}, \cite[Thm.\ 1]{Golse00},
which correspond to the fact that
$\int_\xi^\infty\Phi(\eta)\,d\eta\asymp\xi^{-1}$ as $\xi\to\infty$.

In the paper \cite{partIV}, which makes strong use of %
the results and methods of the present paper, 
we will give an \textit{asymptotic formula} for $\Phi_\bn(\xi,\vecw,\vecz)$
when $\varphi$ is small and $\xi\to\infty$.
This also implies precise asymptotic formulas for
$\overline{\Phi}_\bn(\xi)$ and $\Phi(\xi)$ as $\xi\to\infty$.

\subsection{Organization of the paper}

The paper is organized as follows.
In Section \ref{GENCONVEXSEC} we discuss bounds on $p(\fC)$ for
arbitrary measurable sets and general convex sets,
proving in particular Proposition \ref{GENPRINCBOUNDPROP}
as well as more precise versions for convex sets,
Proposition \ref{GENCONVBOUNDPROP1}
and Proposition~\ref{GENCONVBOUNDPROP2}.
In Sections \ref{GENBALLTHMPROOFSEC}--\ref{MAINTHM2PFSEC}
we apply these methods to the special cases of $\fC$ a ball and
a cut ball, proving Theorem \ref{GENBALLTHM} and Theorem \ref{MAINTHM2}.
In Section \ref{X1RECOLLECTSEC} we recall the precise definition of
$p_\vecp(\fC)$ and introduce a useful parametrization of the associated
homogeneous space $X_1(\vecp)$.
Finally in Section \ref{PPCBOUNDCCONESEC} we prove Theorem \ref{CONEAPEXTHM}
and in Section \ref{CYLINDER2PTSEC} we prove 
Theorem \ref{CYLINDER2PTSMAINTHM} and Proposition~\ref{CYLSUPPPROP}.

\subsection{Acknowledgements}

This paper complements the joint work \cite{partIV} with Jens Marklof,
which is part of our series of papers
\cite{partI}--\cite{partIV} on 
the periodic Lorentz gas in the Boltzmann-Grad limit.
I am grateful to Jens for many inspiring discussions
and valuable comments on the present paper. 
I am also grateful to Tobias Ekholm and Christer Kiselman for 
inspiring and helpful discussions.

\section{Bounds on $p(\fC)$ for a general convex set $\fC$}
\label{GENCONVEXSEC}

\subsection{Preliminaries}
\label{SIEGELSEC}

Throughout this paper we write $G=\SL_d(\R)$ and $\Gamma=\SL_d(\Z)$.
For any $M\in G$, $\Z^dM$ is a $d$-dimensional lattice of covolume one.
This gives an identification of the space $X_1$ with the homogeneous space
$\Gamma\backslash G$.  %
We write $\mu$ 
for the measure on $X_1$ coming from Haar measure on
$G$, normalized to be a probability measure.
We will sometimes write $G^{(d)}$ and $\mu^{(d)}$ for $G$ and $\mu$,
if we need to emphasize the dimension.

Let $A$ be the subgroup of %
diagonal matrices with positive entries,
\begin{align}\label{AADEF}
\aa(a)=\begin{pmatrix} a_1 & & \\ & \ddots & \\ & & a_d \end{pmatrix}
\in G, \qquad a_j>0,
\end{align}
and let $N$ be the subgroup of upper triangular matrices,
\begin{align}
\nn(u)=\begin{pmatrix} 1 & u_{12} & \cdots & u_{1d}
\\ & \ddots & \ddots & \vdots 
\\ & & \ddots & u_{d-1,d} 
\\ & & & 1
\end{pmatrix} \in G.
\end{align}
Every element $M\in G$ has a unique Iwasawa decomposition
\begin{align}\label{IWASAWADEC}
M=\nn(u) \aa(a) \kk,
\end{align}
with $\kk\in \SO(d)$.
In these coordinates the Haar measure takes the form
(\cite[p.\ 172]{DRS})
\begin{align} \label{SLDZHAAR}
d\mu(M) = \frac{2^{d-1}\pi^{d(d+1)/4}}{\prod_{j=1}^{d}
\Gamma(\frac{j}2) \prod_{j=2}^d \zeta(j)} \rho(a) d\nn(u) d\aa(a) d\kk
\end{align}
where $d\nn$, $d\aa$, $d\kk$, are (left and right) 
Haar measures of $N$, $A$, $\SO(d)$,
normalized by $d\nn(u)=\prod_{1\leq j<k\leq d} du_{jk}$,
$d\aa(a)=\prod_{j=1}^{d-1} (a_j^{-1}\, da_j)$ and
$\int_{\SO(d)} d\kk=1$. For $\rho(a)$ one has
\begin{align}
\rho(a)=\prod_{1\leq i<j\leq d} \frac{a_j}{a_i}
=\prod_{j=1}^d a_j^{2j-d-1}.
\end{align}

We set $\mathcal{F}_N=\bigl\{u \col u_{jk} \in (-\sfrac 12,\sfrac 12], \:
1\leq j<k\leq d\bigr\}$; then
$\{\nn(u) \col u\in \mathcal{F}_N\}$
is a fundamental region for $(\Gamma\cap N)\backslash N$.
We define the following Siegel set:
\begin{align} \label{SIDEF}
\Si_d:=\Bigl\{\nn(u) \aa(a) \kk \col u\in \mathcal{F}_N,\:
0<a_{j+1} \leq \sfrac{2}{\sqrt 3}a_j \: (j=1,\ldots,d-1), 
\: \kk \in \SO(d) \Bigr\}.
\end{align}
It is known that $\Si_d$ contains a fundamental region for
$X_1=\Gamma\backslash G$, and on the other hand
$\Si_d$ is contained in a finite union of fundamental regions for $X_1$
(\cite{Borel}).

Given $M=\nn(u)\aa(a)\kk\in G$, its row vectors are
\begin{align} \label{MBASIS}
\vecb_k=(0,\ldots,0,a_k,a_{k+1}u_{k,k+1},\ldots,a_du_{k,d})\kk,
\qquad k=1,\ldots,d.
\end{align}
If $M\in\Si_d$ then we see that, for all $k$,
\begin{align}\label{MBASISINEQ}
||\vecb_k||
\leq \sum_{j=1}^d a_j \leq \bigl(\sum_{j=0}^{d-1} (2/\sqrt 3)^j\bigr) a_1
\ll a_1.
\end{align}
This bound implies that if $M\in\Si_d$ 
and if the lattice $\Z^dM$ has empty intersection with a large ball,
then $a_1$ must be large:
\begin{lem}\label{A1LARGELEM}
For any $M=\nn(u)\aa(a)\kk\in\Si_d$ 
such that the lattice $\Z^dM$ is disjoint from
some ball of radius $R$ in $\R^d$, we have $a_1\gg R$.
\end{lem}
\begin{proof}
Choose $c_1,\ldots,c_d\in\R$ so that
$\vecp=c_1\vecb_1+\ldots+c_d\vecb_d$ is the center of the given ball. 
Let $n_j$ be the integer nearest to $c_j$.
Then $n_1\vecb_1+\ldots+n_d\vecb_d$ is a lattice point of $\Z^dM$,
and has distance $\leq\frac12\bigl(\|\vecb_1\|+\ldots+\|\vecb_d\|\bigr)
\ll a_1$ to $\vecp$.
This distance must be $>R$; hence $a_1\gg R$.
\end{proof}

Next we will introduce a parametrization of $G$ which  will be useful
for us throughout the paper.
Let us fix a function $f$ (smooth except possibly at one point, say)
$\S^{d-1}_1\to\SO(d)$ such that
$\vece_1 f(\vecv)=\vecv$ for all $\vecv\in \S_1^{d-1}$
(where $\vece_1=(1,0,\ldots,0)$).
Given $M=\nn(u)\aa(a)\kk\in G$,
the matrices $\nn(u)$, $\aa(a)$ and $\kk$ can be split uniquely as
\begin{align}\label{NAKSPLITDEF}
\nn(u)=\matr 1\vecu{\trans\bn}{\nn(\tu)}; \qquad
\aa(a)=\matr{a_1}\bn{\trans\bn}{a_1^{-\frac 1{d-1}}\aa(\ta)}; \qquad
\kk=\matr 1\bn{\trans\bn}{\tkk} f(\vecv)
\end{align}
where $\vecu\in\R^{d-1}$, $\nn(\tu)\in N^{(d-1)}$,
$a_1>0$, $\aa(\ta)\in A^{(d-1)}$ and
$\tkk\in\SO(d-1)$, $\vecv\in\S^{d-1}_1$.
We set
\begin{align}\label{TMDEF}
\tM=\nn(\tu)\aa(\ta)\tkk\in G^{(d-1)}.
\end{align}
In this way we get a bijection between $G$ and
$\R_{>0}\times\S_1^{d-1}\times\R^{d-1}\times G^{(d-1)}$;
we write $M=[a_1,\vecv,\vecu,\tM]$ for the element in $G$ corresponding to
the 4-tuple
$\langle a_1,\vecv,\vecu,\tM\rangle\in\R_{>0}\times\S_1^{d-1}\times\R^{d-1}\times G^{(d-1)}$.
In particular note that
\begin{align}\notag
\Si_d =\Bigl\{[a_1,\vecv,\vecu,\tM]\in G\col
\tM\in\Si_{d-1},\:
\ta_1\leq \sfrac 2{\sqrt 3}a_1^{\frac d{d-1}}
,\:\vecu\in(-\sfrac 12,\sfrac 12]^{d-1}\Bigr\}\qquad
\\\label{SIDSUBSSIDM1}
\subset\Bigl\{[a_1,\vecv,\vecu,\tM]\in G\col
\tM\in\Si_{d-1},\:\vecu\in(-\sfrac 12,\sfrac 12]^{d-1}\Bigr\}.
\end{align}

One checks by a straightforward
computation using %
\eqref{SLDZHAAR}
that the Haar measure $\mu$ takes the following form
in the parametrization $M=[a_1,\vecv,\vecu,\tM]$:
\begin{align}\label{SLDRSPLITHAAR}
d\mu(M)=\zeta(d)^{-1} \,d\mu^{(d-1)}(\tM)\,
d\vecu\,
d\!\vol_{\S_1^{d-1}}(\vecv)\,
\frac{da_1}{a_1^{d+1}}.
\end{align}
Note that all of the above claims are valid also for $d=2$,
with the natural interpretation that $\Si_1=\SL(1,\R)=\{1\}$ 
with $\mu^{(1)}(\{1\})=1$.
We will also need to know the explicit expression of 
the lattice $\Z^dM$ in terms of $a_1,\vecv,\vecu,\tM$:
One computes that, for any $\vecm\in\Z^{d-1}$ and $n\in\Z$,
\begin{align}\label{LATTICEINPARAM}
(n,\vecm)M
=na_1\vecv+a_1^{-\frac1{d-1}}\bigl(0,n\vecu\aa(\ta)\tkk+\vecm\tM\bigr)
f(\vecv).
\end{align}
In particular we always have
\begin{align}\label{LATTICECONTAINEMENT}
\Z^dM\subset\bigsqcup_{n\in\Z} \bigl(na_1\vecv+\vecv^\perp\bigr).
\end{align}

\subsection{\texorpdfstring{General bounds on $p(\fC)$}{General bounds on p(C)}}

We start by recalling two well-known %
inequalities.
(Cf., e.g., \cite[p.\ 167]{Schmidt59}.)
\begin{lem}\label{TRIVPLOWBOUNDLEM}
For an arbitrary Borel measurable subset $\fC\subset\R^d$ we have
\begin{align*}
p(\fC)\geq1-|\fC|.
\end{align*}
\end{lem}
\begin{proof}
$p(\fC)\geq\int_{X_1}(1-\#(L\cap \fC\setminus\{\bn\}))\,d\mu(L)=1-|\fC|$,
by Siegel's formula (\cite{siegel45}, \cite{siegel}).
\end{proof}

\begin{lem}\label{PCONTLEM}
For any two Borel measurable subsets $\fC,\fC'\subset\R^d$ 
we have
\begin{align*}
\bigl|p(\fC)-p(\fC')\bigr|\leq\max\bigl(\bigl|\fC\setminus\fC'\bigr|,
\bigl|\fC'\setminus\fC\bigr|\bigr).
\end{align*}
\end{lem}
\begin{proof}
Any $L\in X_1$ which is disjoint from $\fC\setminus\{\bn\}$ 
must either be disjoint from
$\fC'\setminus\{\bn\}$ or have a point in 
$(\fC'\setminus\fC)\setminus\{\bn\}$; 
hence
$p(\fC)\leq p(\fC')+(1-p(\fC'\setminus\fC))$,
and thus $p(\fC)\leq p(\fC')+|\fC'\setminus\fC|$ by 
Lemma \ref{TRIVPLOWBOUNDLEM}.
Similarly $p(\fC')\leq p(\fC)+|\fC\setminus\fC'|$.
\end{proof}

Next we give the simple proof of 
\eqref{ATHREYAMARGULISRES} in the special case of $\fC$ convex.

\begin{lem}\label{TRIVPCBOUNDLEM}
If $\fC\subset\R^d$ contains some $d$-dimensional ellipsoid of
volume $V$ then
\begin{align}\label{TRIVPCBOUNDLEMRES}
p(\fC)\ll V^{-1}.
\end{align}
\end{lem}
\begin{proof}
(Cf.\ \cite[Lemma 8.15]{partI}.)
Using \eqref{PDINV} we may from start assume that
$\fC$ contains a \textit{ball} of volume $V$.
This ball in turn contains a ball $B$ of volume $\geq2^{-d}V$
which does not contain the origin.
Now $\Z^dM\cap\fC\setminus\{\bn\}=\emptyset$ implies
$\Z^dM\cap B=\emptyset$;
hence by Lemma \ref{A1LARGELEM} there is some
$A\gg V^{\frac1d}$ such that
$a_1>A$ holds for all $M=\nn(u)\aa(a)\kk\in\Si_d$ 
satisfying $\Z^dM\cap\fC\setminus\{\bn\}=\emptyset$.
We have
\begin{align*}
p(\fC)=
\mu\bigl(\bigl\{M\in X_1\col \Z^dM\cap\fC\setminus\{\bn\}
=\emptyset\bigr\}\bigr)
\leq\mu\bigl(\bigl\{M\in\Si_d\col \Z^dM\cap\fC\setminus\{\bn\}
=\emptyset\bigr\}\bigr),
\end{align*}
since $\Si_d$ contains a fundamental region for $X_1$.
Using now \eqref{SIDSUBSSIDM1} and \eqref{SLDRSPLITHAAR} we
conclude
\begin{align*}
p(\fC)
\leq
\zeta(d)^{-1}\int_A^\infty\int_{\S_1^{d-1}}\int_{(-\sfrac 12,\sfrac 12]^{d-1}}
\int_{\Si_{d-1}}d\mu^{(d-1)}(\tM)\,d\vecu\,d\vecv
\,\frac{da_1}{a_1^{d+1}}
\ll A^{-d}\ll V^{-1}.
\end{align*}
(From now on we write simply $d\vecv$ for the 
$(d-1)$-dimensional volume measure on $\S_1^{d-1}$.)
Here we used the fact that $\mu^{(d-1)}(\Si_{d-1})$ is finite,
since $\Si_{d-1}$ can be covered by a finite number of fundamental regions
for $\Gamma^{(d-1)}\backslash G^{(d-1)}$.
\end{proof}

\begin{lem}\label{TRIVPCBOUNDLEM2}
If $\fC\subset\R^d$ is convex then
\begin{align*}
p(\fC)\ll|\fC|^{-1}.
\end{align*}
(If $|\fC|=\infty$ this should be interpreted as $p(\fC)=0$.)
\end{lem}
\begin{proof}
If $|\fC|<\infty$ then $\fC$ contains %
an ellipsoid of 
volume $\gg|\fC|$ (cf.\ \cite{John48});
if $|\fC|=\infty$ (viz.\ $\fC$ has non-empty interior and is unbounded)
then for every $V>0$ there is an ellipsoid $E\subset\fC$ 
of volume $>V$. 
Hence the lemma follows from Lemma \ref{TRIVPCBOUNDLEM}.
\end{proof}

To conclude this section we give the proof of 
Proposition \ref{GENPRINCBOUNDPROP}.
The idea is to use the parametrization $M=[a_1,\vecv,\vecu,\tM]\in\Si_d$
and note that if $\Z^dM\cap\fC\subset\{\bn\}$ then $a_1\gg r$ by 
Lemma \ref{A1LARGELEM}, and also by using
$(0,\vecm)M\notin\fC$ for all $\vecm\in\Z^{d-1}\setminus\{\bn\}$
we obtain a precise constraint on $\tM$ (cf.\ \eqref{LATTICEINPARAM}).

\begin{proof}[Proof of Proposition \ref{GENPRINCBOUNDPROP}]
Let $r$ be the radius of some ball contained in $\fC$.
Then $\fC\setminus\{\bn\}$ contains a ball of radius $\frac12r$,
and hence by Lemma~\ref{A1LARGELEM},
if $M=[a_1,\vecv,\vecu,\tM]\in\Si_d$ satisfies
$\Z^dM\cap\fC\subset\{\bn\}$ 
then $a_1>k_2r$, where $k_2$ is a positive constant
which only depends on $d$.
Hence, using \eqref{SIDSUBSSIDM1} and \eqref{SLDRSPLITHAAR},
\begin{align*}
&p^{(d)}(\fC)\leq\mu\bigl(\bigl\{M\in\Si_d\col \Z^dM\cap\fC\subset\{\bn\}\bigr\}\bigr)
\\
&\leq
\zeta(d)^{-1}\int_{k_2r}^\infty\int_{\S_1^{d-1}}
\int_{(-\sfrac 12,\sfrac 12]^{d-1}}
\mu^{(d-1)}\bigl(\bigl\{\tM\in\Si_{d-1}\col
\Z^d[a_1,\vecv,\vecu,\tM]
\cap\fC\subset\{\bn\}\bigr\}\bigr)
\,d\vecu\,d\vecv\,\frac{da_1}{a_1^{d+1}}.
\end{align*}
Using \eqref{LATTICEINPARAM} with $n=0$ we get
\begin{align}\notag
p^{(d)}(\fC)&\leq\zeta(d)^{-1}\int_{k_2r}^\infty\int_{\S_1^{d-1}}
\mu^{(d-1)}\bigl(\bigl\{\tM\in\Si_{d-1}\col
a_1^{-\frac1{d-1}}\iota(\Z^{d-1}\tM)f(\vecv)\cap\fC
\subset\{\bn\}\bigr\}\bigr)
\,d\vecv\,\frac{da_1}{a_1^{d+1}}
\\\notag
&=\zeta(d)^{-1}\int_{k_2r}^\infty\int_{\S_1^{d-1}}
\mu^{(d-1)}\bigl(\bigl\{\tM\in\Si_{d-1}\col
\Z^{d-1}\tM\cap a_1^{\frac1{d-1}}\fC_\vecv
\subset\{\bn\}\bigr\}\bigr)
\,d\vecv\,\frac{da_1}{a_1^{d+1}},
\end{align}
where $\iota$ denotes the embedding $\iota:\R^{d-1}\ni (x_1,\ldots,x_{d-1})
\mapsto (0,x_1,\ldots,x_{d-1})\in\R^d$, and
\begin{align}\label{CVDEFNEW}
\fC_\vecv:=\iota^{-1}(\fC f(\vecv)^{-1})\subset\R^{d-1}.
\end{align}
Now since $\Si_{d-1}$ is contained in a finite union of fundamental
regions for $X_1^{(d-1)}$, we obtain
\begin{align*}
p^{(d)}(\fC)\ll\int_{\S_1^{d-1}}\int_{k_2r}^\infty
p^{(d-1)}\Bigl(a_1^{\frac1{d-1}}\fC_\vecv\Bigr)\,\frac{da_1}{a_1^{d+1}}\,d\vecv.
\end{align*}
We have proved this for any $r$ which is the radius of some ball contained
in $\fC$; hence it also holds for the supremum of these radii.
Now \eqref{GENPRINCBOUNDPROPRES1} follows,
since $\fC\cap\vecv^\perp$ maps to $\fC_\vecv$ by the
volume preserving linear space isomorphism
$\vecv^\perp\ni\vecx\mapsto\iota^{-1}(\vecx f(\vecv)^{-1})\in\R^{d-1}$.
\end{proof}

\subsection{\texorpdfstring{Bounding $p(\fC)$ from above for $\fC$ convex}{Bounding p(C) from above for C convex}}

For \textit{convex} sets $\fC$ we have the following monotonicity property
which allows us to simplify the upper bound in 
Proposition \ref{GENPRINCBOUNDPROP}.

\begin{lem}\label{CONVPDMONOTONELEM}
For any convex set $\fC\subset\R^d$ ($d\geq2$) and any $\alpha>1$
we have $p(\alpha\fC)\leq p(\fC)$.
\end{lem}
The proof is by finding an element $T\in G$ such that 
$\fC T\subset\alpha\overline\fC$.
For the construction we need the following auxiliary lemma.
\begin{lem}\label{CONVPARALLELHYPPLANESLEM}
Let $\fC\subset\R^d$ be a compact convex set with $\bn\notin\fC$.
Then there exist two non-zero vectors $\vecv,\vecw\in\R^d$ 
and two points $\vecq_1,\vecq_2\in\fC\cap\R\vecv$ such that
$\vecq_1\cdot\vecw\leq\vecp\cdot\vecw\leq\vecq_2\cdot\vecw$ for all
$\vecp\in\fC$.
\end{lem}
\begin{proof}
We first assume %
that $\fC$ has only regular boundary points and support planes
(viz.\ to each boundary point there corresponds exactly one
support plane, and each support plane has only one point in common
with $\fC$; cf.\ \cite[Sec.\ 3.9]{bonnesen87}).
Let $K$ be the set of $\veca\in\R^d$ such that
$\fC\cap\R_{>0}\veca\neq\emptyset$. %
For each $\veca\in K$ set $\ell_1(\veca)=\inf\{t>0\col t\veca\in\fC\}$
and $\ell_2(\veca)=\sup\{t>0\col t\veca\in\fC\}$.
Let $\vecv\in K$ be a point where 
$\sup_{\veca\in K}\ell_2(\veca)/\ell_1(\veca)$ is attained;
such a point clearly exists, and in fact $\vecv\in K^\circ$,
and $\ell_2(\vecv)/\ell_1(\vecv)>1$.
Set $\vecq_j=\ell_j(\vecv)\vecv$ ($j=1,2$) and let
$\vecw_1,\vecw_2\neq\bn$ be unit vectors normal to the unique support planes
of $\fC$ at $\vecq_1,\vecq_2$, chosen so that
$\vecq_1\cdot\vecw_1\leq\vecp\cdot\vecw_1$ and
$\vecp\cdot\vecw_2\leq\vecq_2\cdot\vecw_2$ for all $\vecp\in\fC$.
Using the fact that these two support planes are regular and 
$\vecq_1\neq\vecq_2$ it follows that $\vecv\cdot\vecw_j\neq0$ for $j=1,2$.
Hence for each $\vecb\in\R^d$ there exist unique
$x_1,x_2\in\R$ such that $(x_j\vecv+\vecb)\cdot\vecw_j=0$,
viz.\ (if $\vecb\notin\R\vecv$)
the support plane for $\fC$ at $\vecq_j$ contains
the line $\vecq_j+\R(x_j\vecv+\vecb)$.
Since $\vecq_1,\vecq_2$ are regular boundary points this implies that
$\ell_j(\vecv+\ve\vecb)=\ell_j(\vecv)(1+x_j\ve+o(\ve))$ as $\ve\to0$ ($j=1,2$).
Hence by our choice of $\vecv$, $x_1=x_2$ must hold.
Since this is true for every $\vecb\in\R^d$
it follows that $\vecw_1=\vecw_2$, and we are done.

In the case of a general compact convex set $\fC$ with $\bn\notin\fC$
we take a sequence $\fC^{(1)},\fC^{(2)},\ldots$ of compact convex sets
converging to $\fC$ such that each $\fC^{(n)}$ has only regular boundary
points and support planes, and $\bn\notin\fC^{(n)}$
(such a sequence exists by \cite[Sec.\ 6.27]{bonnesen87}).
Applying the above to each $\fC^{(n)}$ we obtain vectors
$\vecv^{(n)},\vecw^{(n)},\vecq_1^{(n)},\vecq_2^{(n)}$
satisfying the required conditions for $\fC^{(n)}$.
We may assume that all $\vecv^{(n)},\vecw^{(n)}$ have length $1$.
By passing to a subsequence we may assume that the four limits
$\vecv=\lim_{n\to\infty}\vecv^{(n)}$,
$\vecw=\lim_{n\to\infty}\vecw^{(n)}$ and
$\vecq_j=\lim_{n\to\infty}\vecq_j^{(n)}$ ($j=1,2$) exist.
These $\vecv,\vecw,\vecq_1,\vecq_2$ are easily seen to satisfy the
required conditions for $\fC$.
\end{proof}

\begin{proof}[Proof of Lemma \ref{CONVPDMONOTONELEM}]
If $\fC$ has no interior points then $\fC$ is contained in some affine subspace
of $\R^d$, and thus $p(\alpha\fC)=p(\fC)=1$ 
(cf.\ Lemma \ref{TRIVPLOWBOUNDLEM}).
Hence from now on we may assume that $\fC$ has interior points.
If $\fC$ is unbounded then $|\fC|=\infty$ and $p(\alpha\fC)=p(\fC)=0$
by Lemma \ref{TRIVPCBOUNDLEM2}.
Hence from now on we may assume that $\fC$ is bounded.
If $\bn\in\overline\fC$ then $\fC\subset\alpha\overline{\fC}$ for any 
$\alpha>1$,
and thus $p(\alpha\overline{\fC})\leq p(\fC)$;
but also $p(\alpha\fC)=p(\alpha\overline{\fC})$ by
Lemma \ref{PCONTLEM}; hence $p(\alpha\fC)\leq p(\fC)$.

Hence it only remains to deal with the case 
$\bn\notin\overline{\fC}$.
Let $\vecv,\vecw,\vecq_1,\vecq_2$ be as in 
Lemma~\ref{CONVPARALLELHYPPLANESLEM} applied to $\overline\fC$.
Note that $\vecv\cdot\vecw\neq0$, since $\fC$ has interior points.
Take any $U\in\GL(d,\R)$ such that
$\vece_1U\in\R\vecv$ and 
$\vece_jU\in\vecw^\perp$ for $j=2,\ldots,d$,
where $\vece_j=(0,\ldots,1,\ldots,0)$ is the $j$th standard basis vector of
$\R^d$.
Set 
$T=U^{-1}\diag[\alpha,\alpha^{-\frac1{d-1}},\ldots,\alpha^{-\frac1{d-1}}]U
\in G$.
We now claim $\fC T\subset\alpha\overline\fC$.
Indeed, consider an arbitrary point $\vecp\in\fC$.
Then there are unique $t\in\R$ and $\vecu\in\vecw^\perp$ such that
$\vecp=t\vecv+\vecu$ and thus
$\vecp T=\alpha t\vecv+\alpha^{-\frac1{d-1}}\vecu$.
But $\vecq_1\cdot\vecw\leq\vecp\cdot\vecw\leq\vecq_2\cdot\vecw$
implies $\vecq_1\cdot\vecw\leq t\vecv\cdot\vecw\leq\vecq_2\cdot\vecw$;
hence $t\vecv$ lies on the line segment between $\vecq_1$ and $\vecq_2$,
and so $t\vecv\in\overline\fC$.
Hence by convexity,
$s\vecp+(1-s)t\vecv=t\vecv+s\vecu$ lies in $\overline\fC$ for all 
$0\leq s\leq1$.
In particular $\alpha^{-1}\vecp T=t\vecv+\alpha^{-\frac d{d-1}}\vecu
\in\overline{\fC}$, thus proving $\fC T\subset\alpha\overline\fC$.
It follows that
$p(\alpha\fC)=p(\alpha\overline\fC)
\leq p(\fC T)=p(\fC)$, and Lemma \ref{CONVPDMONOTONELEM} is proved.
\end{proof}

\begin{remark}\label{CONVPDMONOTONELEMREM}
The convexity assumption in Lemma \ref{CONVPDMONOTONELEM}
cannot be skipped altogether.
For example, for any given $\alpha>1$ with $\alpha^d\notin\Z$,
if $\ve>0$ is sufficiently small then the set
\begin{align*}
\fC:=\scrB_{1/\ve}^d\setminus
\bigcup_{\vecm\in\Z^d\setminus\{\bn\}}\bigl(\alpha^{-1}\vecm+\scrB_\ve^d\bigr)
\end{align*}
satisfies $p(\alpha\fC)>0$ and $p(\fC)=0$.
Indeed, $p(\alpha\fC)>0$ holds since every $L\in X_1$ sufficiently near
$\Z^d$ is disjoint from $\alpha\fC\setminus\{\bn\}$.
On the other hand, assume $L\in X_1$ is disjoint from
$\fC\setminus\{\bn\}$.
Write $L=\Z^dM$ with $M=\nn(u) \aa(a) \kk\in\Si_d$ and take $\vecb_k$ as in
\eqref{MBASIS}; thus $L=\Z\vecb_1+\cdots+\Z\vecb_d$.
Now $L\cap\fC\setminus\{\bn\}=\emptyset$ forces
$\|\vecb_k\|>(2\alpha)^{-1}$ provided $\ve$ is sufficiently small;
hence also $a_k\gg\alpha^{-1}$, by 
the same type of computation as in \eqref{MBASISINEQ}.
It follows that $\|\vecb_k\|\ll a_1=(a_2\cdots a_d)^{-1}\ll\alpha^{d-1}$ 
and hence if $\ve$ is sufficiently small then
$\vecb_k\in\scrB_{1/\ve}^d$ and thus
$\vecb_k\in\alpha^{-1}(\Z^d\setminus\{\bn\})+\scrB_\ve^d$ for all $k$.
By the determinant formula for the covolume of $L=\Z\vecb_1+\cdots+\Z\vecb_d$,
this forces $\alpha^d$ to be ``$\ve$-near'' an integer, and we get a 
contradiction if $\ve$ is taken sufficiently small.
Hence $p(\fC)=0$.

On the other hand for any $\alpha>1$ with $\alpha^d\in\Z$ we have
$p(\alpha\fC)\leq p(\fC)$ for \textit{every measurable set $\fC\subset\R^d$,}
as is easily seen using the modular correspondence $T(\alpha^d)$ on
$X_1$ (cf.,\ e.g., \cite[Ch.\ 3]{Shimura}).
Indeed, if $\scrF\subset G$ is a fundamental region for 
$X_1=\Gamma\backslash G$ and
$\beta_1,\ldots,\beta_r\in M_d(\Z)$ are representatives such that
$T(\alpha^d)=\sqcup_{j=1}^r\Gamma \beta_j$ then
\begin{align*}
p(\alpha\fC)
=\int_\scrF I\bigl(\bigl\{
\Z^dM\cap\alpha\fC\subset\{\bn\}\bigr\}\bigr)\,d\mu(M)
\leq r^{-1}\sum_{j=1}^r\int_\scrF I\bigl(\bigl\{
\alpha^{-1}\Z^d\beta_jM\cap\fC\subset\{\bn\}\bigr\}\bigr)\,d\mu(M)
\\
=r^{-1}\sum_{j=1}^r\int_{\alpha^{-1}\beta_j\scrF} I\bigl(\bigl\{
\Z^dM\cap\fC\subset\{\bn\}\bigr\}\bigr)\,d\mu(M)
=p(\fC).
\end{align*}
(We used $\Z^d \beta_j\subset\Z^d$
and the fact that $\sqcup_{j=1}^r\alpha^{-1}\beta_j\scrF$ 
is an $r$-fold fundamental region for $\Gamma\backslash G$.)
\end{remark}

We next use Lemma \ref{CONVPDMONOTONELEM} to simplify the
bound in Proposition \ref{GENPRINCBOUNDPROP}.
Recall that if $\fC\subset\R^d$ is a bounded convex set then 
the supremum of the 
radii of all $d$-dimensional open balls contained in $\fC$ is attained,
although not necessarily for a unique ball;
we call this supremum the \textit{inradius} of $\fC$.  %
\begin{prop}\label{GENCONVBOUNDPROP1}
Given $d\geq2$ there exist constants $k_1,k_2>0$ such that
for every bounded convex set $\fC\subset\R^d$ we have
\begin{align}\label{GENCONVBOUNDPROP1RES}
p^{(d)}(\fC)\leq
\min\biggl\{1,k_1r^{-d}
\int_{\S_1^{d-1}}
p'_{d-1}\Bigl(k_2r^{\frac1{d-1}}\fC\cap\vecv^\perp\Bigr)\,d\vecv\biggr\},
\end{align}
where $r$ is the 
inradius of $\fC$
and where $p'_{d-1}(\fZ)$ for $\fZ\subset\R^{d-1}$ is defined by
\begin{align}\label{GENCONVBOUNDPROP1DEF}
p'_{d-1}(\fZ)=\begin{cases}p^{(d-1)}(\fZ)&\text{if }\: d\geq3;
\\
I(|\fZ|\leq1)&\text{if }\: d=2.
\end{cases}
\end{align}
\end{prop}
\begin{proof}
Given $d$ we let $k_1,k_2$ be as in Proposition \ref{GENPRINCBOUNDPROP}.
Now let $\fC\subset\R^d$ be a bounded convex set,
and let $r$ be its inradius.
If $d\geq3$ then by Lemma \ref{CONVPDMONOTONELEM},
for all $\vecv\in\S_1^{d-1}$ and all $a_1\geq k_2r$ we have
$p^{(d-1)}(a_1^{\frac1{d-1}}\fC\cap\vecv^\perp)\leq
p^{(d-1)}((k_2r)^{\frac1{d-1}}\fC\cap\vecv^\perp)$
(note that this is true also when
$a_1^{\frac1{d-1}}\fC\cap\vecv^\perp=\emptyset$).
Hence \eqref{GENCONVBOUNDPROP1RES}, with new $k_1,k_2$,
follows directly from Proposition \ref{GENPRINCBOUNDPROP}.
On the other hand if $d=2$ then for any $\vecv\in\S_1^1$ with
$|k_2r\fC\cap\vecv^\perp|>2$ we have
$p^{(1)}(a_1\fC\cap\vecv^\perp)=0$ for all $a_1>k_2r$,
since $a_1\fC\cap\vecv^\perp$ is a line segment of length $>2$ and hence,
after identifying $\vecv^\perp$ with $\R$
this line segment must have non-empty intersection with $\Z\setminus\{\bn\}$.
Hence \eqref{GENCONVBOUNDPROP1RES} (with new $k_1,k_2$)
again follows from Proposition \ref{GENPRINCBOUNDPROP}.
\end{proof}

\subsection{\texorpdfstring{Bounding $p(\fC)$ from below for $\fC$ convex}{Bounding p(C) from below for C convex}}

We next prove that for convex $\fC$,
$p(\fC)$ is bounded from below
by a similar expression as in 
the upper bound in Proposition \ref{GENCONVBOUNDPROP1}.
Recall that if $\fC\subset\R^d$ is a bounded convex set then 
the infimum of the 
radii of all $d$-dimensional closed balls containing $\fC$ is attained
for a unique ball;
we call this infimum the \textit{circumradius} of $\fC$. %
\begin{prop}\label{GENCONVBOUNDPROP2}
Given $d\geq2$ there exist constants $k_3,k_4>0$ such that
for every bounded convex set $\fC\subset\R^d$,
\begin{align}\label{GENCONVBOUNDPROP2RES}
p^{(d)}(\fC)\geq
\min\biggl\{\frac12,k_3r^{-d}
\int_{\S_1^{d-1}}
p'_{d-1}\Bigl(k_4r^{\frac1{d-1}}\fC\cap\vecv^\perp\Bigr)\,d\vecv\biggr\},
\end{align}
where $r$ is the circumradius of $\fC$,
and where $p'_{d-1}(\fZ)$ is as in \eqref{GENCONVBOUNDPROP1DEF}.
\end{prop}

The starting-point of the proof is to use
$p^{(d)}(\fC)\gg\mu(\{M\in\Si_d\col\Z^dM\cap\fC\subset\{\bn\}\})$,
which holds since $\Si_d$ is contained in a finite union
of fundamental regions for $X_1$, and then integrate over $M\in\Si_d$
using the parametrization $M=[a_1,\vecv,\vecu,\tM]$;
cf.\ \eqref{SIDSUBSSIDM1}.
The restriction $\ta_1\leq \sfrac 2{\sqrt 3}a_1^{d/(d-1)}$ in
\eqref{SIDSUBSSIDM1} leads to a technical problem when considering bounds
from below; to handle this we first prove the following
auxiliary lemma, which we will apply with $d-1$ in place of $d$.
The point of the lemma is to show that if 
a convex set $\fC$ is contained in a ball of radius $r$ centered at the origin,
then among all $M\in\Si_d$ satisfying $\Z^dM\cap\fC\subset\{\bn\}$,
a positive proportion (w.r.t.\ $\mu$) actually have $a_1<2r$,
where $a_1=a_1(M)$ as always refers to the $a_1$ occurring in 
the Iwasawa decomposition $M=\nn(u)\aa(a)\kk$,
cf.\ \eqref{IWASAWADEC} and \eqref{AADEF}.
The lemma is proved by using, once more, the parametrization
$M=[a_1,\vecv,\vecu,\tM]$, and noticing that whenever $a_1>r$ the
intersection $\Z^dM\cap\fC$ is in fact contained in
$(0,\Z^{d-1})M\cap\fC$, i.e.\ is independent of $\vecu$
(cf.\ \eqref{LATTICEINPARAM}).
\begin{lem}\label{GENCONVBOUNDPROP2LEM}
For every $r\geq1$ and every convex set
$\fC\subset\overline{\scrB_r^d}$ ($d\geq2$),
we have
\begin{align*}
p^{(d)}(\fC)\ll
\mu\bigl(\bigl\{M\in\Si_d\col a_1<2r,\:\Z^dM\cap\fC\subset\{\bn\}
\bigr\}\bigr).
\end{align*}
\end{lem}
\begin{proof}
Since $p(\fC)\leq\mu(\{M\in\Si_d\col\Z^dM\cap\fC\subset\{\bn\}\})$,
it suffices to prove
\begin{align}\notag
\mu\bigl(\bigl\{M\in\Si_d\col a_1>2r,\:\Z^dM\cap\fC\subset\{\bn\}
\bigr\}\bigr)
\hspace{100pt}
\\\label{GENCONVBOUNDPROP2LEMPF1}
\ll\mu\bigl(\bigl\{M\in\Si_d\col r<a_1<2r,\:
\Z^dM\cap\fC\subset\{\bn\}
\bigr\}\bigr).
\end{align}
Writing $M=[a_1,\vecv,\vecu,\tM]$, and noticing that 
$a_1>r$ implies $(na_1\vecv+\vecv^\perp)\cap\overline{\scrB_r^d}=\emptyset$
for all $\vecv\in\S_1^{d-1}$ and $n\in\Z\setminus\{0\}$,
we see from 
\eqref{SIDSUBSSIDM1}--\eqref{LATTICECONTAINEMENT} 
that \eqref{GENCONVBOUNDPROP2LEMPF1} will follow if we can prove
that, for each $\vecv\in\S_1^{d-1}$,
\begin{align}\notag
\int_{2r}^\infty
\mu^{(d-1)}\bigl(\bigl\{\tM\in\Si_{d-1}\col
\ta_1\leq\sfrac2{\sqrt3}a_1^{\frac d{d-1}},\:
\Z^{d-1}\tM\cap a_1^{\frac1{d-1}}\fC_\vecv\subset\{\bn\}\bigr\}
\bigr)\,\frac{da_1}{a_1^{d+1}}%
\hspace{30pt}
\\\label{GENCONVBOUNDPROP2LEMPF2}
\ll %
\int_r^{2r}
\mu^{(d-1)}\bigl(\bigl\{\tM\in\Si_{d-1}\col
\ta_1\leq\sfrac2{\sqrt3}a_1^{\frac d{d-1}},\:
\Z^{d-1}\tM\cap a_1^{\frac1{d-1}}\fC_\vecv\subset\{\bn\}\bigr\}
\bigr)\,\frac{da_1}{a_1^{d+1}} %
\end{align}
where $\fC_\vecv$ is as in \eqref{CVDEFNEW}
(in particular $\fC_\vecv$ is a bounded convex set).
If $\fC_\vecv=\emptyset$ then \eqref{GENCONVBOUNDPROP2LEMPF2} holds,
since $\mu^{(d-1)}(\{\tM\in\Si_{d-1}\col\ta_1\leq2/\sqrt3\})
\asymp \mu^{(d-1)}(\Si_{d-1})\asymp1$;
hence from now on we may assume $\fC_\vecv\neq\emptyset$.

First assume $d=2$.
Using $\Si_1=G^{(1)}=\{1\}$ and $r\geq1$, 
and writing $r\overline{\fC_\vecv}=[\alpha,\beta]$
we see that our task is to prove
\begin{align}\label{GENCONVBOUNDPROP2LEMPF4}
\int_{2}^\infty I\bigl(\Z\cap(\alpha x,\beta x)\subset\{0\}\bigr)\,
\frac{dx}{x^3}
\ll
\int_1^{2}I\bigl(\Z\cap [\alpha x,\beta x]\subset\{0\}\bigr)\,
\frac{dx}{x^3},
\end{align}
for any $\alpha\leq\beta$.
This bound is verified by a case by case analysis.
Indeed, if $\alpha\leq0\leq\beta$ then 
$\Z\cap [\alpha x,\beta x]\subset\{0\}$ is equivalent with
$\max(|\alpha|,|\beta|)|x|<1$ and thus
\eqref{GENCONVBOUNDPROP2LEMPF4} follows easily.
Also if $\beta-\alpha\geq\frac12$ then the bound holds since
the left hand side vanishes.
Hence by symmetry it remains to consider the case 
$0<\alpha\leq\beta<\alpha+\frac12$.
Now for every integer $n\in[\alpha,2\alpha)$ the interval
$I_n=(\frac n\alpha,\frac{n+1}{\alpha+\frac12})$ is contained in $[1,2]$,
and each $x\in I_n$ satisfies
$\Z\cap[\alpha x,\beta x]\subset\Z\cap(n,n+1)=\emptyset$.
Hence the right hand side of \eqref{GENCONVBOUNDPROP2LEMPF4} is
$\geq8^{-1}\sum_{n\in[\alpha,2\alpha)}|I_n|
=\sum_{n\in[\alpha,2\alpha)}\frac{2\alpha-n}{8\alpha(2\alpha+1)}$,
and if $\alpha\geq\frac{11}{10}$ then this is $\gg1$ so that 
the bound \eqref{GENCONVBOUNDPROP2LEMPF4} holds.
If $\frac35\leq\alpha<\frac{11}{10}$ then we get the same conclusion
since the right hand side of \eqref{GENCONVBOUNDPROP2LEMPF4} is
$\geq8^{-1}|I_1\cap(1,2)|\gg1$.
Also if $\beta\leq\frac9{10}$ the 
right hand side of \eqref{GENCONVBOUNDPROP2LEMPF4} is seen to be $\gg1$
so that \eqref{GENCONVBOUNDPROP2LEMPF4} holds.
Hence from now on we may assume $0<\alpha<\frac35$ and 
$\frac9{10}<\beta<\alpha+\frac12$.
If $\alpha x>2$ then
$x(\beta-\alpha)>\frac2\alpha(\frac9{10}-\frac35)>1$;
hence only $x$ with $\beta x<2$ can contribute to the integral in the
left hand side of \eqref{GENCONVBOUNDPROP2LEMPF4}.
It follows that this integral is $\ll\max(0,2(\beta^{-1}-1))$.
On the other hand if $\beta<1$ then
$\Z\cap[\alpha x,\beta x]\subset\{0\}$ holds for all
$x\in(1,\beta^{-1})$, so that the right hand side of
\eqref{GENCONVBOUNDPROP2LEMPF4} is $\gg\beta^{-1}-1$.
This completes the proof of the bound \eqref{GENCONVBOUNDPROP2LEMPF4},
and hence also of \eqref{GENCONVBOUNDPROP2LEMPF2} in the case $d=2$.

Next assume $d\geq3$.
We first bound the left hand side of \eqref{GENCONVBOUNDPROP2LEMPF2}
using
\begin{align*}
\mu^{(d-1)}(\{\tM\in\Si_{d-1}\col
\ta_1\leq\sfrac2{\sqrt3}a_1^{\frac d{d-1}},\:
\Z^{d-1}\tM\cap a_1^{\frac1{d-1}}\fC_\vecv\subset\{\bn\}\})
\ll p^{(d-1)}(a_1^{\frac1{d-1}}\fC_\vecv),
\end{align*}
which holds since $\Si_{d-1}$ can be covered by a finite number of 
fundamental regions for $X_1^{(d-1)}$. %
Also the function
$a_1\mapsto p^{(d-1)}(a_1^{\frac1{d-1}}\fC_\vecv)$ (for $a_1>0$)
is decreasing, because of Lemma \ref{CONVPDMONOTONELEM}.
Hence \eqref{GENCONVBOUNDPROP2LEMPF2} will follow if we prove that
for every $a_1\in[\sqrt3r,2r]$,
\begin{align}\label{GENCONVBOUNDPROP2LEMPF5}
p^{(d-1)}\bigl(a_1^{\frac1{d-1}}\fC_\vecv\bigr)
\ll\mu^{(d-1)}\bigl(\bigl\{\tM\in\Si_{d-1}\col
\ta_1\leq\sfrac2{\sqrt3}a_1^{\frac d{d-1}},\:
\Z^{d-1}\tM\cap a_1^{\frac1{d-1}}\fC_\vecv\subset\{\bn\}\bigr\}\bigr).
\end{align}
But $a_1^{\frac1{d-1}}\fC_\vecv\subset\overline{\scrB_{r'}^d}$
with $r'=a_1^{\frac1{d-1}}r$,
and because of $a_1\geq\sqrt3r$ we have
$r'>1$ and 
$\sfrac2{\sqrt3}a_1^{\frac d{d-1}}\geq
\sfrac2{\sqrt3}a_1^{\frac 1{d-1}}\sqrt3r=2r'$.
It follows that the right hand side of \eqref{GENCONVBOUNDPROP2LEMPF5} 
does not increase if we replace the condition
$\ta_1\leq\sfrac2{\sqrt3}a_1^{\frac d{d-1}}$ therein
by $\ta_1<2r'$.
Hence \eqref{GENCONVBOUNDPROP2LEMPF5} is true by induction.
\end{proof}

\begin{proof}[Proof of Proposition \ref{GENCONVBOUNDPROP2}]
Let $B$ be the unique closed ball of radius $r$ containing $\fC$.
If $r\leq\frac13$ then $|\fC|<\frac12$
and thus $p^{(d)}(\fC)>\frac12$ by Lemma \ref{TRIVPLOWBOUNDLEM}.
Hence from now on we may assume $r>\frac13$.

Let us first assume that $B$ lies within distance $r$ from the origin.
Then $B\subset\overline{\scrB_{3r}^d}$. Hence if
$a_1>3r$ then $B\cap(na_1\vecv+\vecv^\perp)=\emptyset$
for all $n\in\Z\setminus\{0\}$ and $\vecv\in\S_1^{d-1}$,
so that, by \eqref{LATTICEINPARAM},
\begin{align*}
\Z^d[a_1,\vecv,\vecu,\tM]\cap\fC
=a_1^{-\frac1{d-1}}\iota(\Z^{d-1}\tM)f(\vecv)\cap\fC,
\qquad %
\forall
\vecv\in\S_1^{d-1},\:\vecu\in\R^{d-1},\:\tM\in G^{(d-1)}.
\end{align*}
Using also \eqref{SIDSUBSSIDM1} and \eqref{SLDRSPLITHAAR} we get:
\begin{align}\notag
&p^{(d)}(\fC)\gg\mu\bigl(\bigl\{M\in\Si_d\col \Z^dM\cap\fC\subset\{\bn\}\bigr\}\bigr)
\\\label{GENCONVBOUNDPROP2PF1}
&\gg\int_{\S_1^{d-1}}\int_{3r}^\infty
\mu^{(d-1)}\bigl(\bigl\{\tM\in\Si_{d-1}\col
\ta_1\leq\sfrac2{\sqrt3}a_1^{\frac d{d-1}},\:
\Z^{d-1}\tM\cap a_1^{\frac1{d-1}}\fC_\vecv\subset\{\bn\}\bigr\}\bigr)
\,\frac{da_1}{a_1^{d+1}}\,d\vecv,
\end{align}
where $\fC_\vecv$ is as in \eqref{CVDEFNEW}.
We have $a_1^{\frac1{d-1}}\fC_\vecv\subset\overline{\scrB_{r'}^d}$
with $r'=3ra_1^{\frac1{d-1}}$.
Now if $a_1>3\sqrt3r$ then
$\sfrac2{\sqrt3}a_1^{\frac d{d-1}}> %
2r'$ and $r'>1$, and thus by Lemma \ref{GENCONVBOUNDPROP2LEM}, if also 
$d\geq3$,
\begin{align*}
\mu^{(d-1)}\bigl(\bigl\{\tM\in\Si_{d-1}\col
\ta_1\leq\sfrac2{\sqrt3}a_1^{\frac d{d-1}},\:
\Z^{d-1}\tM\cap a_1^{\frac1{d-1}}\fC_\vecv\subset\{\bn\}\bigr\}\bigr)
\gg p^{(d-1)}\bigl(a_1^{\frac1{d-1}}\fC_\vecv\bigr).
\end{align*}
Using this in \eqref{GENCONVBOUNDPROP2PF1} and recalling that the function
$a_1\mapsto p^{(d-1)}(a_1^{\frac1{d-1}}\fC_\vecv)$ is decreasing by
Lemma~\ref{CONVPDMONOTONELEM},
it follows that \eqref{GENCONVBOUNDPROP2RES} holds
for any fixed choice of $k_4>(3\sqrt3)^{\frac1{d-1}}$, and a
corresponding appropriate constant $k_3>0$.
On the other hand if $d=2$ then if we write 
$3r\overline{\fC_\vecv}=[\alpha_\vecv,\beta_\vecv]$
(set $\alpha_\vecv=\beta_\vecv=0$ when $\fC_\vecv=\emptyset$),
\eqref{GENCONVBOUNDPROP2PF1} takes the form
\begin{align}\label{GENCONVBOUNDPROP2PF1a}
p^{(2)}(\fC)
\gg r^{-2}\int_{\S_1^1}\int_1^\infty
I\bigl(\Z\cap [\alpha_\vecv x,\beta_\vecv x]\subset\{0\}\bigr)
\,\frac{dx}{x^3}\,d\vecv.
\end{align}
However, by using some of the steps from the proof of
\eqref{GENCONVBOUNDPROP2LEMPF4},
the inner integral in \eqref{GENCONVBOUNDPROP2PF1a} 
(even when restricted to $x\in[1,2]$)
is seen to be $\gg1$ whenever
$\beta_\vecv\leq\alpha_\vecv+\frac13$.
Hence \eqref{GENCONVBOUNDPROP2RES} holds with $k_4=9$.

It now remains to treat the case when $B$ has distance
$>r$ from the origin. We may assume that the center of $B$ is $t\vece_1$,
where $t>2r$. 
Given $\vecv\in\S_1^{d-1}$ we note that $na_1\vecv+\vecv^\perp$
has nonempty intersection with $B$ 
only if $|na_1-t\vecv\cdot\vece_1|\leq r$,
and hence $(na_1\vecv+\vecv^\perp)\cap B=\emptyset$ holds for all
$n\in\Z$ whenever $a_1>0$ satisfies
$a_1\Z\cap[t\vecv\cdot\vece_1-r,t\vecv\cdot\vece_1+r]=\emptyset$.
Let $A_\vecv$ be the set of all $a_1>0$ satisfying this condition.
By \eqref{LATTICECONTAINEMENT} 
we have %
$\Z^dM\cap B=\emptyset$
whenever $M=[a_1,\vecv,\vecu,\tM]$ with
$\vecv\in S$ and $a_1\in A_\vecv$.
Hence by \eqref{SIDSUBSSIDM1} and \eqref{SLDRSPLITHAAR},
and since $r>\frac13$,
\begin{align*}
p^{(d)}(\fC)\gg
\int_{\S_1^{d-1}}\int_{A_\vecv}\mu^{(d-1)}\bigl(\bigl\{\tM\in\Si_{d-1}\col
\ta_1\leq\sfrac2{\sqrt3}a_1^{\frac d{d-1}}\bigr\}\bigr)
\,\frac{da_1}{a_1^{d+1}}\,d\vecv
\gg\int_{\S_1^{d-1}}\int_{A_\vecv\cap(3r,\infty)}\frac{da_1}{a_1^{d+1}}\,d\vecv
\\
\gg r^{-d-1}\int_{\S_1^{d-1}}\bigl|A_\vecv\cap(3r,30r)\bigr|\,d\vecv
=r^{-d}\int_{\S_1^{d-1}}
\int_3^{30}I\bigl(x\Z\cap[\alpha_\vecv-1,\alpha_\vecv+1]=\emptyset\bigr)\,dx
\,d\vecv,
\end{align*}
where in the last step we wrote $\alpha_\vecv:=r^{-1}t\vecv\cdot\vece_1$.
Now for every $\vecv\in\S_1^{d-1}$ with 
$|\vecv\cdot\vece_1|>\frac12$ we have $|\alpha_\vecv|>1$,
and this is easily seen to imply that the inner integral is $\gg1$.
It follows that 
$p^{(d)}(\fC)\gg r^{-d}$, and thus \eqref{GENCONVBOUNDPROP2RES} holds.
\end{proof}

\begin{remark}\label{GENCONVEXPLREM}
Proposition \ref{GENCONVBOUNDPROP2} 
(together with Lemma \ref{CONVPDMONOTONELEM}) implies that
the bound in Proposition~\ref{GENCONVBOUNDPROP1} 
is sharp if we require the convex set $\fC$ to have a 
bounded ratio between its circumradius and its inradius.
However, for \textit{every} bounded convex set $\fC$ with nonempty interior
there is some $M\in G$ such that the ``John ellipsoid'' of $\fC M$
(viz.\  the unique $d$-dimensional ellipsoid of maximal volume 
contained in $\fC M$)
is a ball, and the ratio between the circumradius and the inradius of $\fC M$
is then $\leq d$; cf.\ \cite{John48}.
Also recall that $p(\fC)=p(\fC M)$.
Hence by induction on $d$ we obtain a fairly explicit, %
sharp upper bound on $p(\fC)$ for $\fC$ convex,
which we now state.

For any bounded convex set $\fC\subset\R^d$ with non-empty interior,
define $F_d(\fC)$ as follows.
Set $F_1(\fC)=1$ if $|\fC|\leq1$, otherwise $F_1(\fC)=0$.
Then for $d\geq2$, define $F_d(\fC)$ recursively
by taking $M\in\SL_d(\R)$ so that 
the John ellipsoid of $\fC M$
is a \textit{ball}, and setting
\begin{align}\label{FDCDEF}
F_d(\fC)=\min\biggl\{1,|\fC|^{-1}\int_{\S_1^{d-1}}
F_{d-1}\Bigl(|\fC|^{\frac1{d(d-1)}}\fC M\cap\vecv^\perp\Bigr)\,d\vecv\biggr\}.
\end{align}
(This makes $F_d(\fC)$ well-defined, although $M$ is determined
only up to multiplication from the right by an arbitrary
element of $\SO_d(\R)$.)
We extend the definition to arbitrary convex sets $\fC\subset\R^d$
by setting $F_d(\fC)=1$ if $\fC$ has empty interior,
and $F_d(\fC)=0$ if $\fC$ has non-empty interior and is unbounded.

\textit{Then for any $d\geq2$ there exist constants $k_1>k_2>0$ 
such that for every convex set $\fC\subset\R^d$, %
\begin{align}\label{GENCONVEXTHMRES}
\min\bigl\{\sfrac12,F_d(k_1\fC)\bigr\}\leq p(\fC)\leq F_d(k_2\fC).
\end{align}}

\vspace{-10pt}

It is %
an interesting question whether $F_d$ as defined in
\eqref{FDCDEF} can
be replaced by some %
geometrically more transparent
function, so that \eqref{GENCONVEXTHMRES} still holds.
\end{remark}

To conclude this section,
we point out some situations when
the Athreya--Margulis bound $p(\fC)\ll|\fC|^{-1}$ is sharp.
The following is immediate from Proposition \ref{GENCONVBOUNDPROP2}.
\begin{cor}\label{TRIVAMSHARPCOR}
Given $d\geq2$ and $\ve>0$ there is a constant $c=c(d,\ve)>0$ 
such that the following
holds.
For any convex set $\fC\subset\R^d$ which is 
contained in a $d$-dimensional ball of volume
$\leq\ve^{-1}|\fC|$ and which satisfies
$\vol_{\S_1^{d-1}}\{\vecv\in\S_1^{d-1}\col\vecv^\perp\cap\fC=\emptyset\}
\geq\ve$, we have
\begin{align*}
p(\fC)\geq\min\bigl(\sfrac12,c|\fC|^{-1}\bigr).
\end{align*}
\end{cor}

Note that Corollary \ref{TRIVAMSHARPCOR} 
(just as Proposition \ref{GENCONVBOUNDPROP2}) does not
take into account the invariance $p(\fC M)=p(\fC)$.
We next give a simple criterion %
which is invariant under 
$\fC\mapsto \fC M$ ($\forall M\in G$):

\begin{cor}\label{GENCONVAMSHARPCOR}
Given $d\geq2$ and $\ve>0$ there is a constant $c=c(d,\ve)>0$
such that for any convex set $\fC\subset\R^d$,
if $t\vecq\notin\fC$ for all $0\leq t\leq\ve$, 
where $\vecq$ is the centroid of $\fC$,
then
\begin{align}\label{GENCONVAMSHARPCORRES}
p(\fC)\geq\min\bigl(\sfrac12,c|\fC|^{-1}\bigr).
\end{align}
\end{cor}

For the proof we need the following simple geometric lemma.
\begin{lem}\label{JOHNBALLATCENTROIDLEM}
Given $d\geq1$ and $\ve>0$ there is some $c=c(d,\ve)>0$ such that
for any $d$-dimensional ball $B$ and any convex set $\fC\subset B$ of
volume $|\fC|\geq\ve|B|$, there exists a ball $B'$ with center at the
centroid of $\fC$ such that $B'\subset\fC$ and 
$|B'|\geq c|B|$.     %
\end{lem}
\begin{proof}
Let $\fC$ and $B$ be given satisfying the assumptions.
We may assume that the centroid of $\fC$ is $\bn$.
Let $R$ be the radius of $B$, and let
$r>0$ be maximal with the property $B':=\scrB_r^d\subset\fC$.
Then there is a point $\vecp\in\partial\fC$ with $\|\vecp\|=r$.
After a rotation we may assume $\vecp=r\vece_1$;
then $\fC\subset\{\vecx\col\vecx\cdot\vece_1\leq r\}$.
For $x\in\R$ we write $\fC_x=\{\vecy\in\R^{d-1}\col(x,\vecy)\in\fC\}$
and $|\fC_x|=\vol_{d-1}(\fC_x)$.
Now since the centroid of $\fC$ is $\bn$ we have
$0=\int_{-\infty}^r x|\fC_x|\,dx
\leq -r\int_{-\infty}^{-r}|\fC_x|\,dx+r\int_{-r}^r|\fC_x|\,dx$,
and hence
\begin{align*}
|\fC|=\int_{-\infty}^{-r}|\fC_x|\,dx+\int_{-r}^r|\fC_x|\,dx
\leq2\int_{-r}^r|\fC_x|\,dx
\leq4r\vol_{d-1}(\scrB_R^{d-1})\ll rR^{d-1}.
\end{align*}
But also $|\fC|\geq\ve|B|\gg\ve R^d$; thus
$r\gg\ve R$ and $|B'|\gg r^d\gg\ve^dR^d\gg\ve^d|B|$.   %
\end{proof}
\begin{proof}[Proof of Corollary \ref{GENCONVAMSHARPCOR}]
If $\fC$ has no interior points then 
$\fC$ is contained in some affine subspace
of $\R^d$ and thus $p(\fC)=1$ by Lemma \ref{TRIVPLOWBOUNDLEM}.
If $\fC$ has interior points and is unbounded then $|\fC|=\infty$.
Hence from now on we may assume that $\fC$ has interior points and
is bounded.
Then by \cite{John48}, after replacing $\fC$ with $\fC M$ for 
an appropriate $M\in G$ (note that both the assumption and the
conclusion of Corollary \ref{GENCONVAMSHARPCOR}
are invariant under any such replacement),
there is a $d$-dimensional ball $B$ satisfying
$\fC\subset B$ and $|\fC|\gg|B|$.
Hence by Lemma \ref{JOHNBALLATCENTROIDLEM} there is some 
$r>0$ with $r^d\gg|B|$ and $\vecq+\scrB_r^d\subset\fC$.

Now let $\vecp$ be the point in $\overline\fC$ lying closest to $\bn$,
and let $\veca$ be the point on the line segment 
$\R\vecq\cap\overline\fC$ lying closest to $\bn$.
Assume $\vecp\neq\veca$.
Let $\ell$ be the line through $\vecp$ and $\veca$,
and let $\vecb$ and $\vecc$ be the orthogonal projections onto $\ell$ of
$\vecq$ and $\bn$, respectively.
Then $\veca$ lies on the line segment between $\vecb$ and $\vecc$,
since $\veca$ lies on the line segment between $\vecq$ and $\bn$.
Also $\|\vecc\|\leq\|\vecp\|\leq\|\veca\|$ by the definition of $\vecc$
and $\vecp$;
thus $\vecp$ and $\veca$ must lie on the same side of $\vecb$ along $\ell$.
Combining this with the fact that $\vecp,\veca\in\partial\fC$
it follows that $\vecb$ cannot belong to the interior of $\fC$,
and in particular $\vecb\notin\vecq+\scrB_r^d$,
viz.\ $\|\vecb-\vecq\|\geq r$.
Using also $\triangle\vecc\veca\bn\sim\triangle\vecb\veca\vecq$ we obtain:
\begin{align*}
\|\vecp\|\geq\|\vecc\|=\frac{\|\vecb-\vecq\|}{\|\veca-\vecq\|}\|\veca\|
>\frac{r}{\|\vecq\|}\|\veca\|\geq\ve r.
\end{align*}
The same conclusion, $\|\vecp\|\geq\ve r$, trivially holds also when
$\vecp=\veca$.
Now if $R$ is the radius of $B$ then $R\ll r$ and 
$\fC$ is contained in $\vecp+\overline{\scrB_{2R}^d}$
and also in the half space
$\{\vecx\col\vecx\cdot\vecp\geq\|\vecp\|^2\}$.
It follows that $\fC\cap\vecv^\perp=\emptyset$ for every
$\vecv\in\S_1^{d-1}$ with $\varphi(\vecv,\vecp)<\arctan(\frac{\ve r}{2R})$.
Hence \eqref{GENCONVAMSHARPCORRES} now follows from 
Corollary \ref{TRIVAMSHARPCOR}.
\end{proof}

Note that the criterion given in Corollary \ref{GENCONVAMSHARPCOR} is
optimal for $\fC$ running through the family of all 
$d$-dimensional ellipsoids.
Namely, if $\fC$ is a $d$-dimensional ellipsoid with center $\vecq\neq\bn$
and volume $\gg1$, then, by Theorem \ref{GENBALLTHM} which we will prove below,
$p(\fC)\gg|\fC|^{-1}$ holds \textit{if and only if}
the distance from $\bn$ to $\R\vecq\cap\fC$ is $\gg\|\vecq\|$.
On the other hand, $p(\fC)\gg|\fC|^{-1}$ certainly holds
also for many convex sets which do not fulfill the criterion in
Corollary \ref{GENCONVAMSHARPCOR}.
For example it %
holds when $\fC$ is an arbitrary 
convex cone %
of volume $\gg1$ with $\bn$ as apex:
\begin{cor}\label{ATHREYAMFROMBELOWPROP1}
Given $d\geq2$ there is a constant $c=c(d)>0$ such that
if $\fC$ is the cone which is the convex hull of $\bn$ and
some convex subset of a hyperplane $\Pi\subset\R^d$ with $\bn\notin\Pi$,
then
\begin{align}\label{ATHREYAMFROMBELOWPROP1RES}
p(\fC)\geq\min(\sfrac12,c|\fC|^{-1}).
\end{align}
\end{cor}
\begin{proof}
Assume that $\fC$ is the convex hull of $\bn$ and the convex set 
$\fZ\subset\Pi$.
As usual we may assume that $\fC$ is bounded and has interior points,
i.e. that $\fZ$ is bounded and has a nonempty interior relative to $\Pi$.
Using \eqref{PDINV} and \cite{John48} we may furthermore assume that 
$\fZ$ is contained in a $(d-1)$-dimensional ball $B$ with
$\vol_{d-1}(B)\ll\vol_{d-1}(\fZ)$,
that the orthogonal projection of $\bn$ onto $\Pi$ equals the center
$\vecq$ of $B$, and that $\|\vecq\|$ equals the radius of $B$.
It then follows that
$\fC\cap\vecv^\perp=\emptyset$ for all $\vecv\in\S_1^{d-1}$ with
$\varphi(\vecv,\vecq)<\frac\pi4$
and also that $\fC$ is contained in a ball of volume $\ll|\fC|$.
Hence \eqref{ATHREYAMFROMBELOWPROP1RES} now follows from 
Corollary \ref{TRIVAMSHARPCOR}.
\end{proof}

\section{\texorpdfstring{The case of $\fC$ a ball; proof of Theorem \ref*{GENBALLTHM}}{The case of $C$ a ball; proof of Theorem 1.2}}
\label{GENBALLTHMPROOFSEC}

We now give the proof of Theorem \ref{GENBALLTHM}.
Let $\fC$, $r$ and $\vecq$ be as in Theorem \ref{GENBALLTHM},
viz.\ $\fC$ is a $d$-dimensional ball of volume $\geq\frac12$,
with radius $r$ and center $\vecq$.
After a rotation we may assume $\vecq=\|\vecq\|\vece_1$.
We will use the bounds in Propositions \ref{GENCONVBOUNDPROP1} and
\ref{GENCONVBOUNDPROP2}.
Note that it is clear from start that these bounds are sharp,
since the inradius and the circumradius of $\fC$ are the same ($=r$);
cf.\ the discussion in Remark \ref{GENCONVEXPLREM}.
Note also that since the intersection of $\fC$ with any hyperplane
is again a ball (of dimension $d-1$) or empty,
the integrands appearing in Propositions \ref{GENCONVBOUNDPROP1} and
\ref{GENCONVBOUNDPROP2} can be bounded in a sharp way simply
by induction, using Theorem \ref{GENBALLTHM} with $d-1$ in place of $d$.

To treat the integrals over $\S_1^{d-1}$
we parametrize a dense open subset of $\S_1^{d-1}$ as follows:
\begin{align}\label{VSHORTPARA}
\vecv&=(v_1,\ldots,v_d)
=\bigl(\cos\varpi,(\sin\varpi)\alpha_1,
(\sin\varpi)\alpha_2,\ldots,(\sin\varpi)\alpha_{d-1}\bigr)
\in\S_1^{d-1},
\end{align}
where $\varpi\in(0,\pi)$ and
$\vecalf=(\alpha_1,\ldots,\alpha_{d-1})\in\S_1^{d-2}$.
Thus $\varpi$ is the angle between $\vecv$ and $\vece_1$.
In this parametrization the $(d-1)$-dimensional volume measure on 
$\S_1^{d-1}$ takes the following form:
\begin{align}\label{DVINSHORTPARA}
d\vecv=(\sin\varpi)^{d-2}\,d\varpi\,d\vecalf,
\end{align}
where $d\vecalf$ is the $(d-2)$-dimensional volume measure on $\S_1^{d-2}$
(if $d=2$, $d\vecalf$ is the counting measure on $\S_1^0=\{-1,1\}$).

Let us first assume $\bn\in\overline\fC$.
Then $\tau=\frac{\|\vecq\|-r}r\in[-1,0]$ in \eqref{GENBALLTHMRES}.
Given $\vecv$ as in \eqref{VSHORTPARA} and using $\vecq=\|\vecq\|\vece_1$,
we compute that $\fC\cap\vecv^\perp$
is a $(d-1)$-dimensional ball of radius
\begin{align}\notag
r'=\sqrt{r^2-\|\vecq\|^2\cos^2\varpi}=r\sqrt{1-(1+\tau)^2\cos^2\varpi}
=r\sqrt{\sin^2\varpi-\tau(2+\tau)\cos^2\varpi}
\hspace{30pt}
\\\label{GENBALLTHMPFR}
\asymp r\bigl(\sin\varpi+\sqrt{|\tau|}\bigr)
\end{align}
This ball contains $\bn$ in its closure, 
and the distance from $\bn$ to its
boundary (relative to $\vecv^\perp$) is
$r'-\|\vecq\|\sin\varpi$.
Hence if we write $\tau'$ for the 
``$\tau$-ratio'' (the analogue of $\tau$ in \eqref{GENBALLTHMRES})
of this $(d-1)$-dimensional ball then 
$\tau'\in[-1,0]$ and 
\begin{align}\notag
-\tau'=\frac{r'-\|\vecq\|\sin\varpi}{r'}
\asymp\frac{{r'}^2-\|\vecq\|^2\sin^2\varpi}{{r'}^2}
\asymp
\frac{1-(1+\tau)^2\cos^2\varpi-(1+\tau)^2\sin^2\varpi}{\sin^2\varpi+|\tau|}
\hspace{30pt}
\\\label{GENBALLTHMPFTAU}
\asymp\frac{|\tau|}{\sin^2\varpi+|\tau|}.
\end{align}

Now in the case $d=2$, Proposition \ref{GENCONVBOUNDPROP1} gives
\begin{align}\label{GENBALLTHMPF1}
p^{(2)}(\fC)\ll r^{-2}\int_0^{\pi} I(krr'\leq1)\,d\varpi,
\end{align}
where $k>0$ is an absolute constant.
By \eqref{GENBALLTHMPFR}, $krr'\leq1$ holds only if
$\sin\varpi+\sqrt{|\tau|}\ll r^{-2}$.
Hence $p^{(2)}(\fC)\ll r^{-4}\ll |\fC|^{-2}$,
and if $r^4|\tau|$ is sufficiently large we even have $p^{(2)}(\fC)=0$.
Similarly using Proposition \ref{GENCONVBOUNDPROP2}
and the assumption that $|\fC|\geq\frac12$,
we conclude that $p^{(2)}(\fC)\gg r^{-4}\gg|\fC|^{-2}$ whenever
$r^4|\tau|$ is sufficiently small.
Hence \eqref{GENBALLTHMRES} holds when $d=2$ and $\bn\in\overline\fC$.

Next assume $d\geq3$. Then Proposition \ref{GENCONVBOUNDPROP1} gives
\begin{align}\label{GENBALLTHMPF1a}
p^{(d)}(\fC)\ll r^{-d}\int_0^\pi 
p^{(d-1)}\Bigl(kr^{\frac1{d-1}}\fC\cap\vecv^\perp\Bigr)\,
(\sin\varpi)^{d-2}\,d\varpi,
\end{align}
where $k>0$ is a constant which only depends on $d$.
Here $kr^{\frac1{d-1}}\fC\cap\vecv^\perp$ is a $(d-1)$-dimensional
ball of radius $kr^{\frac1{d-1}}r'$ which contains $\bn$ in
its closure and which has $\tau$-ratio $\tau'$ as in \eqref{GENBALLTHMPFTAU}.
By induction we may assume that 
\eqref{GENBALLTHMRES} holds for this ball;
it follows that the integrand in \eqref{GENBALLTHMPF1a}
vanishes whenever
$|\tau'|(r^{\frac1{d-1}}r')^{\frac{2(d-1)}{d-2}}$ is sufficiently large,
viz.\ (by \eqref{GENBALLTHMPFR}, \eqref{GENBALLTHMPFTAU})
whenever $|\tau|(\sin^2\varpi+|\tau|)^{\frac1{d-2}}r^{\frac{2d}{d-2}}$
is sufficiently large.
In particular $p^{(d)}(\fC)=0$ holds whenever 
$|\tau| r^{\frac{2d}{d-1}}$ is sufficiently large.
Furthermore for arbitrary $\tau\leq0$ it follows that
\begin{align}\label{GENBALLTHMPF1b}
p^{(d-1)}(kr^{\frac1{d-1}}\fC\cap\vecv^\perp)
\ll(r^{\frac1{d-1}}r')^{-2(d-1)}
\ll r^{-2d}(\sin\varpi)^{-2(d-1)}.
\end{align}
\begin{remark}\label{TRIVHALFREM}
Here we used the trivial fact that, %
for suitable $k_1$, the \textit{right}
inequality in \eqref{GENBALLTHMRES} holds also when the 
given body has volume \textit{less} than $\frac12$,
since $p(\fC)\leq1$ by definition.
\end{remark}
Now in \eqref{GENBALLTHMPF1a} we may restrict the range of integration
to $\varpi\in(0,\frac\pi2)$ by symmetry
($\varpi\leftrightarrow\pi-\varpi$).
For $\varpi\leq r^{-\frac d{d-1}}$ we use the
trivial bound $p^{(d-1)}(kr^{\frac1{d-1}}\fC\cap\vecv^\perp)\leq1$
and for $\varpi$ larger we use \eqref{GENBALLTHMPF1b}. This gives
\begin{align*}
p^{(d)}(\fC)\ll r^{-d}\biggl(\int_0^{r^{-\frac d{d-1}}}\varpi^{d-2}\,d\varpi
+\int_{r^{-\frac d{d-1}}}^\infty r^{-2d}\varpi^{-d}\,d\varpi\biggr)
\ll r^{-2d}.
\end{align*}
By a similar computation
using Proposition \ref{GENCONVBOUNDPROP2} 
(and the assumption $|\fC|\geq\frac12$)
we obtain $p^{(d)}(\fC)\gg r^{-2d}$
whenever $\tau\leq0$ and
$|\tau| r^{\frac{2d}{d-1}}$ is sufficiently small.
Hence we have proved that \eqref{GENBALLTHMRES} holds when
$d\geq3$ and $\bn\in\overline\fC$.

We now turn to the remaining case, $\bn\notin\overline\fC$.
Now $\tau=\frac{\|\vecq\|-r}{\|\vecq\|}\in(0,1)$ in \eqref{GENBALLTHMRES}.
For any $\vecv\in\S_1^{d-1}$ as in \eqref{VSHORTPARA}
we note that $\overline{\fC}\cap\vecv^\perp$ is nonempty if and only if
$|\cos\varpi|\leq\frac r{\|\vecq\|}=1-\tau$,
or equivalently if and only if
$\varpi\in[\varpi_\tau,\pi-\varpi_\tau]$, where
$\varpi_\tau:=\arccos(1-\tau)\in(0,\frac\pi2)$.
In particular, by Proposition \ref{GENCONVBOUNDPROP2}
(and since $|\fC|\geq\frac12$),
$p^{(d)}(\fC)\gg r^{-d}\asymp|\fC|^{-1}$ whenever $\tau\geq\frac1{10}$,
in agreement with \eqref{GENBALLTHMRES}.
Hence from now on we may assume $0<\tau<\frac1{10}$
(and thus $0<\varpi_\tau<\arccos\frac9{10}<\frac\pi6$).

Now for any $\vecv\in\S_1^{d-1}$ with 
$\varpi\in[\varpi_\tau,\pi-\varpi_\tau]$
we note that $\overline{\fC}\cap\vecv^\perp$ is a
$(d-1)$-dimensional ball of radius
\begin{align*}
r'=\sqrt{r^2-\|\vecq\|^2\cos^2\varpi}
=r\sqrt{1-(1-\tau)^{-2}\cos^2\varpi}
\asymp r\sqrt{(1-\tau)^2-\cos^2\varpi}.
\end{align*}
Setting $\varpi_1:=\min(\varpi-\varpi_\tau,\pi-\varpi_\tau-\varpi)
\in[0,\frac\pi2-\varpi_\tau]$ we may continue:
\begin{align}\label{GENBALLTHMPFR2}
r'\asymp r\sqrt{1-\tau-\cos(\varpi_\tau+\varpi_1)}
=r\sqrt{(1-\tau)(1-\cos\varpi_1)+\sin\varpi_\tau\sin\varpi_1}
\asymp r\sqrt{(\varpi_1+\varpi_\tau)\varpi_1}.
\end{align}
Furthermore the ball $\overline{\fC}\cap\vecv^\perp$ does not contain $\bn$,
and the distance from $\bn$ to its boundary is
$\|\vecq\|\sin\varpi-r'$.
Hence the $\tau$-ratio of $\overline{\fC}\cap\vecv^\perp$ is
\begin{align}\label{GENBALLTHMPFTAU2}
\tau'=\frac{\|\vecq\|\sin\varpi-r'}{\|\vecq\|\sin\varpi}
\asymp\frac{\|\vecq\|^2\sin^2\varpi-{r'}^2}{\|\vecq\|^2\sin^2\varpi}
=\frac{\|\vecq\|^2-{r}^2}{\|\vecq\|^2\sin^2\varpi}
=\frac{(1-\tau)^{-2}-1}{(1-\tau)^{-2}\sin^2\varpi}
\asymp\frac{\tau}{\sin^2\varpi}.
\end{align}

Now in the case $d=2$, Proposition \ref{GENCONVBOUNDPROP1} gives
\begin{align}\label{GENBALLTHMPF2}
p^{(2)}(\fC)\ll r^{-2}\biggl(\varpi_\tau+\int_0^{\frac\pi2-\varpi_\tau} 
I(krr'\leq1)\,d\varpi_1\biggr),
\end{align}
where $k>0$ is an absolute constant.
The contribution from $\varpi_1\leq\varpi_\tau$
in the integral is subsumed by the term $\varpi_\tau$;
hence it suffices to consider
$\varpi_1\in(\varpi_\tau,\frac\pi2-\varpi_\tau)$.
For any such $\varpi_1$ we have $r'\asymp r\varpi_1$ by \eqref{GENBALLTHMPFR2}.
Hence, using also $\varpi_\tau\asymp\tau^{\frac12}$,
we conclude that $p^{(2)}(\fC)\ll r^{-2}(\tau^{\frac12}+r^{-2})$.
Similarly using Proposition \ref{GENCONVBOUNDPROP2} (and $|\fC|\geq\frac12$)
we obtain $p^{(2)}(\fC)\gg r^{-2}(\tau^{\frac12}+r^{-2})$.
Hence \eqref{GENBALLTHMRES} holds when $d=2$ and $\bn\notin\overline\fC$.

Next assume $d\geq3$.
Then Proposition \ref{GENCONVBOUNDPROP1} gives
\begin{align*}
p^{(d)}(\fC)\ll r^{-d}\biggl(\int_0^{\varpi_\tau}(\sin\varpi)^{d-2}\,d\varpi
+\int_{\varpi_\tau}^{\frac\pi2}
p^{(d-1)}\Bigl(kr^{\frac1{d-1}}\fC\cap\vecv^\perp\Bigr)
(\sin\varpi)^{d-2}\,d\varpi\biggr)
\hspace{50pt}&
\\
\ll r^{-d}\biggl(\varpi_\tau^{d-1}
+\int_{2\varpi_\tau}^{\frac\pi2}
p^{(d-1)}\Bigl(kr^{\frac1{d-1}}\fC\cap\vecv^\perp\Bigr)
\varpi^{d-2}\,d\varpi\biggr),&
\end{align*}
where $k>0$ is a constant which only depends on $d$.
Here $kr^{\frac1{d-1}}\fC\cap\vecv^\perp$ is a $(d-1)$-dimensional
ball of radius $kr^{\frac1{d-1}}r'$ which does not contain
$\bn$ in its closure and which has $\tau$-ratio $\tau'$ as in
\eqref{GENBALLTHMPFTAU2}.
By induction we may assume that \eqref{GENBALLTHMRES} holds for
this ball (also recall Remark~\ref{TRIVHALFREM}); it follows that
\begin{align*}
p^{(d)}(\fC)\ll r^{-d}\biggl(\varpi_\tau^{d-1}+
\int_{2\varpi_\tau}^{\frac\pi2}
\min\Bigl\{1,(r^d\varpi^{d-1})^{-1}(\tau/\varpi^2)^{\frac{d-2}2}
+(r^d\varpi^{d-1})^{-2}\Bigr\}
\varpi^{d-2}\,d\varpi\biggr)
\hspace{30pt}&
\\
\ll r^{-d}\biggl(\tau^{\frac{d-1}2}+
\int_{2\varpi_\tau}^{\infty}r^{-d}\tau^{\frac{d-2}2}\varpi^{1-d}\,d\varpi
+\int_{2\varpi_\tau}^{\infty}\min\Bigl\{1,r^{-2d}\varpi^{2-2d}\Bigr\}
\varpi^{d-2}\,d\varpi\biggr)
\hspace{30pt}&
\\
\ll r^{-d}\biggl(\tau^{\frac{d-1}2}+r^{-d}+
\int_{2\varpi_\tau}^{\max(2\varpi_\tau,r^{-\frac d{d-1}})}
\varpi^{d-2}\,d\varpi
+\int_{r^{-\frac d{d-1}}}^\infty r^{-2d}\varpi^{-d}\,d\varpi\biggr)
\hspace{30pt}&
\\
\ll r^{-d}\Bigl(\tau^{\frac{d-1}2}+r^{-d}\Bigr).&
\end{align*}
This proves that the right inequality in \eqref{GENBALLTHMRES} holds for 
$d\geq3$, $\bn\notin\overline\fC$.

The proof of the left inequality in \eqref{GENBALLTHMRES} is similar:
Note $r\gg1$, since $|\fC|\geq\frac12$.
Now Proposition \ref{GENCONVBOUNDPROP2} gives, 
with a new constant $k>0$ (which only depends on $d$),
\begin{align}\label{GENBALLTHMPF3}
p^{(d)}(\fC)\gg r^{-d}\biggl(\int_0^{\varpi_\tau}(\sin\varpi)^{d-2}\,d\varpi
+\int_{\varpi_\tau}^{\frac\pi2}
p^{(d-1)}\Bigl(kr^{\frac1{d-1}}\fC\cap\vecv^\perp\Bigr)
(\sin\varpi)^{d-2}\,d\varpi\biggr).
\end{align}
It follows from \eqref{GENBALLTHMPFR2} that there is a constant $c>0$
which only depends on $d$ such that $cr^{-\frac d{d-1}}<\frac\pi2$ and
for all $\varpi\in(0,cr^{-\frac d{d-1}})$ the intersection
$kr^{\frac1{d-1}}\fC\cap\vecv^\perp$ is either empty or is a 
$(d-1)$-dimensional ball of radius $<\frac12$,
so that $p^{(d-1)}(kr^{\frac1{d-1}}\fC\cap\vecv^\perp)>\frac12$
by Lemma \ref{TRIVPLOWBOUNDLEM}.
Hence the right hand side of \eqref{GENBALLTHMPF3} is $\gg r^{-2d}$.
But also, by just considering the first integral, 
the right hand side of \eqref{GENBALLTHMPF3} is 
$\gg\tau^{\frac{d-1}2}r^{-d}$.
This concludes the proof of the left inequality in 
\eqref{GENBALLTHMRES} for $d\geq3$, $\bn\notin\overline\fC$,
and hence the proof of Theorem \ref{GENBALLTHM} is complete.
\hfill $\square$ $\square$ $\square$

\begin{remark}\label{MINKOWSKIREMARK}
Let us note that the case of the right inequality in \eqref{GENBALLTHMRES}
which says that $p(\fC)=0$ 
whenever $\bn\in\fC$ and $|\tau| |\fC|^{\frac{2}{d-1}}$ is sufficiently
large, may alternatively be proved by a simple application
of Minkowski's Theorem:
We may assume that $\fC$ is the open ball
$\fC=(r-t)\vece_1+\scrB_r^d$ ($0<t\leq r$, $r>0$; thus $\tau=-t/r$).
We then claim that the box
\begin{align*}
F=\bigl[-\sfrac12r,\sfrac12r\bigr]
\times\underbrace{\biggl[-\frac{\sqrt{|\tau|}}{\sqrt{2d}} r,
\frac{\sqrt{|\tau|}}{\sqrt{2d}} r\biggr]
\times\cdots\times\biggl[-\frac{\sqrt{|\tau|}}{\sqrt{2d}} r,
\frac{\sqrt{|\tau|}}{\sqrt{2d}} r\biggr]}_{\text{$d-1$ copies}}
\end{align*}
is contained in the union $\fC\cup(-\fC)$.
To verify this it clearly suffices to check that every point
$(x_1,\ldots,x_d)\in F$ with $x_1\geq0$ lies in $\fC$.
For such a point we have
\begin{align*}
(x_1-(r-t))^2+x_2^2+\ldots+x_d^2
<\max\bigl((r-t)^2,(\sfrac12r-(r-t))^2\bigr)+\sfrac12tr
\\
=\max\bigl(r^2+t(t-\sfrac32r),\sfrac14r^2+t(t-\sfrac12r)\bigr)<r^2,
\end{align*}
thus proving the claim.

Now assume $p(\fC)>0$.
Then there exists a lattice $L\in X_1$ satisfying $L\cap\fC=\{\bn\}$.
Also $L\cap(-\fC)=\{\bn\}$, since $L=-L$.
Hence $L\cap F=\{\bn\}$, and since $F$ is convex and symmetric about the 
origin we conclude via Minkowski's Theorem 
(cf., e.g., \cite[Thm.\ 10]{siegel}) that the volume of $F$ is
$|F|\leq2^d$.
But $|F|\asymp |\tau|^{\frac{d-1}2}r^d$.
Hence $p(\fC)>0$ implies $|\tau| r^{\frac{2d}{d-1}}\ll1$,
viz.\ $|\tau| |\fC|^{\frac2{d-1}}\ll1$,
as desired.
\end{remark}

To conclude this section let us prove that 
Theorem~\ref{GENBALLTHM} leads to a simple explicit sharp 
bound on $p(\fC)$ for any convex body $\fC$
such that $\partial\fC$ has pinched positive curvature
(or, more generally, for any $\fC$ which can be transformed to such a 
body by a map in $\SL_d(\R)$; cf.\ \eqref{PDINV}).
For any given convex body $\fC$ of
class $\C^1$ and any $\vect$ in $T^1(\partial\fC)$,
the unit tangent bundle of $\partial\fC$,
we denote by $\rho_i(\vect)$ and $\rho_s(\vect)$ 
the lower and upper radius of curvature of $\partial\fC$ 
at the base point of $\vect$, in direction $\vect$
(cf.\ \cite[Sec.\ 2.5]{schneider}). %

\begin{cor}\label{GENCURVCOR}
Given $d\geq2$ and $C>1$ there exist
constants $0<k_1<k_2$ such that the following holds.
For any convex body $\fC\subset\R^d$ of class $\C^1$,
volume $|\fC|\geq\frac12$, and 
satisfying 
$\sup_{\vect\in T^1(\partial\fC)}\rho_s(\vect)
/\inf_{\vect\in T^1(\partial\fC)}\rho_i(\vect)
\leq C$,
we have
\begin{align}\label{GENCURVCORRES}
F_{ball}^{(d)}(\tau;k_2|\fC|)\leq
p(\fC)\leq F_{ball}^{(d)}(\tau;k_1|\fC|),
\qquad\text{with }\:\tau=\frac{\delta}{|\delta|+|\fC|^{1/d}},
\end{align}
where $\delta$ is the signed distance from $\bn$ to $\partial\fC$
(negative if $\bn\in\fC^\circ$, positive if $\bn\notin\overline\fC$),
and $F_{ball}^{(d)}$ is the function defined in \eqref{FBALLDEF}.
\end{cor}
\begin{proof}
Let $\fC\subset\R^d$ be any convex body satisfying the assumptions,
and set
$r_1=\inf_{\vect\in T^1(\partial\fC)}\rho_i(\vect)$
and $r_2=\sup_{\vect\in T^1(\partial\fC)}\rho_s(\vect)$
(thus $r_2/r_1\leq C$).
Let $\vecp$ be the point in $\partial\fC$ lying closest to $\bn$,
and let $\fC'$, $\fC''$ be the open %
balls of 
radii $r_1$ and $r_2$, respectively,
which are tangent to $\partial\fC$ at $\vecp$
and lie on the same side of the tangent plane as $\fC$.
Then $\fC'\subset\fC\subset\overline{\fC''}$.
If $\bn\in\fC$ and $\|\vecp\|>r_1$ then we replace $\fC'$
with the ball $\scrB_{\|\vecp\|}^d$; 
note that $\fC'\subset\fC$ %
is still true.
Let $\tau_{\fC'}$ and $\tau_{\fC''}$ be the $\tau$-ratios of
$\fC'$ and $\fC''$, as in \eqref{GENBALLTHMRES}.
Now by Theorem \ref{GENBALLTHM} we have,
for some constants $0<k_1<k_2$ which only depend on $d$,
\begin{align}\label{GENCURVCORPF1}
F_{ball}^{(d)}(\tau_{\fC''},k_2|\fC''|)\leq
p(\fC'')\leq p(\fC)\leq p(\fC')\leq 
F_{ball}^{(d)}(\tau_{\fC'},k_1|\fC'|).
\end{align}
(Here note that $|\fC''|\geq|\fC|\geq\frac12$,
and also recall Remark~\ref{TRIVHALFREM}.)
It follows from our construction that 
$|\fC'|\asymp_{_C}|\fC''|\asymp_{_C}|\fC|$,
and if $\tau=0$ then $\tau_{\fC'}=\tau_{\fC''}=0$;
if $\tau>0$ then $\tau_{\fC'},\tau_{\fC''}>0$ and
$\tau_{\fC'}\asymp_{_C}\tau_{\fC''}\asymp_{_C}\tau$;
if $\tau<0$ then $\tau_{\fC'},\tau_{\fC''}<0$ and
$|\tau_{\fC'}|\asymp_{_C}|\tau_{\fC''}|\asymp_{_C}|\tau|$.
(Here the implied constant in each
``$\asymp_{_C}$'' %
depends only on $d$ and $C$.)
Hence there exist constants $k_1'\in(0,k_1)$ and $k_2'>k_2$ which only
depend on $d$ and $C$ such that
$F_{ball}^{(d)}(\tau_{\fC'},k_1|\fC'|)
\leq F_{ball}^{(d)}(\tau,k_1'|\fC|)$
and $F_{ball}^{(d)}(\tau,k_2'|\fC|)\leq 
F_{ball}^{(d)}(\tau_{\fC''},k_2|\fC''|)$.
Combining this with \eqref{GENCURVCORPF1}, the proof is complete.
\end{proof}

\section{\texorpdfstring{The case of $\fC$ a cut ball; proof of Theorem \ref*{MAINTHM2}}{The case of C a cut ball; proof of Theorem 1.3}}
\label{MAINTHM2PFSEC}

\subsection{Reduction to the case of small cut ratio}

We now start preparing for the proof of Theorem \ref{MAINTHM2}.
In the following, %
by a ``cut ball'', we will always mean a set $\fC\subset\R^d$
of the form in Theorem \ref{MAINTHM2}, viz.
\begin{align}\label{CUTBALLRECDEF}
\fC:=%
B\cap\{\vecx\in\R^d\col \vecw\cdot\vecx>0\},
\end{align}
where $B$ is a $d$-dimensional ball with $\bn\in\overline B$
and $\vecw$ is a unit vector such that $\fC$ has nonempty interior.
Then if $r$ and $\vecp$ are the radius and center of $B$,
we will say that $\fC$ has \textit{radius} $r$
and \textit{cut ratio}
$t=1-r^{-1}(\vecw\cdot\vecp)\in[0,2)$.
Also if $r'$ and $\vecq$ are the radius and center of the
$(d-1)$-dimensional ball $B\cap\vecw^\perp\subset\partial\fC$
then $r'-\|\vecq\|$ is the distance between $\bn$ and the relative
boundary of $B\cap\vecw^\perp$, and we will refer to the
ratio $e=\frac{r'-\|\vecq\|}{r'}$ as the \textit{edge ratio} of $\fC$.
(We leave $e$ undefined when $r'=0$, viz.\ when $t=0$.)

Similarly if $\fC$ is as in Corollary \ref{CONECYLCOR},
viz.\ a cone or a cylinder containing $\bn$ in its base,
then we define the
\textit{edge ratio} of $\fC$ to be
$e=\frac{r'-\|\vecq\|}{r'}$, where $\vecq$ and $r'$ are the center and 
radius of the base %
which contains $\bn$.

To start with,
using simple containment arguments, we will prove a lemma which
reduces our task to that of proving Theorem \ref{MAINTHM2}
in the case of small cut ratio.
Specifically:
\begin{lem}\label{MAINTHM2REDLEM}
Let $d\geq2$. Assume that there exist constants $0<k_1<k_2$
such that 
\begin{align}\label{MAINTHM2REDLEMASS}
F_{\substack{cut\\ball}}^{(d)}\Bigl(e;\sfrac15;k_2|\fC|\Bigr)\leq
p(\fC)\leq F_{\substack{cut\\ball}}^{(d)}\Bigl(e;\sfrac15;k_1|\fC|\Bigr)
\end{align}
holds for any $d$-dimensional cut ball $\fC$ with $|\fC|\geq\frac12$,
edge ratio $e\in[0,1]$ and cut ratio $t=\frac15$.
Then there also exist constants $0<k_1'<k_2'$
such that 
\begin{align}\label{MAINTHM2REDLEMRES}
F_{\substack{cut\\ball}}^{(d)}\Bigl(e;t;k_2'|\fC|\Bigr)\leq
p(\fC)\leq F_{\substack{cut\\ball}}^{(d)}\Bigl(e;t;k_1'|\fC|\Bigr)
\end{align}
holds for any $d$-dimensional cut ball $\fC$ with
$|\fC|\geq\frac12$,
edge ratio $e\in[0,1]$ and cut ratio $\frac15\leq t<2$.
\end{lem}

It is convenient to first prove an auxiliary result which says that
Theorem \ref{MAINTHM2} in the special case $t=\frac15$
implies Corollary \ref{CONECYLCOR}. In precise terms:

\begin{lem}\label{MAINTHM2REDAUXLEM}
Let $d\geq2$. Assume that there exist constants $0<k_1<k_2$
such that 
\begin{align}\label{MAINTHM2REDAUXLEMASS}
F_{\substack{cut\\ball}}^{(d)}\Bigl(e;\sfrac15;k_2|\fC|\Bigr)\leq
p(\fC)\leq F_{\substack{cut\\ball}}^{(d)}\Bigl(e;\sfrac15;k_1|\fC|\Bigr)
\end{align}
holds for any $d$-dimensional cut ball $\fC$ with
$|\fC|\geq\frac12$,
edge ratio $e\in[0,1]$ and cut ratio $t=\frac15$.
Then there also exist constants $0<k_1'<k_2'$ such that
for any $d$-dimensional cone or cylinder $\fC$ as in
Corollary \ref{CONECYLCOR},
\begin{align}\label{MAINTHM2REDAUXLEMRES}
F_{cone}^{(d)}(e;k_2'|\fC|)\leq p(\fC)\leq F_{cone}^{(d)}(e;k_1'|\fC|),
\end{align}
where $e$ is the edge ratio of $\fC$.
\end{lem}

\begin{proof}[Proof of Lemma \ref{MAINTHM2REDAUXLEM}]
First assume that $\fC$ is an arbitrary
\textit{cylinder} as in Corollary \ref{CONECYLCOR}, with edge ratio $e$.
Then there is some $M\in G$ such that the two bases of $\fC M$
are balls (just as for $\fC$ itself),
$\fC M$ is right
(viz.\ the line between the centers of its bases is orthogonal to the bases),
and the radius of $\fC M$ equals its height.
Note that $\fC M$ has the same edge ratio as $\fC$.
Let $B$ be the base of $\fC M$ which contains $\bn$.
Now %
one checks by a quick computation
that if we let $\fC'$ be the unique cut ball with cut ratio $t=\frac15$ 
for which the flat part of $\partial\fC'$ equals $\overline B$,
with $\fC'$ lying on the same side of $B$ as $\fC M$,
then $\fC M\subset\overline{\fC'}$.
(If $\fC M$ has radius and height $r$ then the radius of $\fC'$ is $\frac53r$.)
Also $|\fC'|\geq|\fC|\geq\frac12$.
Hence by \eqref{MAINTHM2REDAUXLEMASS}
(using also Lemma \ref{PCONTLEM}),
\begin{align}\label{CYLFROMBELOW}
p(\fC)=p(\fC M)\geq p(\fC')\geq 
F_{\substack{cut\\ball}}^{(d)}(e;\sfrac15;k|\fC'|)
\geq F_{cone}^{(d)}(e;k'|\fC|)
\end{align}
for certain constants $k,k'>0$ which only depend on $d,k_1,k_2$.

For our next construction it is convenient to use explicit
coordinates.
Let $\fC_0$ be the (closed) cone which is the convex hull of
$\iota(\overline{\scrB_1^{d-1}})$ and $\vece_1-3\vece_2$.
Let $\fC_1$ be the (open) cut ball
\begin{align*}
\fC_1=\bigl(\sfrac1{15}\vece_1-\sfrac{19}{20}\vece_2+
\overline{\scrB_{1/12}^d}\bigr)
\cap\{\vecx\col\vecx\cdot\vece_1\geq0\}.
\end{align*}
One then checks by a straightforward computation that
$(s\vece_2+\fC_1)\subset\fC_0$ for all $0\leq s\leq\sfrac{19}{20}$.
Now if $\fC$ is an arbitrary \textit{cone} as in Corollary \ref{CONECYLCOR},
with edge ratio $e\in[0,1]$,
then there are some $M\in G$ and $c>0$ such that
$\overline\fC M=c((1-e)\vece_2+\fC_0)$,
and hence we conclude
\begin{align*}
c((1-e+s)\vece_2+\fC_1)\subset\overline\fC M,\qquad
\forall s\in[0,\sfrac{19}{20}].
\end{align*}
We apply this with $s=\max(0,e-\frac1{20})$;
then $c((1-e+s)\vece_2+\fC_1)$ is a cut ball with
cut ratio $t=\frac15$, edge ratio $\min(1,20e)$,
and volume $|c\fC_1|\gg |c\fC_0|=|\fC|$.
Hence by the right inequality in \eqref{MAINTHM2REDAUXLEMASS}
(which we may assume holds also when the cut ball has volume
$<\frac12$, similarly as in Remark~\ref{TRIVHALFREM}), we get:
\begin{align}\label{CONEFROMABOVE}
p(\fC)=p(\overline\fC M)\leq p\bigl(c((1-e+s)\vece_2+\fC_1)\bigr)
\leq 
F_{\substack{cut\\ball}}^{(d)}(\min(1,20e);\sfrac15;k|c\fC_1|)
\leq F_{cone}^{(d)}(e;k'|\fC|)
\end{align}
for certain constants $k,k'>0$ which only depend on $d,k_1,k_2$.
Hence we have proved that the
right inequality in \eqref{MAINTHM2REDAUXLEMRES} holds for $\fC$.
Furthermore we may inscribe $\fC$ in a cylinder $\fC'$ 
which has the same base as $\fC$ and thus edge ratio $e$,
and which has volume $|\fC'|=d|\fC|\ll|\fC|$.
Hence using $p(\fC)\geq p(\fC')$ and applying \eqref{CYLFROMBELOW} to $\fC'$,
we conclude that also the 
left inequality in \eqref{MAINTHM2REDAUXLEMRES} holds for $\fC$.

On the other hand, if $\fC$ is an arbitrary \textit{cylinder} as in 
Corollary \ref{CONECYLCOR}, then the
left inequality in \eqref{MAINTHM2REDAUXLEMRES} holds because of
\eqref{CYLFROMBELOW},
and the right inequality in \eqref{MAINTHM2REDAUXLEMRES} follows by
inscribing a cone $\fC'$ in $\fC$ with the same edge ratio as $\fC$
and volume $|\fC'|=d^{-1}|\fC|\gg|\fC|$,
using $p(\fC)\leq p(\fC')$ and applying 
\eqref{CONEFROMABOVE} to $\fC'$.
\end{proof}

\begin{proof}[Proof of Lemma \ref{MAINTHM2REDLEM}]
Let $\fC$ be an arbitrary cut ball with $|\fC|\geq\frac12$,
cut ratio $\frac15\leq t<2$ and edge ratio $e$.
If $1\leq t<2$ 
then let $\fC'$ be the unique closed right cylinder which has the 
flat part of $\partial\fC$ as one base, and has minimal height subject to the
condition $\fC\subset\fC'$.
Then $p(\fC)\geq p(\fC')$,
$\frac12\leq|\fC|\leq|\fC'|\ll|\fC|$, and $\fC'$ has 
the same edge ratio as $\fC$.
Hence we obtain, using %
Lemma \ref{MAINTHM2REDAUXLEM}:
\begin{align*}
p(\fC)\geq p(\fC')\geq F_{cone}^{(d)}(e,k|\fC'|)
\geq F_{\substack{cut\\ball}}^{(d)}(e;t;k'|\fC|),
\end{align*}
where $k,k'>0$ depend only on $d,k_1,k_2$.
Thus the left inequality in \eqref{MAINTHM2REDLEMRES} holds.

A variant of this construction also works if $\frac15\leq t<1$.
Namely, with $B$ and $\vecw$ as in \eqref{CUTBALLRECDEF},
let $\veca$ be a non-zero vector with
$\varphi(\veca,\vecw)=\arcsin(1-t)$;
if $\vecp\notin\R\vecw$ then we furthermore assume
$\veca\in\R\vecw+\R\vecp$ and that $\veca$ and $\vecp$ lie on
distinct sides of the line $\R\vecw$ in the plane $\R\vecw+\R\vecp$.
Let $\fZ$ be the infinite cylinder consisting of all points which
have distance $\leq r$ to the line $\vecp+\R\veca$
(then clearly $B\subset\fZ$), and set
\begin{align*}
\fC':=\fZ\cap\{\vecx\in\R^d\col 0\leq\vecw\cdot\vecx\leq(2-t)r\}.
\end{align*}
Then $\fC\subset\fC'$, 
$\frac12\leq|\fC|\leq|\fC'|\ll|\fC|$,
and there is some $M\in G$ such that
$\fC' M$ is a cylinder as in Corollary \ref{CONECYLCOR}, 
with edge ratio $\asymp e$;
thus the left inequality in \eqref{MAINTHM2REDLEMRES} again holds.

The right inequality in \eqref{MAINTHM2REDLEMRES} is proved in a similar
but easier way by inscribing a cone in $\fC$ and using
Lemma \ref{MAINTHM2REDAUXLEM}.
\end{proof}

\subsection{Some more lemmas}
Let %
\begin{align}\label{CUTBALLPRES}
\fC:=%
B\cap\{\vecx\in\R^d\col \vecw\cdot\vecx>0\}
\end{align}
be a %
cut ball of radius $r>0$, %
cut ratio $t\in[0,2)$ and edge ratio $e\in[0,1]$.
By Lemma~\ref{MAINTHM2REDLEM}, if we can prove Theorem \ref{MAINTHM2}
whenever $t\in[0,\frac15]$ then Theorem \ref{MAINTHM2} holds in general.
We may furthermore assume $\bn\notin\partial B$,
since then the remaining cases with $\bn\in\partial B$ follow 
by a limit argument, using Lemma \ref{PCONTLEM}.
\textit{Hence from now on we will assume
$0<t\leq\frac15$ and $0<e\leq1$.}

In order to prove Theorem~\ref{MAINTHM2} we will use the bounds in
Propositions \ref{GENCONVBOUNDPROP1} and \ref{GENCONVBOUNDPROP2}.
Just as in the proof of Theorem~\ref{GENBALLTHM}
it is clear from start that these give sharp bounds on $p(\fC)$,
since the inradius and the circumradius of $\fC$ are comparable
because of our assumption $t\leq\frac15$.
Furthermore, the intersection of $\fC$ with a hyperplane 
$\vecv^\perp$ ($\vecv\neq\pm\vece_1$)
is again a cut ball (of dimension $d-1$);
hence we will again be able to use induction
to bound the integrands appearing in 
Propositions \ref{GENCONVBOUNDPROP1} and \ref{GENCONVBOUNDPROP2}.
However first we need to make a careful study of 
how the size and shape of this cut ball $\fC\cap\vecv^\perp$
varies with $\vecv$.

Let $\vecp$ be the center of the ball $B$ in \eqref{CUTBALLPRES}.
After a rotation we may assume 
\begin{align}\label{CUTBALLFIXEDCHOICE}
\vecw=\vece_1\quad\text{and}\quad
\vecp=r(1-t)\vece_1+rs\vece_2,\quad
\text{where }\:s=(1-e)\sqrt{t(2-t)}.
\end{align}
We write $r'=r\sqrt{t(2-t)}$ and $\vecq=rs\vece_2$ 
for the radius and the center of 
the $(d-1)$-dimensional ball $B\cap\vecw^\perp$.

In the integrals in 
Propositions \ref{GENCONVBOUNDPROP1} and \ref{GENCONVBOUNDPROP2},
we note that we may restrict to $\vecv\in(v_1,\ldots,v_d)\in\S_1^{d-1}$
satisfying $v_2>0$, since $(-\vecv)^\perp=\vecv^\perp$.
If $d\geq3$ then 
we parametrize a dense open subset of $\S_1^{d-1}\cap\{v_2>0\}$ as follows:
\begin{align}\notag
\vecv&=(v_1,\ldots,v_d)
\\\label{VPARA}
&=\bigl(\cos\varpi,\sin\varpi\cos\omega,
(\sin\varpi\sin\omega)\alpha_1,(\sin\varpi\sin\omega)\alpha_2,
\ldots,(\sin\varpi\sin\omega)\alpha_{d-2}\bigr)
\in\S_1^{d-1},
\end{align}
where $\varpi\in(0,\pi)$, $\omega\in(0,\frac\pi2)$ and
$\vecalf=(\alpha_1,\ldots,\alpha_{d-2})\in\S_1^{d-3}$.
Thus $\varpi$ is the angle between $\vecv$ and $\vece_1$,
and $\omega$ is the angle between $\vecv':=(v_2,\ldots,v_d)$
and $\vece_1$ in $\R^{d-1}$.
Note in particular that $\vecv\neq\pm\vece_1$, since $\varpi\in(0,\pi)$.
The $(d-1)$-dimensional volume measure on $\S_1^{d-1}$ takes the following
form in our parametrization:
\begin{align}\label{DVINPARA}
d\vecv=(\sin\varpi)^{d-2}(\sin\omega)^{d-3}\,d\varpi\,d\omega\,d\vecalf,
\end{align}
where $d\vecalf$ is the $(d-3)$-dimensional volume measure on $\S_1^{d-3}$
(if $d=3$: $d\vecalf$ is the counting measure on $\S_1^0=\{-1,1\}$).

As noted above, 
for any $\vecv\in\S_1^{d-1}\setminus\{\pm\vece_1\}$
the intersection $\fC\cap\vecv^\perp$
is a $(d-1)$-dimensional cut ball inside $\vecv^\perp$.
To express the radius and cut ratio of $\fC\cap\vecv^\perp$
in a convenient way we
introduce the following functions,
for $0<t\leq\frac15$,
$u\in\bigl[0,\sqrt{t(2-t)}\bigr)$ and $\varpi\in[0,\pi]$,
\begin{align}
&h(u,\varpi)=(1-t)\sin\varpi-u\cos\varpi;
\\\label{GDEF}
&g(u,\varpi)=\sqrt{t(2-t)-u^2+h(u,\varpi)^2};
\\\label{FDEF}
&f(u,\varpi)=g(u,\varpi)+h(u,\varpi)
=\frac{t(2-t)-u^2}{g(u,\varpi)-h(u,\varpi)};
\\\label{TDEF}
&T(u,\varpi)=\frac{g(u,\varpi)-h(u,\varpi)}{g(u,\varpi)}
=\frac{t(2-t)-u^2}{g(u,\varpi)f(u,\varpi)}.
\end{align}
(All four functions depend on $t$ but we leave this out from the
notation.)
We compute that, for any $\vecv$ with
$\varpi\in(0,\pi)$ and $\omega\in(0,\frac\pi2)$,
\textit{the $(d-1)$-dimensional cut ball $\fC\cap\vecv^\perp$ has radius
$rg(s\cos\omega,\varpi)$ and cut ratio $T(s\cos\omega,\varpi)$.}

In the remaining case $d=2$
we parametrize $\vecv\in\S_1^1$ as
$\vecv=(\cos\varpi,\sin\varpi)$, where we may restrict to $\varpi\in(0,\pi)$.
We then compute that
\textit{$\fC\cap\vecv^\perp$ is a line segment of length $rf(s,\varpi)$;}
in fact this line segment may be viewed as a 1-dimensional cut ball
of radius $rg(s,\varpi)$ and cut ratio $T(s,\varpi)$.

In the next four lemmas we give some
information on the size of the functions $h,g,f,T$.
Set
\begin{align*}
\varpi_0(u)=\arctan\Bigl(\frac{u}{1-t}\Bigr)\in[0,\sfrac\pi2).
\end{align*}

\begin{lem}\label{LEM5}
For any $u\in\bigl[0,\sqrt{t(2-t)}\bigr)$
and $\varpi\in(\frac\pi2,\pi]$ we have
\begin{align*}
h(u,\varpi)=\bigl|h(u,\varpi)\bigr|\geq\bigl|h(u,\pi-\varpi)\bigr|,
\quad
g(u,\varpi)\geq g(u,\pi-\varpi),
\:\text{ and }\:
f(u,\varpi)\geq f(u,\pi-\varpi).
\end{align*}
\end{lem}
\begin{proof}
Clear by inspection.
\end{proof}

\begin{lem}\label{LEM1}
For any fixed $u\in\bigl[0,\sqrt{t(2-t)}\bigr)$,
$h(u,\varpi)$ is a strictly increasing function of $\varpi\in[0,\frac\pi2]$
with $h(u,\varpi_0(u))=0$;
also $f(u,\varpi)$ is a strictly increasing function of 
$\varpi\in[0,\frac\pi2]$,
and $T(u,\varpi)$ is a strictly decreasing function of
$\varpi\in[0,\frac\pi2]$ with $T(u,\varpi_0(u))=1$.
\end{lem}
\begin{proof}
The statements about $h$ are clear. The statement about
$f$ follows from this, using the first formula in \eqref{FDEF}
for $\varpi\in[\varpi_0(u),\frac\pi2]$ and the second for
$\varpi\in[0,\varpi_0(u)]$.
Also the statement about $T$ follows, since 
$1-\frac h{\sqrt{t(2-t)-u^2+h^2}}$ is a strictly decreasing 
function of $h\in\R$.
\end{proof}

\begin{lem}\label{LEM3}
For any $u\in\bigl[0,\sqrt{t(2-t)}\bigr)$ and $\varpi\in(0,\frac12\varpi_0(u)]$
we have
\begin{align*}
g(u,\varpi)\asymp\sqrt t
\quad\text{and}\quad 
f(u,\varpi)\asymp\sqrt{t(2-t)}-u.
\end{align*}
\end{lem}
\begin{proof}
For these $u,\varpi$ we have $u\ll -h(u,\varpi)\leq u$, and thus
$g(u,\varpi)\asymp\sqrt t$.
Now the claim about $f$ follows using the second formula in \eqref{FDEF}.
\end{proof}

\begin{lem}\label{LEM4}
For any $u\in\bigl[0,\sqrt{t(2-t)}\bigr)$ with
$2\varpi_0(u)\leq\frac\pi2$ we have
\begin{align*}
g(u,\varpi)\asymp f(u,\varpi)\asymp\sqrt t+\varpi,
\qquad\text{for all }\:\varpi\in[2\varpi_0(u),\sfrac\pi2].
\end{align*}
\end{lem}
\begin{proof}
Assume $2\varpi_0\leq\varpi\leq\frac\pi2$,
where we write $\varpi_0=\varpi_0(u)$.
Then $\sin\varpi\geq(1-2^{-\frac12})\sin\varpi+\sin\varpi_0$,
and thus $h(u,\varpi)\asymp\sin\varpi\asymp\varpi$.
Hence $g(u,\varpi)\asymp\sqrt{t(2-t)-u^2}+\varpi$,
and now either $u\leq\frac12\sqrt{t(2-t)}$ in which case
$g(u,\varpi)\asymp\sqrt{t}+\varpi$ is clear,
or else $u>\frac12\sqrt{t(2-t)}$ in which case
$\varpi\geq2\varpi_0\gg u\gg\sqrt t$ so that
$g(u,\varpi)\asymp\sqrt{t}+\varpi$ still holds.
Now also $f(u,\varpi)\asymp\sqrt t+\varpi$ follows.
\end{proof}

Finally in the next lemma we give some bounds on $|\fC_\vecv|$,
i.e.\ the $(d-1)$-dimensional volume of $\fC\cap\vecv^\perp$
(recall \eqref{CVDEFNEW}), in the case $d\geq3$.
\begin{lem}\label{LEM2}
If $d\geq3$ then for any $\vecv\in\S_1^{d-1}$ with
$\omega\in(0,\frac\pi2)$ and 
$\varpi\in(0,\frac\pi2]$ we have
\begin{align}\label{LEM2RES1}
|\fC_\vecv|\gg r^{d-1}t^{\frac{d-1}2}(e+\omega^2)^{\frac d2}.
\end{align}
If $0<\varpi\leq\frac12\varpi_0(s\cos\omega)$ then this bound is sharp, i.e.
\begin{align}\label{LEM2RES2}
|\fC_\vecv|\asymp r^{d-1}t^{\frac{d-1}2}(e+\omega^2)^{\frac d2}.
\end{align}
Finally if $2\varpi_0(s\cos\omega)\leq\varpi\leq\frac\pi2$ then
\begin{align}\label{LEM2RES3}
|\fC_\vecv|\asymp r^{d-1}(\sqrt t+\varpi)^{d-1}.
\end{align}
\end{lem}
\begin{proof}
Recall that $\fC_\vecv$ is a $(d-1)$-dimensional cut ball of radius
$rg(s\cos\omega,\varpi)$ and cut ratio $T(s\cos\omega,\varpi)$.
If $\varpi\leq\varpi_0(s\cos\omega)$ then $T(s\cos\omega,\varpi)\geq1$
(cf.\ Lemma \ref{LEM1}) and hence, writing $f=f(s\cos\omega,\varpi)$,
$g=g(s\cos\omega,\varpi)$ and $T=T(s\cos\omega,\varpi)$,
\begin{align}\notag
|\fC_\vecv|\asymp(2-T)^{\frac d2}(rg)^{d-1}
=\Bigl(\frac fg\Bigr)^{\frac d2}r^{d-1}g^{d-1}
&=r^{d-1}(fg)^{\frac d2-1}f
\\\label{LEM2PF1}
&=r^{d-1}\bigl(t(2-t)-s^2\cos^2\omega\bigr)^{\frac d2-1}T^{1-\frac d2}f.
\end{align}
If $\varpi\in(0,\frac12\varpi_0(s\cos\omega)]$ then using the next to
last of these expressions together with
Lemma~\ref{LEM3} and the fact that 
\begin{align}\label{LEM3PF2}
\sqrt{t(2-t)}-s\cos\omega
=\sqrt{t(2-t)}\bigl(e+(1-e)(1-\cos\omega)\bigr)
\asymp\sqrt t(e+\omega^2),
\end{align}
we obtain \eqref{LEM2RES2}.
Furthermore the last expression in \eqref{LEM2PF1} 
is an increasing function of $\varpi\in[0,\frac\pi2]$,
since $f$ is increasing and $T$ is decreasing (cf.\ Lemma \ref{LEM1}).
Hence \eqref{LEM2RES1} holds for all $\varpi\in(0,\varpi_0(s\cos\omega)]$.
Finally for $\varpi\in[\varpi_0(s\cos\omega),\frac\pi2]$
we have $T\leq1$ and hence
$|\fC_\vecv|\asymp(rg)^{d-1}$.
But $g=g(s\cos\omega,\varpi)$ is strictly increasing as a function of
$\varpi\in[\varpi_0(s\cos\omega),\frac\pi2]$
(by \eqref{GDEF} and Lemma \ref{LEM1}); thus
$g\geq\sqrt{t(2-t)-s^2\cos^2\omega}
\asymp\sqrt{t(e+\omega^2)}$ (cf.\ \eqref{LEM3PF2})
and $|\fC_\vecv|\gg\bigl(r\sqrt{t(e+\omega^2)}\bigr)^{d-1}$.
Hence \eqref{LEM2RES1} is valid (but crude) also when
$\varpi\in[\varpi_0(s\cos\omega),\frac\pi2]$.
Finally if $2\varpi_0(s\cos\omega)\leq\varpi\leq\frac\pi2$ then
we have $g\asymp\sqrt t+\varpi$ by
Lemma \ref{LEM4}, and hence \eqref{LEM2RES3} holds.
\end{proof}

\subsection{\texorpdfstring{Proof of Theorem \ref*{MAINTHM2}}{Proof of Theorem 1.3}}
\label{PROOFOFTHMMAINTHM2SEC}

We keep the notation from the previous section.
We will start by proving the \textit{right} inequality in \eqref{MAINTHM2RES}.
First assume $d=2$.
Note that the inradius of $\fC$ is $>\frac12r$, since $t\leq\frac15$; 
hence by Proposition~\ref{GENCONVBOUNDPROP1}
(and using $(-\vecv)^\perp=\vecv^\perp$)
there exists an absolute constant $k>0$ such that
\begin{align}\label{BALLMAINTHMD2PF1}
p^{(2)}(\fC)\ll r^{-2}\int_{\substack{\vecv\in\S_1^{1}\\ (v_2>0)}}
I\Bigl(kr|\fC_\vecv|\leq1\Bigr)\,d\vecv
=r^{-2}\int_0^\pi I\Bigl(kr^2 f(s,\varpi)\leq1\Bigr)\,d\varpi.
\end{align}
Using Lemma \ref{LEM5} and Lemma \ref{LEM4}
and the fact that $\varpi_0(s)\ll\sqrt t$,
it follows that there exists an absolute constant $c>0$ such that 
$kr^2f(s,\varpi)>1$ holds
whenever $c(r^{-2}+\sqrt t)\leq\varpi\leq\pi-c(r^{-2}+\sqrt t)$.
Hence
$p^{(2)}(\fC)\ll r^{-2}(r^{-2}+t^{\frac12})
\ll|\fC|^{-2}(1+|\fC|t^{\frac12})$.
Furthermore, by Lemma~\ref{LEM3}, if $er^2t^{\frac12}$ is larger than a 
certain absolute constant then $kr^2f(s,\varpi)>1$ holds for all 
$\varpi\in(0,\frac12\varpi_0(s)]$, and hence by Lemma \ref{LEM5}
and Lemma \ref{LEM1} it actually holds for all $\varpi\in(0,\pi)$;
thus $p^{(2)}(\fC)=0$.
Hence we have proved the right inequality in \eqref{MAINTHM2RES}
for $d=2$.

\vspace{5pt}

Next assume $d\geq3$.
Now by Proposition~\ref{GENCONVBOUNDPROP1}
there exists a constant $k>0$ which only depends on $d$ such that
\begin{align}\label{BALLMAINTHMPF1}
p^{(d)}(\fC)\ll r^{-d}
\int_{\substack{\vecv\in\S_1^{d-1}\\ (v_2>0)}}
p^{(d-1)}\Bigl(kr^{\frac1{d-1}}\fC\cap\vecv^\perp\Bigr)\,d\vecv.
\end{align}
By induction we have, for any $\vecv\in\S_1^{d-1}$
with $\omega\in(0,\frac\pi2)$ and $\varpi\in(0,\pi)$,
\begin{align}\label{BALLMAINTHMPF2}
p^{(d-1)}\Bigl(kr^{\frac1{d-1}}\fC\cap\vecv^\perp\Bigr)
\ll\min\Bigl(1,\bigl(r|\fC_\vecv|\bigr)^{-2}
+\bigl(r|\fC_\vecv|\bigr)^{-2+\frac2{d-1}}
T(s\cos\omega,\varpi)^{\frac{d-2}{d-1}}\Bigr).
\end{align}
If $\varpi\in(0,\frac\pi2]$ then 
we use $T(s\cos\omega,\varpi)\leq2$ and \eqref{LEM2RES1} %
to see that \eqref{BALLMAINTHMPF2} implies
\begin{align}\label{BALLMAINTHMPF4}
p^{(d-1)}\Bigl(kr^{\frac1{d-1}}\fC\cap\vecv^\perp\Bigr)
\ll\min\Bigl(1,\bigl(r^dt^{\frac{d-1}2}\omega^d\bigr)^{-2+\frac2{d-1}}\Bigr).
\end{align}
If furthermore $2\varpi_0(s\cos\omega)\leq\varpi\leq\frac\pi2$ then
we get a stronger bound by using \eqref{LEM2RES3} %
and
\begin{align*}
T(s\cos\omega,\varpi)\asymp\frac{t(2-t)-s^2\cos^2\omega}{t+\varpi^2}
\ll\frac t{t+\varpi^2}
\end{align*}
(cf.\ \eqref{TDEF} and Lemma \ref{LEM4}), %
namely:
\begin{align}\label{BALLMAINTHMPF5}
p^{(d-1)}\Bigl(kr^{\frac1{d-1}}\fC\cap\vecv^\perp\Bigr)
\ll\min\Bigl(1, r^{-2d}(t+\varpi^2)^{1-d}
+r^{-\frac{2d(d-2)}{d-1}}(t+\varpi^2)^{-\frac{d(d-2)}{d-1}}
t^{\frac{d-2}{d-1}}\Bigr).
\end{align}

Finally we consider the case $\varpi\in[\frac\pi2,\pi)$.
For these $\varpi$ we have,
with notation as in the proof of Lemma \ref{LEM2},
$T\leq1$ (since $h(s\cos\omega,\varpi)>0$)
and hence $|\fC_\vecv|\asymp(rg)^{d-1}$.
Thus in view of Lemma \ref{LEM5} we certainly have
$|\fC_\vecv|\gg|\fC_{\tilde\vecv}|$,
where $\tilde{\vecv}$ is the unit vector corresponding to
$\pi-\varpi$ in place of $\varpi$, but the same $\omega$ and $\vecalf$
(cf.\ \eqref{VPARA}).
Also
$T(s\cos\omega,\varpi)\leq T(s\cos\omega,\pi-\varpi)$,
by \eqref{TDEF} and Lemma \ref{LEM5}.
Hence the right hand side of \eqref{BALLMAINTHMPF2} is %
(possibly up to a constant factor which only depends on $d$)
\textit{smaller} for $\vecv$ than for $\tilde\vecv$.
Hence when bounding the integral in \eqref{BALLMAINTHMPF1} using
\eqref{BALLMAINTHMPF2} it suffices to consider $\varpi\in(0,\frac\pi2)$.

Using now \eqref{BALLMAINTHMPF1}, \eqref{BALLMAINTHMPF4}
and \eqref{BALLMAINTHMPF5} we get
\begin{align}\notag
p^{(d)}(\fC)\ll & r^{-d}
\int_0^{\frac\pi2}
\int_0^{2\varpi_0(s\cos\omega)}
\min\Bigl(1,\bigl(r^dt^{\frac{d-1}2}\omega^d\bigr)^{-2+\frac2{d-1}}\Bigr)
\,\varpi^{d-2}\,d\varpi\,\omega^{d-3}\,d\omega
\\\label{BALLMAINTHMPF6}
&+r^{-d}\int_0^{\frac\pi2}
\int_{0}^\infty \min\Bigl(1,r^{-2d}(t+\varpi^2)^{1-d}\Bigr)
\,\varpi^{d-2}\,d\varpi\,\omega^{d-3}\,d\omega
\\\notag
&+r^{-d}\int_0^{\frac\pi2}
\int_{0}^\infty\min\Bigl(1,
r^{-\frac{2d(d-2)}{d-1}}(t+\varpi^2)^{-\frac{d(d-2)}{d-1}}
t^{\frac{d-2}{d-1}}\Bigr)
\,\varpi^{d-2}\,d\varpi\,\omega^{d-3}\,d\omega.
\end{align}
Using $\varpi_0(s\cos\omega)\ll\sqrt t$
we see that the first term is
\begin{align*}
\ll r^{-d}t^{\frac{d-1}2}\int_0^{\frac\pi2}
\min\Bigl(1,\bigl(r^dt^{\frac{d-1}2}\omega^d\bigr)^{-2+\frac2{d-1}}\Bigr)
\,\omega^{d-3}\,d\omega
\asymp\left.\begin{cases}r^{-d}t^{\frac{d-1}2}
&\text{if }\:t\leq r^{-\frac{2d}{d-1}}
\\r^{2-2d}t^{1-\frac1d}&\text{if }\:t\geq r^{-\frac{2d}{d-1}}
\end{cases}\right\}.
\end{align*}
The second term in \eqref{BALLMAINTHMPF6} is
\begin{align*}
\ll r^{-d}\int_0^{\sqrt t}
\min\Bigl(1,r^{-2d}t^{1-d}\Bigr)\,\varpi^{d-2}\,d\varpi
+r^{-d}\int_{\sqrt t}^\infty\min\Bigl(1,r^{-2d}\varpi^{2-2d}\Bigr)\,
\varpi^{d-2}\,d\varpi \hspace{30pt}&
\\
\asymp 
\left.\begin{cases}r^{-d}t^{\frac{d-1}2}+r^{-2d}
&\text{if }\:t\leq r^{-\frac{2d}{d-1}}
\\r^{-3d}t^{\frac{1-d}2}&\text{if }\:t\geq r^{-\frac{2d}{d-1}}
\end{cases}\right\}
\asymp
\left.\begin{cases}r^{-2d}
&\text{if }\:t\leq r^{-\frac{2d}{d-1}}
\\r^{-3d}t^{\frac{1-d}2}&\text{if }\:t\geq r^{-\frac{2d}{d-1}}
\end{cases}\right\}.
\end{align*}
The third term in \eqref{BALLMAINTHMPF6} is
\begin{align*}
\ll r^{-d}\int_0^{\sqrt t}
\min\Bigl(1,r^{-\frac{2d(d-2)}{d-1}}t^{2-d}\Bigr)\,\varpi^{d-2}\,d\varpi
+r^{-d}\int_{\sqrt t}^\infty
\min\Bigl(1,r^{-\frac{2d(d-2)}{d-1}}
\varpi^{-\frac{2d(d-2)}{d-1}}t^{\frac{d-2}{d-1}}\Bigr)\,\varpi^{d-2}\,d\varpi
\\
\asymp\left.\begin{cases}r^{-d}t^{\frac{d-1}2}+r^{1-2d}t^{\frac{d-1}{2d}}
&\text{if }\:t\leq r^{-\frac{2d}{d-1}}
\\r^{-3d+\frac{2d}{d-1}}t^{\frac{3-d}2}&\text{if }\:t\geq r^{-\frac{2d}{d-1}}
\end{cases}\right\}
\asymp\left.\begin{cases}r^{1-2d}t^{\frac{d-1}{2d}}
&\text{if }\:t\leq r^{-\frac{2d}{d-1}}
\\r^{-3d+\frac{2d}{d-1}}t^{\frac{3-d}2}&\text{if }\:t\geq r^{-\frac{2d}{d-1}}
\end{cases}\right\}.
\end{align*}
Collecting these three bounds we conclude that
$p^{(d)}(\fC)\ll r^{-2d}$ if $t\leq r^{-\frac{2d}{d-1}}$,
and $p^{(d)}(\fC)\ll r^{2-2d}t^{1-\frac1d}$ if $t\geq r^{-\frac{2d}{d-1}}$.
In other words, for general $r$ and $t$:
\begin{align}\label{BALLMAINTHMPF7}
p^{(d)}(\fC)\ll r^{-2d}+r^{2-2d}t^{1-\frac1d}
\asymp
|\fC|^{-2}\Bigl(1+|\fC|^{\frac2d}t^{\frac{d-1}d}\Bigr).
\end{align}

Next we wish to prove that $p^{(d)}(\fC)=0$ holds whenever
$e|\fC|^{\frac2d}t^{1-\frac1d}$ 
is larger than a certain constant which only depends on
$d$.
This can of course be proved using \eqref{BALLMAINTHMPF1} and induction;
however we get a simpler proof by using Minkowski's Theorem
in a similar way as in Remark~\ref{MINKOWSKIREMARK}.
Namely, recalling $r'=r\sqrt{t(2-t)}$, let us note that the box
\begin{align*} 
F=[-r,r]\times\bigl[-\sfrac12er',\sfrac12er'\bigr]\times
\underbrace{\Bigl[-\frac{\sqrt e}{2\sqrt d}r',\frac{\sqrt e}{2\sqrt d}r'\Bigr]
\times\cdots\times
\Bigl[-\frac{\sqrt e}{2\sqrt d}r',\frac{\sqrt e}{2\sqrt d}r'\Bigr]}_{\text{$d-2$ copies}}
\end{align*}
is contained in the union $\fC\cup\vece_1^\perp\cup(-\fC)$.
To verify this, 
it clearly suffices to check that every point
$(x_1,\ldots,x_d)\in F$ with $x_1>0$ lies in $\fC$.
For such a point we have 
\begin{align*}
(x_1-r(1-t))^2+(x_2-(1-e)r')^2+x_3^2+\ldots+x_d^2
<(1-t)^2r^2+(1-\sfrac12e)^2{r'}^2+\sfrac14e{r'}^2&
\\
\leq(1-t)^2r^2+{r'}^2&=r^2,
\end{align*}
and this implies $(x_1,\ldots,x_d)\in\fC$
(cf.\ \eqref{CUTBALLPRES}, \eqref{CUTBALLFIXEDCHOICE}).
Having thus verified
$F\subset\fC\cup\vece_1^\perp\cup(-\fC)$,
the argument in Remark~\ref{MINKOWSKIREMARK} 
(and using Lemma \ref{PCONTLEM} to see $p^{(d)}(\fC)
=p^{(d)}(\fC\cup\vece_1^\perp)$)
now gives that
if $p^{(d)}(\fC)>0$ then $|F|\leq 2^d$.
But $|F|\asymp e^{\frac d2}t^{\frac{d-1}2}r^d
\asymp e^{\frac d2}t^{\frac{d-1}2}|\fC|$.
In other words $p^{(d)}(\fC)=0$ holds whenever
$e|\fC|^{\frac2d}t^{1-\frac1d}$ is sufficiently large.
Taking this fact together with \eqref{BALLMAINTHMPF7}, we have now proved 
that the right inequality in \eqref{MAINTHM2RES} holds for $d\geq3$.

\begin{remark}\label{LASTDISCREM}
The last discussion applies also when $d=2$, i.e.\ it gives an alternative
proof of the fact that
$p^{(2)}(\fC)=0$ whenever $e|\fC|t^{\frac12}$ is sufficiently large.
\end{remark}

Next we prove the left inequality in \eqref{MAINTHM2RES}.
We give the proof for $d\geq3$;
the case $d=2$ can be treated in a similar %
way.

By Proposition~\ref{GENCONVBOUNDPROP2} there is a constant $k>0$
which only depends on $d$ such that
\begin{align}\label{CUTBALLLOWBOUNDPF1}
p^{(d)}(\fC)\gg
\min\biggl\{1,r^{-d}
\int_{\S_1^{d-1}}
p^{(d-1)}\Bigl(kr^{\frac1{d-1}}\fC\cap\vecv^\perp\Bigr)\,d\vecv\biggr\}.
\end{align}

By \eqref{LEM2RES3} %
we have
$|\fC_\vecv|\asymp r^{d-1}(\sqrt t+\varpi)^{d-1}$ whenever
$0<\omega<\frac\pi2$ and 
$2\varpi_0\leq\varpi\leq\frac\pi2$, where $\varpi_0:=\varpi_0(s\cos\omega)$.
Also $r\gg1$ since $|\fC|\geq\frac12$,
and $\varpi_0\ll\sqrt t$ for all $\omega\in(0,\frac\pi2)$.
Hence there exist some constants $c_1,c_2>0$ %
which only depend on $d$ and $k$ such
that $2c_2r^{-\frac d{d-1}}<\frac\pi2$ holds, and also,
if $t\leq c_1r^{-\frac{2d}{d-1}}$, then
$|\fC_\vecv|<(2k^{d-1}r)^{-1}$ holds for %
all $\vecv$ with $\omega\in(0,\frac\pi2)$ and 
$\varpi\in [c_2r^{-\frac d{d-1}},2c_2r^{-\frac d{d-1}}]$.
Note that, by Lemma \ref{TRIVPLOWBOUNDLEM},
$|\fC_\vecv|<(2k^{d-1}r)^{-1}$ implies that
\begin{align*}
p^{(d-1)}\Bigl(kr^{\frac1{d-1}}\fC\cap\vecv^\perp\Bigr)>\sfrac12.
\end{align*}
Hence if $t\leq c_1r^{-\frac{2d}{d-1}}$, then
\begin{align}\label{CUTBALLLOWBOUNDPF4}
p(\fC)\gg
r^{-d}\vol_{\S_1^{d-1}}\Bigl(\Bigl\{\vecv\in\S_1^{d-1}\col
\omega\in(0,\sfrac\pi2),\:
\varpi\in[c_2r^{-\frac d{d-1}},2c_2r^{-\frac d{d-1}}]\Bigr\}\Bigr)
\gg r^{-2d}. %
\end{align}

On the other hand, for all $\varpi\in(0,\frac12\varpi_0]$ we have 
\begin{align*}
|\fC_\vecv|\asymp r^{d-1}t^{\frac{d-1}2}(e+\omega^2)^{\frac d2},
\end{align*}
by \eqref{LEM2RES2}. %
Hence there exist some constants $c_3\in(0,\frac12c_1^{\frac{d-1}d})$ and 
$c_4>0$ which only depend on $d$, $k$, $c_1$,
such that if
$t>c_1r^{-\frac{2d}{d-1}}$ and $e<c_3r^{-2}t^{\frac{1-d}d}$, then
$c_4r^{-1}t^{\frac{1-d}{2d}}<\frac\pi4$, and 
$|\fC_\vecv|<(2k^{d-1}r)^{-1}$ holds for all 
$\vecv$ with $\omega\in(0,c_4r^{-1}t^{\frac{1-d}{2d}})$
and $\varpi\in(0,\frac12\varpi_0]$.
In this situation
we also have
$e<\frac12(c_1r^{-\frac{2d}{d-1}}t^{-1})^{\frac{d-1}d}<\frac12$ 
and $\omega<\frac\pi4$, and thus
$\varpi_0\gg s=(1-e)\sqrt{t(2-t)}\gg\sqrt t$,
i.e.\ $\frac12\varpi_0>c_5\sqrt t$ where $c_5>0$ is an absolute constant.
Hence, arguing as before,
the following holds whenever $t>c_1r^{-\frac{2d}{d-1}}$
and $e<c_3r^{-2}t^{\frac{1-d}d}$:
\begin{align}\label{CUTBALLLOWBOUNDPF5}
p(\fC)\gg r^{-d}\vol_{\S_1^{d-1}}\Bigl(\Bigl\{
\vecv\in\S_1^{d-1}\col
\omega\in(0,c_4r^{-1}t^{\frac{1-d}{2d}}),\:
\varpi\in(0,c_5\sqrt t)\Bigr\}\Bigr)
\gg r^{2-2d}t^{\frac{d-1}d}.
\end{align}
Taking \eqref{CUTBALLLOWBOUNDPF4} and \eqref{CUTBALLLOWBOUNDPF5} together,
we conclude that $p(\fC)\gg|\fC|^{-2}(1+|\fC|^{\frac2d}t^{\frac{d-1}d})$
holds whenever $e<c_3r^{-2}t^{\frac{1-d}d}$.
In other words, we have proved the left inequality in \eqref{MAINTHM2RES}.

This concludes the proof of Theorem \ref{MAINTHM2}.
\hfill$\square$ $\square$ $\square$

\subsection{\texorpdfstring{Proof of Corollary \ref*{CONECYLCOR} and Corollary \ref*{DCBCOR}}{Proof of Corollary 1.4 and Corollary 1.5}}

\begin{proof}[Proof of Corollary \ref{CONECYLCOR}]
This follows from Theorem \ref{MAINTHM2} and 
Lemma \ref{MAINTHM2REDAUXLEM}.
\end{proof}

\begin{proof}[Proof of Corollary \ref{DCBCOR}]
Also this result will be derived from Theorem \ref{MAINTHM2} by a 
containment argument, i.e.\ using the fact that
$p(\fC)\leq p(\fC')$ whenever $\fC'\subset\fC$.

As before we may assume $\bn\notin\partial B$, thus $\mytA>0$.
If $\mytA\geq\frac15$ then 
\eqref{DCBCORRES} follows from Corollary \ref{CONECYLCOR},
by a similar argument as in the proof of Lemma \ref{MAINTHM2REDLEM}.
Hence we may assume $0<\mytA<\frac15$.
After a rotation we may write
\begin{align*}
\fC=\Bigl\{(x_1,\ldots,x_d)\in\R^d\col 0<x_1<(\mytB-\mytA)r,\:
\hspace{200pt}
\\
(x_1-(1-\mytA)r)^2+(x_2-(1-e)r')^2+x_3^2+\ldots+x_d^2<r^2\Bigr\},
\end{align*}
where $r'=r\sqrt{\mytA(2-\mytA)}$, and $0<e\leq1$.
Note that $|\fC|\asymp (\mytB-\mytA)\mytB^{\frac{d-1}2}r^d$.
Now let $E$ be the ellipsoid
\begin{align}\label{BALLMAINTHMGENPF1}
E=\Bigl\{(x_1,\ldots,x_d)\in\R^d\col
\alpha(x_1-\beta r)^2+(x_2-(1-e)r')^2+x_3^2+\ldots+x_d^2
<\gamma r^2\Bigr\}
\end{align}
where
\begin{align*}
\alpha=\frac{\mytA^2-2\mytA\mytB+2\mytB}{(\mytB-\mytA)^2};
\qquad
\beta=\frac{(1-\mytA)(\mytB-\mytA)^2}{\mytA^2-2\mytA\mytB+2\mytB};
\qquad
\gamma=\frac{(\mytA+\mytB-\mytA\mytB)^2}{\mytA^2-2\mytA\mytB+2\mytB}.
\end{align*}
Noticing $\mytA^2-2\mytA\mytB+2\mytB=(\mytB-\mytA)^2+\mytB(2-\mytB)>0$ 
and $\mytA+\mytB-\mytA\mytB=1-(1-\mytA)(1-\mytB)>0$
we see that $\alpha,\beta,\gamma$ are well-defined and positive.

We claim that $E\subset B$. A sufficient condition for this is clearly
\begin{align*}
\gamma r^2-\alpha(x_1-\beta r)^2\leq r^2-(x_1-(1-\mytA)r)^2,
\qquad\forall x_1\in\R.
\end{align*}
However noticing that 
$\alpha\beta=1-\mytA$ and $\gamma-\alpha\beta^2=1-(1-\mytA)^2$ we see that
the above inequality simplifies to
$(\alpha-1)x_1^2\geq0$, and this is true for all $x_1\in\R$ since
$\alpha-1=\frac{\mytB(2-\mytB)}{(\mytB-\mytA)^2}\geq0$.
Hence we indeed have $E\subset B$.
Note also that $E\subset\{x_1<(\mytB-\mytA)r\}$, since
$\beta+\sqrt{\gamma/\alpha}=\mytB-\mytA$.
Hence $\fC'\subset\fC$, where
\begin{align*}
\fC':=E\cap\{x_1>0\}. %
\end{align*}
Set 
\begin{align}\label{BALLMAINTHMGENPF2}
M_0:=\text{diag}[\alpha^{\frac{d-1}{2d}},\alpha^{-\frac1{2d}}
,\alpha^{-\frac1{2d}},\ldots,\alpha^{-\frac1{2d}}]\in G.
\end{align}
Then $\fC'M_0$ is a cut ball of radius
$\alpha^{-\frac1{2d}}\gamma^{\frac12}r$
and cut ratio 
\begin{align*}
\tau=1-\frac{\beta}{\sqrt{\gamma/\alpha}}
=\frac{\mytA(2-\mytA)}{\mytA+(1-\mytA)\mytB}\asymp\frac{\mytA}{\mytB}.
\end{align*}
Also its edge ratio is $e$ %
since $B\cap\vece_1^\perp=E\cap\vece_1^\perp$.
Since $\tau<1$ we have
\begin{align*}
|\fC'|=|\fC'M_0|\asymp\alpha^{-\frac12}\gamma^{\frac d2}r^d
\asymp (\mytB-\mytA)\mytB^{\frac{d-1}2}r^d,
\end{align*}
where the last relation follows since
$\alpha\asymp \mytB(\mytB-\mytA)^{-2}$ and $\gamma\asymp \mytB$.
Thus $|\fC'|\asymp |\fC|$.
Now by Theorem \ref{MAINTHM2} we have
\begin{align*}
p(\fC)\leq p(\fC')=p(\fC'M_0)
\leq F_{\substack{cut\\ball}}^{(d)}\Bigl(e;\tau;k|\fC'|\Bigr)
\leq F_{\substack{cut\\ball}}^{(d)}\Bigl(e;\frac\mytA\mytB;k'|\fC|\Bigr),
\end{align*}
for some constants $k,k'>0$ which only depend on $d$.
Hence the right inequality in \eqref{DCBCORRES} holds.

We now turn to proving the left inequality in \eqref{DCBCORRES}.
If $1\leq \mytB\leq2$ then the desired inequality follows directly from
the left inequality in \eqref{MAINTHM2RES} in Theorem \ref{MAINTHM2},
applied to the cut ball 
$\fC'=\{\vecx\in B\col\vecw\cdot\vecx>0\}$, since $\fC\subset\fC'$,
$|\fC|\asymp|\fC'|$ and $\mytA\asymp \mytA/\mytB$ in this case.
Hence from now on we may assume $\mytB<1$
(and $0<\mytA<\frac15$ as before).
Furthermore we may assume $\mytB>2\mytA$, since otherwise
the desired statement again follows from an application of
Corollary \ref{CONECYLCOR} similar to what we did 
in the proof of Lemma \ref{MAINTHM2REDLEM}.
Now let $E$ be the ellipsoid \eqref{BALLMAINTHMGENPF1}, but with
\begin{align*}
\alpha=\frac{2-\mytA-\mytB}{\mytB-\mytA};\qquad\beta=\mytB-\mytA;\qquad\gamma=\mytB(2-\mytB).
\end{align*}
We then claim $\fC\subset E$. A sufficient condition for this is clearly
\begin{align*}
\gamma r^2-\alpha(x_1-\beta r)^2\geq r^2-(x_1-(1-\mytA)r)^2,
\qquad\forall x_1\in(0,(\mytB-\mytA)r).
\end{align*}
But the inequality is equivalent with
$(1-\mytB)x_1(\frac{x_1}{\mytB-\mytA}-r)\leq0$, which is indeed true for all
$x_1\in(0,(\mytB-\mytA)r)$, since $1-\mytB>0$. Hence $\fC\subset E$, and thus also
\begin{align*}
\fC\subset\fC':=E\cap\{x_1>0\}.
\end{align*}
Take $M_0$ as in \eqref{BALLMAINTHMGENPF2}, but with our new $\alpha$.
Then $\fC'M_0$ is again a cut ball of radius
$\alpha^{-\frac1{2d}}\gamma^{\frac12}r$,
cut ratio $\tau=1-\frac{\beta}{\sqrt{\gamma/\alpha}}
=1-\sqrt{\frac{(\mytB-\mytA)(2-\mytA-\mytB)}{\mytB(2-\mytB)}}
\asymp1-\frac{(\mytB-\mytA)(2-\mytA-\mytB)}{\mytB(2-\mytB)}
=\frac{\mytA(2-\mytA)}{\mytB(2-\mytB)}\asymp\frac{\mytA}{\mytB}$,
and edge ratio $e$.
As before we find 
$|\fC'M_0|\asymp|\fC|$
(for this we use $\mytB>2\mytA$ to see that
$\alpha\asymp \mytB^{-1}$ and $\gamma\asymp \mytB$), and 
now the left inequality in \eqref{DCBCORRES}
follows from Theorem \ref{MAINTHM2} applied to $\fC'M_0$,
since $p(\fC)\geq p(\fC')=p(\fC'M_0)$.
\end{proof}

\section{\texorpdfstring{Basic facts about the conditional probabilities $\nu_\vecp$ and $p_\vecp$}{Basic facts about the conditional probabilities * and *}}
\label{X1RECOLLECTSEC}

\subsection{Definitions and basic properties}

We start by recollecting some definitions from \cite{partI}.
Recall that we have identified $X_1$ with the homogeneous space
$\Gamma\backslash G$ (where $\Gamma=\SL_d(\Z)$ and $G=\SL_d(\R)$).
Given any $\vecp\in\R^d\setminus\{\bn\}$ we define $X_1(\vecp)$ to be 
the subset
\begin{align*}
X_1(\vecp)=\bigl\{M\in X_1\col \vecp\in\Z^dM\bigr\}.
\end{align*}
For any $\veck\in\Z^d\setminus\{\bn\}$ we set
\begin{align*}
X_1(\veck,\vecp)=\Gamma\backslash\Gamma G_{\veck,\vecp},
\end{align*}
where
\begin{align*}
G_{\veck,\vecp}:=\bigl\{M\in G\col\veck M=\vecp\bigr\}.
\end{align*}
Then $X_1(\veck,\vecp)$ is a connected embedded submanifold of $X_1$
of codimension $d$,
and $X_1(\vecp)$ can be expressed as a countable disjoint union,
\begin{align}\label{X1YDISJUNION}
X_1(\vecp)=\bigsqcup_{k=1}^\infty X_1(k\vece_1,\vecp).
\end{align}
(Cf.\ \cite[Sec.\ 7.1]{partI} for details.)

For every $\vecp\in\R^d\setminus\{\bn\}$ we fix some 
$M_\vecp\in G$ such that $\vecp=\vece_1M_\vecp$.
Then
\begin{align*}
G_{\veck,\vecp}=M_\veck^{-1}HM_\vecp,
\end{align*}
where $H$ is the closed subgroup of $G$ given by
\begin{align}\label{HAV}
H=\bigl\{g\in G \col \vece_1 g =\vece_1\bigr\}
=\Bigl\{\begin{pmatrix} 1 & \bn \\ \trans\vecv & A\end{pmatrix} \col
\vecv\in\R^{d-1}, \: A\in G^{(d-1)} \Bigr\},
\end{align}
Let $\mu_H$ be the Haar measure on $H$ given by
$d\mu_H(g)=d\mu^{(d-1)}(A)\,d\vecv$, with $A,\vecv$ as in \eqref{HAV},
and let $\nu_\vecp$ be the Borel measure on $G_{\veck,\vecp}$
which corresponds to $\zeta(d)^{-1}\mu_H$ under the diffeomorphism
$h\mapsto M_\veck^{-1}hM_\vecp$ from $H$ onto $G_{\veck,\vecp}$.
This measure is independent of the choices of $M_\veck$ and $M_\vecp$.
Finally we use the same notation $\nu_\vecp$ also for the measure on
$X_1(\veck,\vecp)$ corresponding to $\nu_\vecp$ via the covering map
$G_{\veck,\vecp}\mapsto X_1(\veck,\vecp)$,
and also for the measure on $X_1(\vecp)$ obtained via \eqref{X1YDISJUNION}.
(Again cf.\ \cite[Sec.\ 7.1]{partI} for details.)

The measure $\nu_\vecp$ on $X_1(\vecp)$ is a probability
measure (\cite[Prop.\ 7.5]{partI}).
In fact,
as the following proposition shows,
$\nu_\vecp$ may be viewed as determining the \textit{conditional}
distribution of a ($\mu-$)random lattice $L\subset\R^d$ of covolume one,
given that $\vecp\in L$.

\begin{prop}\label{CONDPROP}
Let $\scrE\subset X_1$ be any Borel set.
Then $\vecp\mapsto\nu_\vecp(\scrE\cap X_1(\vecp))$ is a measurable
function of $\vecp\in\R^d\setminus\{\bn\}$, and
for any (Lebesgue) measurable set $U\subset\R^d\setminus\{\bn\}$ we have
\begin{align}\label{CONDPROPPRES}
\int_\scrE\#(L\cap{U})\,d\mu(L)
=\int_{U}\nu_\vecp(\scrE\cap X_1(\vecp))\,d\vecp.
\end{align}
\end{prop}
\begin{proof}
The measurability of the map $\vecp\mapsto\nu_\vecp(\scrE\cap X_1(\vecp))$
was proved in \cite[Prop.\ 7.3]{partI}, and it was seen in the proof of the
same proposition that
\begin{align}\label{CONDPROPPF1}
\int_{U}\nu_\vecp(\scrE\cap X_1(\vecp))\,d\vecp
=\sum_{\veck\in\Z^d\setminus\{\bn\}}\mu(\scrF\cap\scrE_\veck),
\end{align}
where $\scrF\subset G$ is a fixed (measurable) fundamental domain for
$\Gamma\backslash G$ and $\scrE_\veck$ is the set of all
$M\in G$ satisfying both $\Gamma M\in\scrE$ and $\veck M\in{U}$.
Using 
$\mu(\scrF\cap\scrE_\veck)=\int_\scrF I(\Gamma M\in\scrE,\:\veck M\in U)
\,d\mu(M)$
in the right hand side of \eqref{CONDPROPPF1} 
and then changing order of summation and integration,
we obtain \eqref{CONDPROPPRES}.
\end{proof}
\begin{remark}\label{PALMREMARK}
For a general point process $\xi$, %
the precise concept which %
corresponds to the intuitive notion of conditioning on 
$\xi$ having a point at a given position,
is the
\textit{Palm distribution} 
(cf., e.g., \cite{Kallenberg1984}, \cite{Kallenberg}, \cite{Daleyverejones}).
In the special case of the point process $L$
(a $\mu$-random lattice $\subset\R^d$),
it follows from Proposition \ref{CONDPROP} that
the measures $\nu_\vecy$ for $\vecy\in\R^d\setminus\{\bn\}$,
together with $\mu$ itself for $\vecy=\bn$,
give a version of the local Palm distributions.  %
Note in this connection that,
by Siegel's formula (cf.\ \cite{siegel45}, \cite{siegel}),
the intensity $\mathbb EL$ of the point process $L$
equals the sum of the standard Lebesgue measure on $\R^d$
and the Dirac measure $\delta_\bn$ assigning unit mass to $\bn$.
\end{remark}

\vspace{5pt}

Having defined the conditional measure $\nu_\vecp$,
we now define $p_\vecp(\fC)$:
For any $\vecp\in\R^d\setminus\{\bn\}$
and any measurable set $\fC\subset\R^d$ we set
\begin{align*}
p_\vecp(\fC):=\nu_\vecp
\bigl(\bigl\{M\in X_1(\vecp)\col \Z^dM\cap\fC\setminus\{\bn,\vecp\}=\emptyset
\bigr\}\bigr).
\end{align*}
We record that we have the natural invariance relation
\begin{align}\label{PPDINV}
p_\vecp(\fC)=p_{\vecp M}(\fC M), \qquad\forall M\in G,
\end{align}
due to a similar relation for $\nu_\vecp$,
cf.\ \cite[Lemma 7.2]{partI}.
We will furthermore have use for the symmetry relation
\begin{align}\label{PPFSYMM}
p_\vecp(\fC)=p_\vecp(\vecp-\fC),
\end{align}
which holds since $L=\vecp-L$ for all $L\in X_1(\vecp)$.

\subsection{\texorpdfstring{A useful parametrization of $X_1(\vecp)$}{A useful parametrization of X1(p)}}

We will now 
give a parametrization of $X_1(\vecp)$
which allows us to use the (Siegel-)reduced lattice bases 
introduced in Section \ref{SIEGELSEC}
when bounding $p_\vecp(\fC)$.
We will assume $d\geq3$ for the remainder of this section.\footnote{Lemma \ref{GKYMEASLEM1} and its proof is in fact valid also for $d=2$, but we will not make direct use of this fact.}
We start by parametrizing $G_{\veck,\vecp}$, where 
$\veck=(k_1,\ldots,k_d)\in\Z^d\setminus\{\bn\}$ and
$\vecp\in\R^d\setminus\{\bn\}$ are fixed and arbitrary.
Let us write $\veck'=(k_2,\ldots,k_d)\in\Z^{d-1}$.
Then by \eqref{LATTICEINPARAM} we have
$[a_1,\vecv,\vecu,\tM]\in G_{\veck,\vecp}$ if and only if
\begin{align}\label{VECKM}
\vecp=\veck [a_1,\vecv,\vecu,\tM]=k_1a_1\vecv
+\iota\Bigl(k_1a_1^{-\frac1{d-1}}\vecu \aa(\ta)\tkk
+a_1^{-\frac1{d-1}}\veck'\tM\Bigr)f(\vecv),
\end{align}
where $\ta$ and $\tkk$ (and $\tu$ in \eqref{UFORMULAIFK1NEQ0} below)
are defined through the Iwasawa decomposition of $\tM$, cf.\ \eqref{TMDEF},
and $\iota$ denotes the embedding $\iota:\R^{d-1}\ni (x_1,\ldots,x_{d-1})
\mapsto (0,x_1,\ldots,x_{d-1})\in\R^d$.

Let us first assume $k_1\neq0$.
Now \eqref{VECKM} implies
$\vecp\cdot\vecv=k_1a_1$, and thus $\vecv$ must lie in the open half sphere
\begin{align*}
\S^{d-1}_{k_1\vecp+}:=\S_1^{d-1}\cap\R^d_{k_1\vecp+},
\quad\text{where }\:
\R_{\vecw+}^d:=\bigl\{\vecx\in\R^d\col\vecw\cdot\vecx>0\bigr\}.
\end{align*}
Conversely, for any $\tM\in{G^{(d-1)}}$ and any 
$\vecv\in\S^{d-1}_{k_1\vecp+}$ there is a unique choice of
$a_1>0$ and $\vecu\in\R^{d-1}$ such that
$[a_1,\vecv,\vecu,\tM]\in G_{\veck,\vecp}$, namely:
\begin{align}\label{UFORMULAIFK1NEQ0}  
a_1=k_1^{-1}(\vecp\cdot\vecv);\qquad
\vecu
&=k_1^{-1}a_1^{\frac1{d-1}}
\iota^{-1}\bigl((\vecp-k_1a_1\vecv)f(\vecv)^{-1}\bigr)
\tkk^{-1}\aa(\ta)^{-1}-k_1^{-1}\veck'\nn(\tu).
\end{align}
Let us write $[\vecv,\tM]_{\veck,\vecp}$
for this element
$[a_1,\vecv,\vecu,\tM]\in G_{\veck,\vecp}$.
We have thus exhibited a bijective map
\begin{align*}
\S^{d-1}_{k_1\vecp+}\times{G^{(d-1)}}\ni\langle\vecv,\tM\rangle
\mapsto[\vecv,\tM]_{\veck,\vecp}\in G_{\veck,\vecp}.
\end{align*}
Note that this map depends on our fixed choice of $f:\S_1^{d-1}\to\SO(d)$
made in Section \ref{SIEGELSEC}.

We now express the measure $\nu_\vecp$ in the parameters $\vecv,\tM$:
\begin{lem}\label{GKYMEASLEM1}
Given $\veck\in\Z^d$ with $k_1\neq0$ and $\vecp\in\R^d\setminus\{\bn\}$,
the measure $\nu_\vecp$ on $G_{\veck,\vecp}$ takes the following form
in the $[\vecv,\tM]_{\veck,\vecp}$-parametrization:
\begin{align}\label{GKYMEASLEM1RES}
d\nu_\vecp %
=\zeta(d)^{-1}|\vecp\cdot\vecv|^{-d}\,d\mu^{(d-1)}(\tM)\,d\vecv.
\end{align}
\end{lem}
\begin{proof}
Recall that $\nu_\vecp$ %
is the Borel measure which corresponds
to $\zeta(d)^{-1}\mu_H$ under the diffeomorphism
$h\mapsto M_\veck^{-1}hM_\vecp$ from $H$ onto $G_{\veck,\vecp}$,
and this measure is independent of the choices of $M_\veck$ and $M_\vecp$.
Take $\vecv_0\in\S_1^{d-1}$ so that $f:\S_1^{d-1}\to\SO(d)$ is smooth on
$\S_1^{d-1}\setminus\{\vecv_0\}$.
Then $[\vecv,\tM]_{\veck,\vecp}$ depends smoothly on
$\vecv,\tM,\vecp$ in the set
\begin{align*}
S=\bigl\{\langle \vecv,\tM,\vecp\rangle\in
\S_1^{d-1}\times{G^{(d-1)}}\times(\R^d\setminus\{\bn\})\col
k_1\vecp\cdot\vecv>0,\:
\vecv\neq\vecv_0\bigr\}.
\end{align*}
Given any point $\vecp_0\in\R^d\setminus\{\bn\}$ we 
may assume $M_\vecp$ to be chosen so as to depend smoothly on $\vecp$
in a neighborhood of $\vecp_0$;
then also $M_\veck[\vecv,\tM]_{\veck,\vecp}M_\vecp^{-1}\in H$
depends smoothly on $\vecv,\tM,\vecp$ when
$\langle\vecv,\tM,\vecp\rangle\in S$ and $\vecp$ is near $\vecp_0$.
Hence, since $\vecp_0$ is arbitrary, it follows that there is a 
smooth function $\alpha: S\to\R_{\geq0}$ such that
$d\nu_\vecp=\alpha(\vecv,\tM,\vecp)\,d\mu^{(d-1)}(\tM)\,d\vecv$
on all of $S$.
Hence to prove \eqref{GKYMEASLEM1RES} it suffices to prove that
\begin{align}\label{GKYMEASLEM1PF1}
\int_{\R^d\setminus\{\bn\}}
\biggl(\int_{G_{\veck,\vecp}}\rho(M)\,d\nu_\vecp(M)\biggr)\,d\vecp
=\zeta(d)^{-1}
\int_S
\rho\bigl([\vecv,\tM]_{\veck,\vecp}\bigr)
|\vecp\cdot\vecv|^{-d}\,d\mu^{(d-1)}(\tM)\,d\vecv\,d\vecp,
\end{align}
holds for any $\rho\in C_c(G)$
(the space of continuous functions on $G=\SL_d(\R)$ with compact support).

But the left hand side of \eqref{GKYMEASLEM1PF1} equals
$\int_G\rho(M)\,d\mu(M)$, by 
the definition of $\nu_\vecp$ and \cite[Lemma 7.4]{partI},
and using \eqref{SLDRSPLITHAAR} this is
\begin{align*}
=\zeta(d)^{-1}\int_{\S_1^{d-1}}\int_{{G^{(d-1)}}}
\int_0^\infty\int_{\R^{d-1}}
\rho\bigl([a_1,\vecv,\vecu,\tM]\bigr)\,
d\vecu\,\frac{da_1}{a_1^{d+1}}\,d\mu^{(d-1)}(\tM)\,d\vecv.
\end{align*}
Here for any fixed $\vecv,\tM$ we note that
the formula \eqref{VECKM} gives a diffeomorphism between
$\langle a_1,\vecu\rangle\in\R_{>0}\times\R^{d-1}$
and $\vecp\in\R_{k_1\vecv+}^d$, under which
$d\vecp=|k_1|^da_1^{-1}\,d\vecu da_1
=\frac{|\vecp\cdot\vecv|^d}{a_1^{d+1}}\,d\vecu da_1$.
Hence we get
\begin{align*}
=\zeta(d)^{-1}\int_{\S_1^{d-1}}\int_{{G^{(d-1)}}}\int_{\R^d_{k_1\vecv+}}
\rho\bigl([\vecv,\tM]_{\veck,\vecp}\bigr)\,
|\vecp\cdot\vecv|^{-d}\,d\vecp\,d\mu^{(d-1)}(\tM)\,d\vecv,
\end{align*}
which agrees with the right hand side of \eqref{GKYMEASLEM1PF1}.
\end{proof}

We next turn to the case $k_1=0$, i.e.\ $\veck=(0,\veck')$.
In this case $[a_1,\vecv,\vecu,\tM]\in G_{\veck,\vecp}$
implies $\vecv\in\vecp^\perp$ and
\begin{align}\label{RA1V}
\veck'\tM=\vecy,\qquad\text{where }\:
\vecy=\vecy(a_1,\vecv)
:=a_1^{\frac1{d-1}}\iota^{-1}\bigl(\vecp f(\vecv)^{-1}\bigr)
\end{align}
(cf.\ \eqref{VECKM}).
In other words $\tM\in G_{\veck',\vecy}$, where
of course $G_{\veck',\vecy}$ is a subset of ${G^{(d-1)}}$
(recall that we are assuming $d\geq3$).
Conversely, for any $a_1\in\R_{>0}$,
$\vecv\in\S_1^{d-1}\cap\vecp^\perp$, $\vecu\in\R^{d-1}$
and $\tM\in G_{\veck',\vecy(a_1,\vecv)}$
we have $[a_1,\vecv,\vecu,\tM]\in G_{\veck,\vecp}$.
We thus have a bijective map between $G_{\veck,\vecp}$ and the subset of
$\langle a_1,\vecv,\vecu,\tM\rangle$ in 
$\R_{>0}\times(\S_1^{d-1}\cap\vecp^\perp)\times\R^{d-1}\times{G^{(d-1)}}$
for which $\tM\in G_{\veck',\vecy(a_1,\vecv)}$.
The following lemma expresses the measure $\nu_\vecp$ in this parametrization.
Let us write $d_\vecp(\vecv)$ for the Lebesgue measure on the
$(d-2)$-dimensional unit sphere $\S_1^{d-1}\cap\vecp^\perp$.
Furthermore, $G_{\veck',\vecy}$ is of course endowed with a Borel measure
$\nu_\vecy$, defined as in the beginning of this section.
\begin{lem}\label{GKYMEASLEM2}
Given $\veck=(0,\veck')$
($\veck'\in\Z^{d-1}$, $\veck'\neq\bn$) and $\vecp\in\R^d\setminus\{\bn\}$,
we have
\begin{align*}
\int_{G_{\veck,\vecp}}\rho\,d\nu_\vecp
=\frac1{\zeta(d)\|\vecp\|}%
\int_{\R_{>0}\times(\S_1^{d-1}\cap\vecp^\perp)\times\R^{d-1}}
\biggl(\int_{G_{\veck',\vecy}} 
\rho\bigl([a_1,\vecv,\vecu,\tM]\bigr)\,d\nu_{\vecy}(\tM)\biggr)\,
a_1^{-d}\,da_1\,d_\vecp(\vecv)\,d\vecu
\end{align*}
for all $\rho\in L^1(G_{\veck,\vecp},\nu_\vecp)$.
Here $\vecy=\vecy(a_1,\vecv)$ is as in \eqref{RA1V}.
\end{lem}
\begin{proof}
By a similar argument as in the proof of Lemma \ref{GKYMEASLEM1} we see that
it suffices to prove that
\begin{align}\notag
\zeta(d)^{-1}\int_{\R^d\setminus\{\bn\}}\biggl(
\int_{\R_{>0}\times(\S_1^{d-1}\cap\vecp^\perp)\times\R^{d-1}}
\Bigl(\int_{G_{\veck',\vecy}} 
\rho\bigl([a_1,\vecv,\vecu,\tM]\bigr)\,d\nu_{\vecy}(\tM)\Bigr)\,
a_1^{-d}\,da_1\,d_\vecp(\vecv)\,d\vecu\biggr)\,\frac{d\vecp}{\|\vecp\|}
\\\label{GKYMEASLEM2PF1}
=\int_G\rho(M)\,d\mu(M)
\end{align}
holds for any $\rho\in C_c(G)$.

In the left hand side of \eqref{GKYMEASLEM2PF1} we express $\vecp$ in
polar coordinates as $\vecp=\ell\vecw$
($\ell>0$, $\vecw\in\S_1^{d-1}$).
Then note that the pair $\langle\vecv,\vecw\rangle$ runs through the
set of all orthonormal 2-frames in $\R^d$,
viz.\ the Stiefel manifold $V_2(\R^d)=\{\langle\vecv,\vecw\rangle\in
\S_1^{d-1}\times\S_1^{d-1}\col\vecv\cdot\vecw=0\}$,
and we have
\begin{align*}
\int_{V_2(\R^d)}\psi\,d_{\vecw}(\vecv)\,d\vecw
=\int_{V_2(\R^d)}\psi\,d_\vecv(\vecw)\,d\vecv
\end{align*}
for any $\psi\in C(V_2(\R^d))$, 
this being the unique $O(d)$-invariant measure on $V_2(\R^d)$ up to a constant.
Hence the left hand side of \eqref{GKYMEASLEM2PF1} equals
\begin{align*}
\zeta(d)^{-1}\int_0^\infty\int_{\R^{d-1}}
\int_{\S_1^{d-1}}\int_0^\infty\int_{\S_1^{d-1}\cap\vecv^\perp}
\Bigl(\int_{G_{\veck',\vecy}} 
\rho\bigl([a_1,\vecv,\vecu,\tM]\bigr)\,d\nu_{\vecy}(\tM)\Bigr)\,
d_\vecv(\vecw)\,\ell^{d-2}\,d\ell\,d\vecv\,d\vecu\,\frac{da_1}{a_1^d}.
\end{align*}
Here
$\vecy=a_1^{\frac1{d-1}}\iota^{-1}(\ell\vecw f(\vecv)^{-1})$,
and we note that (for any fixed $a_1,\vecv$) 
this formula gives a diffeomorphism between
$\langle\vecw,\ell\rangle\in(\S_1^{d-1}\cap\vecv^\perp)\times\R_{>0}$
and $\vecy\in\R^{d-1}$. Hence we get
\begin{align*}
=\zeta(d)^{-1}\int_0^\infty\int_{\R^{d-1}}
\int_{\S_1^{d-1}}\int_{\R^{d-1}}
\Bigl(\int_{G_{\veck',\vecy}} 
\rho\bigl([a_1,\vecv,\vecu,\tM]\bigr)\,d\nu_{\vecy}(\tM)\Bigr)\,d\vecy
\,d\vecv\,d\vecu\,\frac{da_1}{a_1^{d+1}}.
\end{align*}
Using now the definition of $\nu_\vecy$ and \cite[Lemma 7.4]{partI}
(with $d-1$ in place of $d$) this is
\begin{align*}
=\zeta(d)^{-1}\int_0^\infty\int_{\R^{d-1}}
\int_{\S_1^{d-1}}\int_{{G^{(d-1)}}}
\rho\bigl([a_1,\vecv,\vecu,\tM]\bigr)\,d\mu^{(d-1)}(\tM)
\,d\vecv\,d\vecu\,\frac{da_1}{a_1^{d+1}},
\end{align*}
and by \eqref{SLDRSPLITHAAR} this equals
$\int_G\rho(M)\,d\mu(M)$,
i.e.\ \eqref{GKYMEASLEM2PF1} is proved.
\end{proof}

\section{\texorpdfstring{Bounding $p_\vecp(\fC)$ for $\fC$ a cone; proof of Theorem \ref*{CONEAPEXTHM}}{Bounding pp(C) for C a cone; proof of Theorem 1.6}}

\label{PPCBOUNDCCONESEC}

\subsection{\texorpdfstring{Bounding $p_\vecp(\fC)$ from above}{Bounding pp(C) from above}}

Using the parametrization from last section 
we will now prove Theorem \ref{CONEAPEXTHM}.
Thus let $d\geq2$ and let $\fC\subset\R^d$ 
be the open cone which is the interior of the 
convex hull of a point $\vecp$ and a
relatively open $(d-1)$-dimensional ball $B$ with $\bn\in B$.
We will start by proving the bound
\begin{align}\label{CONEAPEXBOUNDABOVE}
p_\vecp(\fC)\ll|\fC|^{-2+\frac2d}.
\end{align}
Using \eqref{PPDINV}
we may assume
\begin{align}\label{CONEAPEXTHMPF1}
B=t\vece_2+\iota\bigl(\scrB_r^{d-1}\bigr)
\quad\text{and}\quad
\vecp=r\vece_1+t\vece_2,
\quad\text{for some }\:0\leq t<r,
\end{align}
We may furthermore assume that $r\geq1$, since otherwise
the bound \eqref{CONEAPEXBOUNDABOVE} is trivial.

By definition we have, since $\bn,\vecp\notin\fC$,
\begin{align*}
p_\vecp(\fC)=\nu_\vecp\bigl(\bigl\{
M\in X_1(\vecp)\col \Z^dM\cap\fC=\emptyset\bigr\}\bigr).
\end{align*}
Recall the splitting \eqref{X1YDISJUNION}.
If $k\geq2$ then every lattice $\Z^dM$ with $k\vece_1M=\vecp$
has $\vece_1M=k^{-1}\vecp\in\fC$; i.e.\ every lattice
in $X_1(k\vece_1,\vecp)$ has non-empty intersection with $\fC$.
Hence
\begin{align*}
p_\vecp(\fC)=\nu_\vecp\bigl(\bigl\{
M\in X_1(\vece_1,\vecp)\col \Z^dM\cap\fC=\emptyset\bigr\}\bigr).
\end{align*}
In the special case $d=2$ we now find directly from the definitions
(cf.\ the proof of Theorem~2 in \cite{partIII}) that $p_\vecp(\fC)$ is
less than or equal to $(6/\pi^2)$ times the Lebesgue measure of the
set of $\nu\in[0,1)$ for which $(0,r^{-1})+\nu(r,t)\notin\fC$,
viz.\ $\leq\frac6{\pi^2}r^{-1}(t+r)^{-1}\ll r^{-2}$.
Hence the bound \eqref{CONEAPEXBOUNDABOVE} 
is proved for $d=2$, and from now on we %
assume $d\geq3$.

Recall from Section \ref{X1RECOLLECTSEC} that
$X_1(\vece_1,\vecp)=\Gamma\backslash\Gamma G_{\vece_1,\vecp}$.
For any $\gamma\in\Gamma$ we have
$\gamma G_{\vece_1,\vecp}=G_{\vece_1\gamma^{-1},\vecp}$,
and $\{\vece_1\gamma^{-1}\col\gamma\in\Gamma\}$ equals the subset
$\widehat\Z^d$ of primitive vectors in $\Z^d$; thus
\begin{align}\label{X1E1P}
X_1(\vece_1,\vecp)=\Gamma\backslash \Bigl(\bigsqcup_{\veck\in\widehat\Z^d}
G_{\veck,\vecp}\Bigr).
\end{align}
Furthermore the measure
``$\nu_\vecp$ on $G_{\vece_1\gamma^{-1},\vecp}$''
(as defined just below \eqref{HAV})
corresponds to ``$\nu_\vecp$ on $G_{\vece_1,\vecp}$'' under
the diffeomorphism 
$G_{\vece_1,\vecp}\ni M\mapsto\gamma M\in G_{\vece_1\gamma^{-1},\vecp}$.
Hence since $\Si_d$ contains a fundamental domain for $\Gamma\backslash G$, 
we get
\begin{align}\label{CONEAPEXTHMPF3}
p_\vecp(\fC)
\leq\sum_{\veck\in\widehat\Z^d}
\nu_\vecp\bigl(\bigl\{
M\in G_{\veck,\vecp}\cap\Si_d\col \Z^dM\cap\fC=\emptyset\bigr\}\bigr)
={\sum}_0+{\sum}_1,
\end{align}
where ${\sum}_0$ and ${\sum}_1$ are the sums corresponding to $k_1=0$ and 
$k_1\neq0$, respectively.
By Lemma~\ref{A1LARGELEM},
if $M=[a_1,\vecv,\vecu,\tM]\in\Si_d$ satisfies
$\Z^dM\cap\fC=\emptyset$ then $a_1>Cr$, where $C$ is a positive constant
which only depends on $d$.
Also recall that $M\in\Si_d$ forces 
$\vecu\in(-\sfrac12,\sfrac12]^{d-1}$.
We first consider $\sum_1$.
By \eqref{UFORMULAIFK1NEQ0} only $\veck$ with 
$|k_1|\leq\|\vecp\|/(Cr)<2/C$ contribute to this sum
(cf.\ \eqref{CONEAPEXTHMPF1} for the last inequality).
Hence, using \eqref{SIDSUBSSIDM1}, \eqref{LATTICEINPARAM} with $n=0$,
and Lemma \ref{GKYMEASLEM1}, we get
\begin{align}\notag
{\sum}_1\leq
\zeta(d)^{-1}
\sum_{\substack{|k_1|<2/C\\k_1\neq0}}
\int_S %
\sum_{\veck'\in\Z^{d-1}}
\mu^{(d-1)}\Bigl(\Bigl\{\tM %
\in\Si_{d-1}\col
\vecu(\vecv,\veck,\tM)\in(-\sfrac12,\sfrac12]^{d-1}\:\text{ and} 
\hspace{45pt}&
\\\label{CONEAPEXTHMPF4}
\Z^{d-1}\tM\cap a_1^{\frac1{d-1}}\fC_\vecv=\emptyset\Bigr\}\Bigr)
\,\frac{d\vecv}{|\vecp\cdot\vecv|^d}&,
\end{align}
wherein
\begin{align*}
S=S^{(k_1)}=\{\vecv\in\S_1^{d-1}\col k_1^{-1}(\vecp\cdot\vecv)>Cr\}
\end{align*}
and $a_1=k_1^{-1}(\vecp\cdot\vecv)$, %
and $\vecu(\vecv,\veck,\tM)$ is given by \eqref{UFORMULAIFK1NEQ0}.
Also $\fC_\vecv=\iota^{-1}(\fC f(\vecv)^{-1})$ as in \eqref{CVDEFNEW}.

Note that in \eqref{UFORMULAIFK1NEQ0} the term 
$\vecz=k_1^{-1}a_1^{\frac1{d-1}}
\iota^{-1}\bigl((\vecp-k_1a_1\vecv)f(\vecv)^{-1}\bigr)\tkk^{-1}a(\ta)^{-1}$
is independent of $\veck'$, and $(-\sfrac12,\sfrac12]^{d-1}$
is a fundamental domain for $\R^{d-1}/\Z^{d-1}\nn(\tu)$;
hence for any fixed $\vecv\in S$ and $\tM\in\Si_{d-1}$ there are exactly
$|k_1|^{d-1}$ choices of $\veck'\in\Z^{d-1}$ for which
$\vecu(\vecv,\veck,\tM)=\vecz-k_1^{-1}\veck'n(\tu)$ lies in $(-\sfrac12,\sfrac12]^{d-1}$.
Thus
\begin{align}\label{SUM1BOUNDFIRSTTIMEEVER}
{\sum}_1\ll
\sum_{\substack{|k_1|<2/C\\k_1\neq0}}
|k_1|^{d-1}\int_{S}
\mu\Bigl(\Bigl\{\tM\in\Si_{d-1}\col
\Z^{d-1}\tM\cap a_1^{\frac1{d-1}}\fC_\vecv=\emptyset\Bigr\}\Bigr)
\,\frac{d\vecv}{(|k_1|r)^d}.
\end{align}
We will use the following lemma to bound the integrand in the right hand side.
\begin{lem}\label{CONEINTERSECTLEM}
For any $\vecv$ as in \eqref{VPARA} with $\varpi,\omega\in(0,\pi)$,
$\fC_\vecv$ 
contains a $(d-1)$-dimensional cone of
volume $\gg r^{d-1}(\sin\omega)^d$,
the base of which contains $\bn$.
\end{lem}
\begin{proof}
Note that (by \eqref{CONEAPEXTHMPF1}, and since $0<\varpi<\pi$)
$B\cap\vecv^\perp$ is an open $(d-2)$-dimensional ball of radius
$r'=\sqrt{r^2-t^2\cos^2\omega}$.
Furthermore the diameter of $B$ which is perpendicular to
$\vece_1^\perp\cap\vecv^\perp$ intersects $\vece_1^\perp\cap\vecv^\perp$
at the center $\vecc$ of $B\cap\vecv^\perp$, and
if $\vecy_1$ and $\vecy_2$ are the endpoints of this diameter,
taken so that $\|\vecy_1-\vecc\|\leq\|\vecy_2-\vecc\|$,
then $\|\vecy_1-\vecc\|=r-t|\cos\omega|$.
Since $\vecy_1$ and $\vecy_2$ lie on different sides of the hyperplane 
$\vecv^\perp$
it follows that $\vecv^\perp$ must intersect either the line segment 
$\vecp\vecy_1$ or the line segment $\vecp\vecy_2$, say at the point
$\vecy'\in\vecp\vecy_j$.
The angle between $\vece_1^\perp$ and the line $\vecp\vecy_j$
is $\frac\pi4$ (cf.\ \eqref{CONEAPEXTHMPF1})
and thus the distance between $\vecy'$ and
$\vece_1^\perp\cap\vecv^\perp$ must be 
$\geq2^{-\frac12}\|\vecy_j-\vecc\|\geq2^{-\frac12}(r-t|\cos\omega|)$.
By convexity $\fC\cap\vecv^\perp$ contains the $(d-1)$-dimensional
open cone with base $B\cap\vecv^\perp$ and apex $\vecy'$,
and this cone has volume
$\gg(r-t|\cos\omega|){r'}^{d-2}\geq r^{d-1}(1-|\cos\omega|)^{\frac d2}
\gg r^{d-1}(\sin\omega)^d.$
\end{proof}

The lemma together with \eqref{ATHREYAMARGULISRES}
imply that, for any $\vecv$ as in \eqref{VPARA} and any $a_1>Cr$,
\begin{align}\label{TRIVCONEDM1BOUND}
p^{(d-1)}\bigl(a_1^{\frac1{d-1}}\fC_\vecv\bigr)
\ll\min\bigl(1,r^{-d}(\sin\omega)^{-d}\bigr).
\end{align}
Hence we get from \eqref{SUM1BOUNDFIRSTTIMEEVER},
using also \eqref{DVINPARA} and the fact that
$\Si_{d-1}$ can be covered by a finite number of fundamental regions
for $\Gamma^{(d-1)}\backslash G^{(d-1)}$, 
\begin{align}\label{CONEAPEXTHMPF1a}
{\sum}_1\ll 
r^{-d}\int_0^{\pi/2}\min\bigl(1,r^{-d}\omega^{-d}\bigr)\omega^{d-3}\, d\omega
\ll r^{2-2d}.
\end{align}

It remains to treat $\sum_0$.
Using \eqref{SIDSUBSSIDM1}, \eqref{LATTICEINPARAM} %
and Lemma \ref{GKYMEASLEM2},
we get
\begin{align*}
{\sum}_0
\ll&\|\vecp\|^{-1}
\int_{Cr}^\infty
\int_{\S_1^{d-1}\cap\vecp^\perp}
\sum_{\veck'\in\widehat\Z^{d-1}}
\nu_{\vecy}\Bigl(\Bigl\{\tM\in G_{\veck',\vecy}\cap\Si_{d-1}\col
\Z^{d-1}\tM\cap a_1^{\frac1{d-1}}\fC_\vecv=\emptyset
\Bigr\}\Bigr)
\,d_\vecp(\vecv)\,a_1^{-d}\,da_1,
\end{align*}
where $\vecy=\vecy(a_1,\vecv)=a_1^{\frac1{d-1}}\iota^{-1}(\vecp f(\vecv)^{-1})$
as in \eqref{RA1V}.
Now for any fixed $a_1,\vecv$ we have,
using \eqref{X1E1P} applied to
$X_1^{(d-1)}(\vece_1,\vecy)$, and the fact that
$\Si_{d-1}$ is contained in a finite union of fundamental regions
for $\Gamma^{(d-1)}\backslash G^{(d-1)}$,
\begin{align*}
\sum_{\veck'\in\widehat\Z^{d-1}}
\nu_{\vecy}\Bigl(\Bigl\{\tM\in G_{\veck',\vecy}\cap\Si_{d-1}\col
\Z^{d-1}\tM\cap a_1^{\frac1{d-1}}\fC_\vecv=\emptyset
\Bigr\}\Bigr)
\ll p_{\vecy}^{(d-1)}\bigl(a_1^{\frac1{d-1}}\fC_\vecv\bigr).
\end{align*}
But $\fC_\vecv$ is isometric with $\fC\cap\vecv^\perp$,
and (since $\vecv\in\S_1^{d-1}\cap\vecp^\perp$)
this set equals the $(d-1)$-dimensional open cone with base
$B\cap\vecv^\perp$ and apex $\vecp$.
This cone has height $\geq r$ and 
radius $r'=\sqrt{r^2-t^2\cos^2\omega}\geq r\sin\omega$,
where as usual $\omega\in[0,\pi]$ is the angle between 
$\vecv'=(v_2,\ldots,v_d)$ and $\vece_1$ in $\R^{d-1}$.
Hence if $a_1\geq Cr$ then the set $a_1^{\frac1{d-1}}\fC_\vecv$
contains a $(d-1)$-dimensional open cone with $\bn$ in its base,
apex $\vecy$, and which has volume $\gg r^d(\sin\omega)^{d-2}$.
Hence by induction
we have
\begin{align*}
p_{\vecy}^{(d-1)}\bigl(a_1^{\frac1{d-1}}\fC_\vecv\bigr)
\ll %
\bigl(r^d(\sin\omega)^{d-2}\bigr)^{-2+\frac2{d-1}}. %
\end{align*}
This gives
\begin{align*}
{\sum}_0\ll r^{-d}\int_{\S_1^{d-1}\cap\vecp^\perp}
\min\Bigl(1,\bigl(r^d(\sin\omega)^{d-2}\bigr)^{-2+\frac2{d-1}}\Bigr)
\,d_\vecp(\vecv).
\end{align*}
Finally note that the map $\vecv\mapsto\|\vecv'\|^{-1}\vecv'$ is a 
diffeomorphism from $\S_1^{d-1}\cap\vecp^\perp$ onto $\S_1^{d-2}$ 
whose Jacobian determinant is uniformly bounded from below by a positive
constant which only depends on $d$, 
since $\varphi(\vecp,\vece_1)\leq\frac\pi4$.
Hence we get
\begin{align}\label{RESULTFROMBOUNDEDJAC}
{\sum}_0\ll r^{-d}\int_0^{\pi}
\min\bigl(1,(r^{-d}(\sin\omega)^{2-d})^{2-\frac2{d-1}}\bigr)\,
(\sin\omega)^{d-3}\,d\omega
\ll \begin{cases} r^{-6}\log(1+r)&\text{if }\:d=3
\\ r^{-2d}&\text{if }\:d\geq4,
\end{cases}
\end{align}
since $(2-d)(2-\frac2{d-1})+d-3$ equals $-1$ when $d=3$ but is less than
$-1$ when $d\geq4$.
Now the bound \eqref{CONEAPEXBOUNDABOVE}
follows from \eqref{CONEAPEXTHMPF1a} and \eqref{RESULTFROMBOUNDEDJAC}.

\subsection{\texorpdfstring{Proof that $p_\vecp(\fC)$ vanishes whenever $e|\fC|^{\frac2d}$ is sufficiently large}{Proof that pp(C) vanishes whenever * is sufficiently large}}
\label{SEC6P2}

Let $\fC$ be as in the previous section (cf.\ \eqref{CONEAPEXTHMPF1}).
As in the proof of Lemma \ref{MAINTHM2REDAUXLEM} there is some $M\in G$
such that we may imbed in $\fC M$ a cut ball with cut ratio $t=\frac15$,
edge ratio $\min(1,20e)$ and volume $\gg|\fC|$.
Applying now the argument just above Remark \ref{LASTDISCREM}
in Section \ref{PROOFOFTHMMAINTHM2SEC} 
involving Minkowski's Theorem,
we conclude that there is a constant $c>0$ which only depends on $d$
such that if $e\geq c|\fC|^{-\frac2d}$ then there does not exist any
lattice $L\in X_1$ which satisfies both $L\cap\vece_1^\perp=\{\bn\}$
and $L\cap\fC=\emptyset$.
Hence a fortiori there is no such lattice in $X_1(\vecp)$.
This implies that $p_\vecp(\fC)=0$, once we note that
\begin{align}\label{CONEAPEXTHMPF2}
\nu_\vecp\bigl(\bigl\{M\in X_1(\vecp)\col\Z^dM\cap\vece_1^\perp\neq\{\bn\}
\bigr\}\bigr)=0.
\end{align}
This relation is valid for any $\vecp\notin\vece_1^\perp$
and can be proved in many ways;
for instance if $d\geq3$ it follows immediately from
\cite[Prop.\ 7.6]{partI} (applied with $q=1$, $\vecalf=\bn$, $\vecy=\vecp$ and
$F$ as the characteristic function of $\vece_1^\perp\cap\scrB_R^d$,
and then letting $R\to\infty$).
For $d=2$, \eqref{CONEAPEXTHMPF2} 
follows in a similar way from \cite[Prop.\ 7.8]{partI},
or more easily directly from the definition of $\nu_\vecp$
(cf.\ \cite[Sec.\ 7.1]{partI} and also 
the proof of Theorem~2 in \cite{partIII}).

Hence we have proved that $p_\vecp(\fC)=0$ holds whenever
$e\geq c|\fC|^{-\frac2d}$.
Combining this with \eqref{CONEAPEXBOUNDABOVE} we have now proved
the \textit{right} inequality in \eqref{CONEAPEXTHMRES}.

\subsection{\texorpdfstring{Bounding $p_\vecp(\fC)$ from below}{Bounding pp(C) from below}}
\label{PPCFROMBELOWSEC}

Finally we will prove the \textit{left} inequality in \eqref{CONEAPEXTHMRES},
viz.\ that there exists a constant $c>0$ which only depends on $d$
such that whenever $|\fC|\geq\frac12$ and
the edge ratio is $e<c|\fC|^{-\frac2d}$ then
\begin{align}\label{CONEAPEXTHMRES3RESTAT}
p_\vecp(\fC)\gg|\fC|^{-2+\frac2d}.
\end{align}
If $d=2$ then this statement can be verified easily
by an explicit computation; 
alternatively it follows from the explicit formula in
\cite[Theorem 2]{partIII} (cf.\ \eqref{PF2DIMGENERIC}),
as we now indicate.
We keep $\fC$ as in \eqref{CONEAPEXTHMPF1}.
Let $\fC'$ be the open parallelogram with vertices
$(0,t\pm r)$, $(r,\pm r)$.
Then $\fC\subset\fC'$, and using \eqref{PPDINV} and the definition
of $\Phi_\bn$ (cf.\ \eqref{PHI0DEF})
we see that $p_\vecp(\fC)\geq p_\vecp(\fC')=\Phi_\bn(r^2,\frac tr,\frac tr)$.
Now if $e=\frac{r-t}r<\frac12r^{-2}$ and $r\geq1$ then
$\Phi_\bn(r^2,\frac tr,\frac tr)=\frac6{\pi^2}\frac{r^{-1}-(r-t)}{2t}
\gg r^{-2}$; on the other hand if $r\leq1$ then
$\Phi_\bn(r^2,\frac tr,\frac tr)\geq\frac3{\pi^2}$ for all $t\in[0,r]$.
Hence \eqref{CONEAPEXTHMRES3RESTAT} holds whenever 
$e<\frac12r^{-2}=\frac12|\fC|^{-1}$
and $|\fC|\geq\frac12$.

From now on we assume $d\geq3$.
Arguing as in \eqref{X1E1P}, \eqref{CONEAPEXTHMPF3}, \eqref{CONEAPEXTHMPF4}
(considering only $k_1=1$)
and using the first relation in \eqref{SIDSUBSSIDM1} and the fact that 
$\Si_d$ can be covered by a finite number of fundamental regions
for $\Gamma\backslash G$, we obtain
\begin{align}\notag
p_\vecp(\fC)\gg\int_{\S_{\vecp+}^{d-1}}\sum_{\veck'\in\Z^{d-1}}
\mu\Bigl(\Bigl\{\tM\in\Si_{d-1}\col
\ta_1\leq\sfrac2{\sqrt3}a_1^{\frac d{d-1}},\:
\vecu(\vecv,\veck,\tM)\in(-\sfrac12,\sfrac12]^{d-1},
\hspace{50pt}
\\\label{PPCFROMBELOWSECSTAT1}
\Z^dM\cap\fC=\emptyset\Bigr\}\Bigr)
\,\frac{d\vecv}{|\vecp\cdot\vecv|^d},
\end{align}
wherein $a_1=\vecp\cdot\vecv$, $\veck=(1,\veck')$,
and $\vecu(\vecv,\veck,\tM)$ is given by \eqref{UFORMULAIFK1NEQ0}, and
$M=[\vecv,\tM]_{\veck,\vecp}$.

Let us restrict the range of $\vecv$ to the set
\begin{align}\label{PPCFROMBELOWSECSTAT2}
S=\{\vecv\in\S_1^{d-1}\col 0<\varpi<\sfrac14\pi,\:0<\omega<cr^{-1}\},
\end{align}
where from start we assume $0<c<\frac12$ and $r\geq1$, 
$e=\frac{r-t}r<\frac12$.
In particular $cr^{-1}<\frac12<\frac\pi6$ and hence for all $\vecv\in S$ 
the number $a_1=\vecp\cdot\vecv$ is bounded from above and below by
\begin{align}\label{A1FROMABOVEBELOW}
2^{-\frac12}r<r\cos\varpi\leq r\cos\varpi+t\sin\varpi\cos\omega
=\vecp\cdot\vecv\leq\|\vecp\|\leq2^{\frac12}r.
\end{align}
Let us note that for any $\vecv\in S$ and any $n\in\Z\setminus\{0\}$,
the hyperplane $na_1\vecv+\vecv^\perp$ is disjoint from $\fC$.
Indeed, for $n\geq1$ this holds since
$\vecp\in a_1\vecv+\vecv^\perp$ and $\varphi(\vecv,\vece_1)=\varpi<\frac\pi4$;
on the other hand for $n\leq-1$ it holds since
for every $\vecx\in\fC$ we have $x_1>0$, $x_2>t-r$ and 
$\|(x_3,\ldots,x_d)\|<r$,
and thus (using also \eqref{VPARA}, $\frac{r-t}r<\frac12$ and 
$\sin\omega<\frac12$):
\begin{align*}
\vecx\cdot\vecv\geq x_2\sin\varpi\cos\omega
-\|(x_3,\ldots,x_d)\|\sin\varpi\sin\omega
>(\sin\varpi)(t-r-\sfrac12r)
>-2^{-\frac12}r\geq-a_1.
\end{align*}
Using the disjointness just proved
and \eqref{LATTICEINPARAM}, \eqref{LATTICECONTAINEMENT}, 
we now have
\begin{align*}
p_\vecp(\fC)\gg r^{-d}\int_S\sum_{\veck'\in\Z^{d-1}}
\mu\Bigl(\Bigl\{\tM\in\Si_{d-1}\col
\ta_1\leq(2^{-1/2}r)^{\frac d{d-1}},\:
\vecu(\vecv,\veck,\tM)\in(\sfrac12,\sfrac12]^{d-1},
\hspace{50pt}
\\
\Z^{d-1}\tM\cap a_1^{\frac1{d-1}}\fC_\vecv=\emptyset\Bigr\}\Bigr)
\,d\vecv.
\end{align*}
As we noted above \eqref{SUM1BOUNDFIRSTTIMEEVER}, for any fixed
$\vecv\in S$ and $\tM\in\Si_{d-1}$
there is exactly one $\veck'\in\Z^{d-1}$ for which 
$\vecu(\vecv,\veck,\tM)\in(\sfrac12,\sfrac12]^{d-1}$. Hence
\begin{align}\label{PPCBOUNDFROMBELOW}
p_\vecp(\fC)\gg r^{-d}\int_S
\mu\Bigl(\Bigl\{\tM\in\Si_{d-1}\col
\ta_1\leq(2^{-1/2}r)^{\frac d{d-1}},\:
\Z^{d-1}\tM\cap a_1^{\frac1{d-1}}\fC_\vecv=\emptyset\Bigr\}\Bigr)
\,d\vecv.
\end{align}
Note that for $\vecv\in S$, the set $\fC_\vecv$ is contained in a 
$(d-1)$-dimensional cylinder of radius $\sqrt{r^2-t^2\cos^2\omega}$
and height $\sqrt2(r-t\cos\omega)\ll er+c^2r^{-1}$;
hence $|\fC_\vecv|\ll r^{\frac d2-1}(er+c^2r^{-1})^{\frac d2}$.
Using also \eqref{A1FROMABOVEBELOW} it follows that 
if the constant $c$ is sufficiently small (in a way that only depends on $d$)
and if $e<c|\fC|^{-\frac2d}$,
then $|a_1^{\frac1{d-1}}\fC_\vecv|<\frac12$ holds for all $\vecv\in S$.
Hence there is some $C>1$ which only depends on $d$ such that if
both $e<c|\fC|^{-\frac2d}$ and $|\fC|\geq C$ hold,
then the integrand in \eqref{PPCBOUNDFROMBELOW} is $\geq\frac13$
for all $\vecv\in S$;
and thus by also using \eqref{DVINPARA} it follows that  
\eqref{CONEAPEXTHMRES3RESTAT} holds.

The case when $\frac12\leq|\fC|<C$
may be treated by applying the following lemma to $\fC$ and an
appropriate cone $\fC'$ with $|\fC'|=C$.
\begin{lem}
If $\fC,\fC'\subset\R^d$ are two open cones %
with $|\fC|<|\fC'|$
which both contain $\bn$ in their bases, with equal edge ratios,
and if $\vecp,\vecp'$ are their respective apexes,
then $p_\vecp(\fC)\geq p_{\vecp'}(\fC').$
\end{lem}
\begin{proof}
Using \eqref{PPDINV} we may assume that $\fC$ and $\fC'$ have the same base
$B=t\vece_2+\iota(\scrB_1^{d-1})$
($0\leq t<1$),
and that $\fC$ has apex $\vecp=h\vece_1$ and $\fC'$ has apex
$\vecp'=h'\vece_1$ (some $h'>h>0$).
Set 
$M=\text{diag}[h/h',(h'/h)^{\frac1{d-1}},$ $\cdots,(h'/h)^{\frac1{d-1}}]\in G$;
then $\vecp'M=\vecp$ and $\fC'M\supset\fC$;
hence $p_{\vecp'}(\fC')=p_{\vecp}(\fC'M)\leq p_\vecp(\fC).$
\end{proof}
We conclude that, with a new constant $c$,
\eqref{CONEAPEXTHMRES3RESTAT} indeed holds
whenever $|\fC|\geq\frac12$ and $e<c|\fC|^{-\frac2d}$.
This concludes the proof of Theorem \ref{CONEAPEXTHM}.
\hfill$\square$ $\square$ $\square$

\begin{remark}\label{PPCVOLSMALLREM}
Regarding the restriction $|\fC|\geq\frac12$ in Theorem \ref{CONEAPEXTHM},
we remark that $p_\vecp(\fC)>\frac1{10}$ holds whenever
$|\fC|<\frac12$.
This follows from 
\cite[Prop.\ 7.6]{partI} applied with $F$ as the characteristic function
of $\fC$; for note that the right hand side of \cite[(7.30)]{partI}
is then $\leq|\fC|+\sum_{t=2}^\infty\varphi(t)t^{-d}
\leq|\fC|+\frac{\zeta(2)}{\zeta(3)}-1<\frac12+\frac25=\frac9{10}$.
\end{remark}

\section{\texorpdfstring{Bounding $\Phi_\bn(\xi,\vecw,\vecz)$; proof of Theorem \ref*{CYLINDER2PTSMAINTHM} and Proposition \ref*{CYLSUPPPROP}}{Bounding Phi0; proof of Theorem 1.8 and Proposition 1.9}}
\label{CYLINDER2PTSEC}

\subsection{\texorpdfstring{On the support of $\Phi_\bn$; proof of Proposition \ref*{CYLSUPPPROP}}{On the support of Phi0; proof of Proposition 1.9}}

Our first task is to prove that
$\Phi_\bn(\xi,\vecz,\vecw)>0$ implies
$e\ll s_d(\xi,\varphi)$, where
\begin{align*}
e:=\max(1-\|\vecz\|,1-\|\vecw\|),
\end{align*}
and
\begin{align*}
s_d(\xi,\varphi):=
\begin{cases}\min(\xi^{-\frac2d},
(\varphi \xi)^{-\frac2{d-1}})&\text{if }\:\varphi\leq\frac\pi2
\\
\max(\xi^{-\frac2{d-2}},(\frac{\xi}{\pi-\varphi})^{-\frac2{d-1}})
&\text{if }\:\varphi>\frac\pi2,
\end{cases}
\end{align*}
and $s_d(\xi,\varphi):=\xi^{-\frac2{d-2}}$ when $\varphi$ is undefined.
We will prove this using Minkowski's Theorem, 
similarly as in Section \ref{SEC6P2}
and in Section \ref{PROOFOFTHMMAINTHM2SEC} just above Remark \ref{LASTDISCREM}.
We first make some convenient reductions by invoking Theorem \ref{CONEAPEXTHM}.

Recall that $\Phi_\bn(\xi,\vecw,\vecz)=p_\vecp(\fC)$,
where $\fC$ is the open cylinder which is the interior of the
convex hull of the two $(d-1)$-dimensional balls 
\begin{align*}
B_1=\iota(\vecz+\scrB_1^{d-1})\quad\text{and}\quad
B_2=\xi\vece_1+B_1,
\end{align*}
and where $\vecp=(\xi,\vecz+\vecw)$.
Now by Theorem \ref{CONEAPEXTHM},
since $\fC$ contains the open cone with base $B_1$ and apex $\vecp$,
$\Phi_\bn(\xi,\vecw,\vecz)>0$ implies 
$1-\|\vecz\|\ll|\fC|^{-\frac2d}\ll\xi^{-\frac2d}$.
Next let us set
\begin{align*}
\fC':=\vecp-\fC.
\end{align*}
Then $p_\vecp(\fC)=p_\vecp(\fC')$ by \eqref{PPFSYMM}, and since 
$\fC'$ is the convex hull of $B_1'=\iota(\vecw+\scrB_1^{d-1})$ and
$\xi\vece_1+B_1'$, we conclude
\begin{align}\label{PHI0SYMM}
\Phi_\bn(\xi,\vecw,\vecz)=\Phi_\bn(\xi,\vecz,\vecw).
\end{align}
Hence $\Phi_\bn(\xi,\vecw,\vecz)>0$ also implies
$1-\|\vecw\|\ll\xi^{-\frac2d}$
and thus
$e\ll\xi^{-\frac2d}$.
This gives the desired conclusion in the case $\varphi\leq\xi^{-\frac1d}$,
since then $s_d(\xi,\varphi)\gg\xi^{-\frac2d}$.
Hence from now on we may assume $\varphi>\xi^{-\frac1d}$.
We may also assume that $e$ is small
(since otherwise $e\ll\xi^{-\frac2d}$ forces
$\xi\ll1$ and thus $s_d(\xi,\varphi)\gg1$);
say $0<e<\frac1{10}$.

Using \eqref{PPDINV}
we may assume
\begin{align}\label{CYLSUPPPROPPF1}
\vecw=(y,w,0,\ldots,0);\qquad
\vecz=(y,z,0,\ldots,0),
\end{align}
for some $y\geq0$, $w,z\in\R$. Now let $F$ be the box
\begin{align*}
F=[-\sfrac12\xi,\sfrac12\xi]\times [-s,s]
\times\Bigl[-\sfrac12|z-w|,\sfrac12|z-w|\Bigr]
\times
\underbrace{
\Bigl[-\frac{\sqrt e}{2\sqrt{d}},\frac{\sqrt e}{2\sqrt{d}}\Bigr]
\times\cdots\times
\Bigl[-\frac{\sqrt e}{2\sqrt{d}},\frac{\sqrt e}{2\sqrt{d}}\Bigr]}_{\text{$d-3$ copies}}
\end{align*}
where $s=\frac{e}{10(\sqrt e+y)}$.
We then claim that 
\begin{align}\label{FINCLUSION}
F\quad\subset\quad
\fC\cup\fC'\cup\vece_1^\perp\cup(-\fC)\cup(-\fC').
\end{align}
Indeed, let $\vecx=(x_1,\ldots,x_d)$ be an arbitrary point in $F$.
Using $|x_3|\leq\frac12|z-w|$ and splitting into the two cases
$zw\leq0$ and $zw>0$ we check that
$\min(|x_3-z|,|x_3-w|)\leq\frac12(|z|+|w|)$. 
Hence we have for $\vecx':=(x_2,x_3,\ldots,x_d)$:
\begin{align}\notag
\min\bigl(\|\vecx'-\vecz\|^2,\|\vecx'-\vecw\|^2\bigr)
&<(s+y)^2+\sfrac14\bigl(|z|+|w|\bigr)^2+\sfrac14e
\\\notag
&\leq y^2+2sy+s^2+\sfrac 12(z^2+w^2)+\sfrac14e
<\sfrac12\bigl(1+(1-e)^2+e\bigr)<1,
\end{align}
where we used $2sy\leq\frac15e$, $s^2\leq\frac1{100}e$
and the fact that one of $y^2+z^2$ and $y^2+w^2$ equals $(1-e)^2$
while both are $<1$.
The above inequality 
proves that $\vecx\in\fC\cup\fC'$ whenever $x_1\in(0,\frac12\xi]$.
Similarly $\vecx\in(-\fC)\cup(-\fC')$ whenever $x_1\in[-\frac12\xi,0)$,
and this completes the proof of the inclusion \eqref{FINCLUSION}.

Now assume $\Phi_\bn(\xi,\vecw,\vecz)>0$, viz.\ $p_\vecp(\fC)>0$.
Then by \eqref{CONEAPEXTHMPF2} there is some lattice
$L\in X_1(\vecp)$ which is disjoint from $\fC$ and which
satisfies $L\cap\vece_1^\perp=\{\bn\}$.
As noted above $L$ must also be disjoint from $\fC'$ 
and thus also from $-\fC$ and $-\fC'$.
Hence by \eqref{FINCLUSION}, $L\cap F=\{\bn\}$, and
now Minkowski's Theorem implies $|F|\leq2^d$, thus
\begin{align}\label{CYLSUPPPROPPF2}
\xi s|z-w|e^{\frac{d-3}2}\ll1.
\end{align}
Note that $|z-w|=\|\vecz-\vecw\|\gg\varphi$.
(Indeed this is obvious if $\varphi\geq\frac\pi2$,
since $\|\vecz\|,\|\vecw\|\geq1-e>\frac9{10}$,
and if $\varphi<\frac\pi2$ then 
$\|\vecz-\vecw\|\geq\min_{t\in\R}\|\vecz-t\vecw\|
=\|\vecz\|\sin\varphi\gg\varphi$.)
Using this together with \eqref{CYLSUPPPROPPF2} and
$s>\frac e{20}$ we obtain $e\ll (\varphi\xi)^{-\frac2{d-1}}$,
which is the desired bound if $\varphi\leq\frac\pi2$,
since we are assuming $\varphi>\xi^{-\frac1d}$.
On the other hand if $\varphi>\frac\pi2$ then we note that
$y\asymp\pi-\varphi$, e.g.\ since the area of the
triangle with vertices $\bn,\vecw,\vecz$
can be expressed both as $\frac12y|z-w|\asymp y$ and as
$\frac12\|\vecw\|\cdot\|\vecz\|\sin\varphi\asymp\pi-\varphi$.
Hence $s\asymp\min(\sqrt e,\frac e{\pi-\varphi})$ and thus
\eqref{CYLSUPPPROPPF2} implies
$e\ll\max(\xi^{-\frac2{d-2}},(\frac\xi{\pi-\varphi})^{-\frac2{d-1}})$.
Hence in all cases we have $e\ll s_d(\xi,\varphi)$,
and this concludes the proof of the first statement in
Proposition \ref{CYLSUPPPROP}.
\hfill$\square$

\vspace{5pt}

We now turn to the proof of 
the second statement of Proposition \ref{CYLSUPPPROP},
viz.\ that there is a constant $c>0$ which only depends on $d$ such that
$\Phi_\bn(\xi,\vecw,\vecz)>0$ holds whenever $e<c\cdot s_d(\xi,\varphi)$.
Note that in order to prove $\Phi_\bn(\xi,\vecw,\vecz)>0$ it suffices to
construct a lattice $L\in X_1(\vecp)$ satisfying
$L\cap\overline{\fC}=\{\bn,\vecp\}$.
The reason for this is that $\tilde L\cap\overline{\fC}=\{\bn,\vecp\}$
(and thus $\tilde L\cap\fC=\emptyset$) must then hold for all
$\tilde L$ in some neighbourhood of $L$ in $X_1(\vecp)$.

Given $\vecw,\vecz$ as in \eqref{CYLSUPPPROPPF1},
to construct an admissible lattice $L$ we take
$\vecv=(1,\ve,0,0,\ldots,0)$ with $\ve>0$ small,
and let $I_\vecv:\R^{d-1}\to\vecv^\perp$ be the linear map
\begin{align*}
I_\vecv(x_1,\ldots,x_{d-1})=(-\ve x_1,x_1,\ldots,x_{d-1}).
\end{align*}
We will take $L$ to be the lattice spanned by
$I_\vecv(L')$ and $\vecp$ for an appropriate $(d-1)$-dimensional lattice 
$L'\subset\R^{d-1}$.
Then note that $L\subset\sqcup_{n\in\Z}(n\vecp+\vecv^\perp)$.
We claim that, if $\ve<\frac12\xi$, then
\begin{align}\label{CYLSUPPPROPPF3}
(n\vecp+\vecv^\perp)\cap\overline\fC=\emptyset,\qquad
\forall n\in\Z\setminus\{0,1\}.
\end{align}
Indeed, assume $\vecx\in(n\vecp+\vecv^\perp)\cap\overline\fC$
for some $n\in\Z\setminus\{0,1\}$.
Then $(\vecx-n\vecp)\cdot\vecv=0$, i.e.\
$\ve x_2=-x_1+n(\xi+2\ve y)$.
Since $x_1\in[0,\xi]$, $y\geq0$ and $n\geq2$ or $n\leq-1$, this forces
$\ve x_2\geq\xi$ or $\ve x_2\leq-\xi$.
Therefore $|x_2|>2$ and
$\|\vecx'-\vecz\|\geq|x_2-y|>2-1=1$, 
contradicting $\vecx\in\overline\fC$.

Because of \eqref{CYLSUPPPROPPF3},
$L\cap\overline{\fC}=\{\bn,\vecp\}$ holds if and only if
$I_\vecv(L')\cap\overline{\fC}=\{\bn\}$ and
$(\vecp+I_\vecv(L'))\cap\overline{\fC}=\{\vecp\}$.
The latter condition is equivalent with
$I_\vecv(L')\cap\overline{\fC'}=\{\bn\}$, since $I_\vecv(L')=-I_\vecv(L')$.
From the definition of $I_\vecv$ we see that 
for any $\vecell=(\ell_1,\ldots,\ell_{d-1})$ with $\ell_1>0$ we have
$I_\vecv(\vecell)\cdot\vece_1<0$, and thus $I_\vecv(\vecell)\notin\overline\fC\cup
\overline{\fC'}$. Hence we have $L\cap\overline{\fC}=\{\bn,\vecp\}$
so long as every non-zero lattice point
$\vecell=(\ell_1,\ldots,\ell_{d-1})\in L'$ with $\ell_1\leq0$ satisfies
both $\|\vecell-\vecz\|>1$ and $\|\vecell-\vecw\|>1$.

Let us first assume $\varphi>\frac\pi2$ (or $\varphi$ undefined,
i.e.\ $\vecz=\bn$ or $\vecw=\bn$).
Then choose %
\begin{align}\label{EXPLLATTICE1}
L':=\text{Span}_\Z\Bigl\{\alpha\frac e{\sqrt e+y}\vece_1,
3\vece_2,2\sqrt e\vece_3,\ldots,2\sqrt e\vece_{d-1}\Bigr\}\subset\R^{d-1},
\end{align}
for some $\alpha>4$.
To verify that $L'$ has the required property, we consider an
arbitrary non-zero lattice point $\vecell=(\ell_1,\ldots,\ell_{d-1})\in L'$
with $\ell_1\leq0$.
If $\ell_2\neq0$ then $\|\vecell\|\geq|\ell_2|\geq3$ and thus
$\|\vecell-\vecz\|,\|\vecell-\vecw\|>1$. %
On the other hand if $\ell_2=0$ and $\ell_1<0$ then
$\ell_1\leq-\alpha\frac e{\sqrt e+y}$
and this gives
\begin{align*}
\|\vecell-\vecz\|^2\geq(\ell_1-y)^2+z^2
\geq\Bigl(y+\alpha\frac e{\sqrt e+y}\Bigr)^2+z^2>y^2+z^2+\alpha e>(1-e)^2+4e>1.
\end{align*}
(We used the fact that $(y+\alpha\frac e{\sqrt e+y})^2>y^2+\alpha e$;
this is clear if $y\geq\sqrt e$, and if $y<\sqrt e$ then it follows from
$(\frac12\alpha\sqrt e)^2>\alpha e$.)
In the same way $\|\vecell-\vecw\|>1$.
Finally if $\ell_1=\ell_2=0$;
then since $\vecell\neq\bn$ we must have $d\geq4$ and
$\ell_j\neq0$ for some $j\geq3$;
then $|\ell_j|\geq2\sqrt e$ and thus
$\|\vecell-\vecz\|^2\geq y^2+z^2+4e\geq(1-e)^2+4e>1$
and similarly $\|\vecell-\vecw\|^2>1$.
Hence $L'$ has the required property, i.e.\ $L'$ leads to a lattice
$L\subset\R^d$ with $\vecp\in L$ and $L\cap\overline{\fC}=\{\bn,\vecp\}$.

Note that $\text{covol}(L)=(\xi+2\ve y)\text{covol}(L')$.
Hence \textit{if}
$\xi\frac e{\sqrt e+y}\cdot3\cdot(2\sqrt e)^{d-3}<\frac14$
then by appropriate choice of $\alpha>4$ and $\ve>0$ small we obtain
$\text{covol}(L)=1$, and thus $\Phi_\bn(\xi,\vecw,\vecz)>0$.
Also note that $\sqrt e+y\asymp\sqrt e+\pi-\varphi$;
indeed this is clear if $e\geq\frac1{10}$, and if $e<\frac1{10}$
it follows from $y\asymp\pi-\varphi$ which we proved above.
Combining these facts we conclude that there is a constant $c>0$
which only depends on 
$d$ such that $\Phi_\bn(\xi,\vecw,\vecz)>0$ holds
whenever $\varphi>\frac\pi2$ and 
$e<c\max(\xi^{-\frac2{d-2}},(\frac{\xi}{\pi-\varphi})^{-\frac2{d-1}})$.

We now turn to the remaining case, $\varphi\leq\frac\pi2$.
First assume that $\vecz\neq\vecw$ and that the triangle
$\triangle \bn\vecz\vecw$ is acute or right.
We then again write $\vecw,\vecz$ as in \eqref{CYLSUPPPROPPF1}, and choose
\begin{align}\label{EXPLLATTICE2}
L':=\text{Span}_\Z\Bigl\{\alpha e\vece_1,
5(\varphi+\sqrt e)\vece_2,2\sqrt e\vece_3,\ldots,2\sqrt e\vece_{d-1}\Bigr\}\subset\R^{d-1},
\end{align}
for some $\alpha>3$.
In order to prove that this $L'$ has the required property, we first note that
under the present assumptions we have
\begin{align}\label{CYLSUPPPROPPF4}
\sqrt{1-y^2}<2(\varphi+\sqrt e)\qquad\text{and}\qquad
y\geq2^{-\frac12}(1-e).
\end{align}
Indeed, if %
$|z|\leq2^{-\frac12}(1-y^2)^{\frac12}$
then the first inequality
follows from $1-y^2\leq1-(1-e)^2+z^2<2e+\frac12(1-y^2)$.
On the other hand if $|z|>2^{-\frac12}(1-y^2)^{\frac12}$ then
since $\triangle \bn\vecz\vecw$ is acute or right we have
$\varphi\geq\varphi(\vece_1,\vecz)$ and thus
$\tan\varphi\geq\frac{|z|}{y}>2^{-\frac12}(1-y^2)^{\frac12}$.
If $\varphi\leq\frac\pi4$ then we get
$\varphi\geq\frac{\pi}4\tan\varphi>\frac{\pi}{4}2^{-\frac12}(1-y^2)^{\frac12}
>\frac12(1-y^2)^{\frac12}$,
while if $\varphi>\frac\pi4$ then trivially
$(1-y^2)^{\frac12}\leq1<2\varphi$. Hence we have proved the
first inequality in \eqref{CYLSUPPPROPPF4}.
Next note that
$\varphi=\varphi(\vece_1,\vecz)+\varphi(\vece_1,\vecw)$,
since $\triangle \bn\vecz\vecw$ is acute or right.
Hence since $\varphi\leq\frac\pi2$, at least one of
$\varphi(\vece_1,\vecz)$ and $\varphi(\vece_1,\vecw)$ must be $\leq\frac\pi4$,
and now the second inequality in \eqref{CYLSUPPPROPPF4} follows since
$\cos\varphi(\vece_1,\vecz)=\frac y{\|\vecz\|}\leq\frac y{1-e}$ and 
similarly $\cos\varphi(\vece_1,\vecw)\leq\frac y{1-e}$.

Now consider an
arbitrary non-zero lattice point $\vecell=(\ell_1,\ldots,\ell_{d-1})\in L'$
with $\ell_1\leq0$.
If $\ell_2\neq0$ then $|\ell_2|\geq5(\varphi+\sqrt e)>\frac52\sqrt{1-y^2}$ 
and using $|z|<\sqrt{1-y^2}$ %
we get  %
$\|\vecell-\vecz\|^2\geq(\ell_1-y)^2+(\ell_2-z)^2
\geq y^2+(\frac32\sqrt{1-y^2})^2>1$. Similarly $\|\vecell-\vecw\|>1$.
On the other hand if $\ell_2=0$ and $\ell_1<0$ then
$\ell_1\leq-\alpha e$ and hence
\begin{align*}
\|\vecell-\vecz\|^2\geq(\alpha e+y)^2+z^2
\geq(1-e)^2+2\alpha ey+\alpha^2e^2
>(1-e)^2+2e>1,
\end{align*}
where $2\alpha ey+\alpha^2e^2>2e$ holds since
either $e\geq\frac12$ or $y>\frac13$, by \eqref{CYLSUPPPROPPF4}.
Similarly $\|\vecell-\vecw\|>1$.
Finally if $\ell_1=\ell_2=0$ then
$\|\vecell-\vecz\|>1$ and $\|\vecell-\vecw\|>1$ just as for
\eqref{EXPLLATTICE1}.
Hence $L'$ has the required property.

On the other hand if $\triangle \bn\vecz\vecw$ is obtuse
or $\vecz=\vecw$
then we choose coordinates as follows.
Using the symmetry \eqref{PHI0SYMM} we may assume that the
angle at $\vecz$ is $>\frac\pi2$,
and we rotate $\fC$ so that
\begin{align*}
\vecw=y'\vece_1+w\vece_2;\qquad\vecz=y\vece_1,
\end{align*}
for some $y'\geq y\geq0$, $w\in\R$.
Now $e=1-y$, and also $|w|<\sqrt{1-{y'}^2}\leq\sqrt{2(1-y')}\leq\sqrt{2e}$.
Note that \eqref{CYLSUPPPROPPF3} remains true for our present
$\vecp$, so long as $\ve<\frac12\xi$.
Now choose
\begin{align}\label{EXPLLATTICE3}
L':=\text{Span}_\Z\Bigl\{\alpha e\vece_1,
4\sqrt e\vece_2,2\sqrt e\vece_3,\ldots,2\sqrt e\vece_{d-1}\Bigr\}
\subset\R^{d-1},
\end{align}
where $\alpha>1$. Consider an
arbitrary non-zero lattice point $\vecell=(\ell_1,\ldots,\ell_{d-1})\in L'$
with $\ell_1\leq0$.
If $\ell_2\neq0$ then $|\ell_2|\geq4\sqrt e$
and thus 
$|\ell_2-w|\geq(4-\sqrt2)\sqrt e>2\sqrt e$ and
$\|\vecell-\vecw\|^2>{y'}^2+(\ell_2-w)^2\geq(1-e)^2+4e>1$.
Also $\|\vecell-\vecz\|^2\geq(1-e)^2+16e>1$.
On the other hand if $\ell_2=0$ and $\ell_1<0$ then
$\ell_1\leq-\alpha e$ and hence
$y'-\ell_1\geq y-\ell_1\geq y+\alpha e=1+(\alpha-1)e>1$,
thus $\|\vecell-\vecz\|>1$ and $\|\vecell-\vecw\|>1$.
Finally if $\ell_1=\ell_2=0$ then
$\|\vecell-\vecz\|>1$ and $\|\vecell-\vecw\|>1$ as before.
Hence $L'$ has the required property.

Note that both for \eqref{EXPLLATTICE2} and \eqref{EXPLLATTICE3} we have
$\text{covol}(L')\ll\alpha e^{\frac{d-1}2}(\varphi+\sqrt e)$.
Hence, arguing as before, we conclude that there is a constant $c>0$ which
only depends on $d$ such that $\Phi_\bn(\xi,\vecw,\vecz)>0$ holds
whenever $\varphi\leq\frac\pi2$ and
$e<c\min(\xi^{-\frac2d},(\varphi \xi)^{-\frac2{d-1}})$.

This concludes the proof of Proposition \ref{CYLSUPPPROP}.
\hfill$\square$ $\square$ $\square$

\subsection{\texorpdfstring{Bounding $\Phi_\bn(\xi,\vecw,\vecz)$ from above}{Bounding Phi0(xi,w,z) from above}}
\label{CYLBDSEC}

We will now prove Theorem \ref{CYLINDER2PTSMAINTHM}, viz.\ the bound
(for $d\geq3$, and with $\varphi=\varphi(\vecw,\vecz)$)
\begin{align}\label{CYLINDER2PTSMAINTHMRESRESTATE}
\Phi_\bn(\xi,\vecw,\vecz)\ll\begin{cases}
\xi^{-2+\frac2d}\min\Bigl\{1,
(\xi\varphi^{d})^{-1+\frac2{d(d-1)}}\Bigr\}
&\text{if }\:\varphi\leq\frac\pi2
\\
\xi^{-2}\min\Bigl\{1,
(\xi(\pi-\varphi)^{d-2})^{-1+\frac2{d-1}}\Bigr\}
&\text{if }\:\varphi\geq\frac\pi2.
\end{cases}
\end{align}
The main idea of the proof is to use the parametrization from Section 
\ref{X1RECOLLECTSEC} and try to carefully bound the integrals arising,
just as in the proof of Theorem \ref{CONEAPEXTHM} in 
Section \ref{PPCBOUNDCCONESEC}.

Given $\xi,\vecw,\vecz$ we may use \eqref{PPDINV}
to see that $\Phi_\bn(\xi,\vecw,\vecz)=p_\vecp(\fC)$,
where $\fC$ is a cylinder which has radius and height both equal
to $r=\xi^{1/d}$.
More specifically we may take $\fC$ as the interior of the 
convex hull of $B_1$ and $B_2$, where
\begin{align}\label{CYLINDER2PTSMAINTHMPF1}
B_1=\iota(r\vecz+\scrB_r^{d-1}),
\quad
B_2=r\vece_1+B_1,\quad
\text{and }\:
\vecp=r(1,\vecw+\vecz).
\end{align}
We might also
assume $\vecz=z\vece_1$, $\vecw=w(\cos\varphi)\vece_1+w(\sin\varphi)\vece_2$
for some $z,w\in[0,1)$.
However in order to make the symmetry between $\vecz$ and $\vecw$
somewhat more explicit in our arguments, and also make the lemmas in 
the present section and the next more directly applicable
when we will later derive an asymptotic formula for
$\Phi_\bn(\xi,\vecw,\vecz)$ in \cite{partIV},
we will just assume that the points $\vecz,\vecw\in\scrB_1^{d-1}$
satisfy the following:
\begin{alignat}{2}\label{CYLINDER2PTSMAINTHMPF2alt}
\begin{cases}
\text{if }\:\varphi\leq\sfrac\pi2:&
\varphi(\vece_1,\vecz),\varphi(\vece_1,\vecw)\leq\varphi;
\\
\text{if }\:\varphi\geq\sfrac\pi2:&
\varphi(\vece_1,\vecz)\leq\pi-\varphi,\quad
\varphi(-\vece_1,\vecw)\leq\pi-\varphi.
\end{cases}
\end{alignat}
(We tacitly assume $\vecz,\vecw\neq\bn$; this is ok in view
of Proposition \ref{CYLSUPPPROP}.)
We may furthermore assume that $r=\xi^{1/d}\geq1$, since otherwise
the bound \eqref{CYLINDER2PTSMAINTHMRESRESTATE} is trivial.

It is convenient, as preparation for the next section,
to define  %
\begin{align*}
\Si_d':=\bigl\{[a_1,\vecv,\vecu,\tM]\in\Si_d\col\vecv\cdot\vece_1\geq0\bigr\}.
\end{align*}
Note that if $M=\nn(u)\aa(a)\kk$ lies in $\Si_d$ but not in $\Si_d'$, then
if we set $D=\text{diag}[-1,-1,1,\ldots,1]\in\Gamma$
and take $\gamma\in\Gamma\cap N$ so that $\gamma D\nn(u)D\in\F_N$,
we obtain $\gamma DM\in\Si_d'$,
since $\gamma DM$ has the Iwasawa decomposition
$\gamma DM=(\gamma D\nn(u)D)\aa(a)D\kk$,
and $\vece_1 D\kk=-\vecv$.
Hence since $\Si_d$ contains a fundamental domain for $\Gamma\backslash G$,
it follows that also the subset $\Si_d'$ contains a fundamental 
domain for $\Gamma\backslash G$.
Now by the same argument as in Section~\ref{PPCBOUNDCCONESEC} we have
\begin{align}\label{PHI0FIRSTSPLIT}
\Phi_\bn(\xi,\vecw,\vecz)=p_\vecp(\fC)
\leq\sum_{\veck\in\widehat\Z^d}
\nu_\vecp\bigl(\bigl\{
M\in G_{\veck,\vecp}\cap\Si_d'\col \Z^dM\cap\fC=\emptyset\bigr\}\bigr)
={\sum}_0+{\sum}_1,
\end{align}
where ${\sum}_0$ and ${\sum}_1$ are the sums corresponding to $k_1=0$ and 
$k_1\neq0$, respectively,
and there is a positive constant $C$ which only depends on $d$ such that
$a_1>Cr$ holds for all
$M=[a_1,\vecv,\vecu,\tM]\in\Si_d'$ satisfying $\Z^dM\cap\fC=\emptyset$.
We first consider ${\sum}_1$.
Note that for $M=[\vecv,\tM]_{\veck,\vecp}\in G_{\veck,\vecp}$ we have,
using \eqref{LATTICEINPARAM} and $\veck M=\vecp$,
\begin{align*}
\iota(\vecm)M=a_1^{-\frac1{d-1}}\iota(\vecm\tM)f(\vecv)
\quad\text{and}\quad
(\veck+\iota(\vecm))M=\vecp+a_1^{-\frac1{d-1}}\iota(\vecm\tM)f(\vecv),
\quad\forall\vecm\in\Z^{d-1}.
\end{align*}
Hence $\Z^dM\cap\fC=\emptyset$ implies that the $(d-1)$-dimensional lattice
$a_1^{-\frac1{d-1}}\iota(\Z^{d-1}\tM)f(\vecv)$ is disjoint from both
$\fC$ and from $\fC':=\vecp-\fC$. 
Note here that $\fC'$ is the cylinder which is the interior of the
convex hull of $\iota(r\vecw+B_r^{d-1})$ and 
$r\vece_1+\iota(r\vecw+B_r^{d-1})$.
Mimicking now the argument in Section \ref{PPCBOUNDCCONESEC} leading up to
\eqref{SUM1BOUNDFIRSTTIMEEVER}, we get
\begin{align}\label{CYLINDER2PTSMAINTHMPF9}
{\sum}_1\ll
\sum_{\substack{|k_1|<2/C\\k_1\neq0}}
|k_1|^{d-1}\int_{S}
\mu\Bigl(\Bigl\{\tM\in\Si_{d-1}\col
\Z^{d-1}\tM\cap a_1^{\frac1{d-1}}(\fC_\vecv\cup\fC'_\vecv)=\emptyset
\Bigr\}\Bigr)
\,\frac{d\vecv}{(|k_1|r)^d},
\end{align}
where $\fC_\vecv=\iota^{-1}(\fC f(\vecv)^{-1})$ (as usual),
$\fC'_\vecv=\iota^{-1}(\fC' f(\vecv)^{-1})$,
$a_1=k_1^{-1}(\vecp\cdot\vecv)$, and 
\begin{align}\label{SK1DEF}
S=S^{(k_1)}=\{\vecv\in\S_1^{d-1}\col v_1>0,\:k_1^{-1}(\vecp\cdot\vecv)>Cr\}.
\end{align}

Recall that $\fC_\vecv$ is isometric with $\fC\cap\vecv^\perp$.
The following lemma gives a precise description of a 
certain $(d-1)$-dimensional cone contained in $\fC\cap\vecv^\perp$;
the point is that this will allow us to apply Corollary \ref{CONECYLCOR} 
to obtain a bound on the 
integrand in the right hand side of \eqref{CYLINDER2PTSMAINTHMPF9}.
For $\vecv=(v_1,\ldots,v_d)$ we write $\vecv':=(v_2,\ldots,v_d)$ and
\begin{align}\label{OMEGAZWDEF}
\omega_\vecz:=\varphi(\vecv',\vecz)\qquad\text{and}\quad
\omega_\vecw:=\varphi(\vecv',\vecw).
\end{align}

\begin{lem}\label{CONVENIENTCONELEM}
For any $\vecv\in\S_1^{d-1}$ with $0<v_1<1$,
the intersection $\fC\cap\vecv^\perp$ contains a right relatively open 
$(d-1)$-dimensional cone with $\bn$ in its base,
of height $\asymp r\min(1,\frac{1-\|\vecz\|+\omega_\vecz^2}{v_1})
\gg r\omega_\vecz^2$,
radius
$\asymp r(1-\|\vecz\|+\sin^2\omega_\vecz)^{\frac12}\gg r\sin\omega_\vecz$
(thus volume $\gg r^{d-1}\omega_\vecz^2(\sin\omega_\vecz)^{d-2}$),
and edge ratio $\asymp\min(1,\frac{1-\|\vecz\|}{\sin^2\omega_\vecz})$.
\end{lem}
(We say that a cone is ``right'' if the line between its apex and the
center of its base is orthogonal to the base.
In the present section we will only use the fact that the cone has volume
$\gg r^{d-1}(\sin\omega_\vecz)^{d}$; the more detailed information in the
lemma will be used later.)
\begin{proof}
After a rotation inside $\iota(\R^{d-1})$ we may assume,
for the proof of the present lemma,
that $\vecz=z\vece_1$ ($0<z<1$),
and thus $\omega_\vecz=\varphi(\vecv',\vece_1)$.
Now $B_1\cap\vecv^\perp$ is a $(d-2)$-dimensional ball with center
\begin{align}\label{QCENTERDEF}
\vecq=(q_1,\ldots,q_d)=
r\cdot\iota\Bigl(\vecz-\frac{z\cos\omega_\vecz}{\|\vecv'\|}\vecv'\Bigr)
\end{align}
and radius $r'=r\sqrt{1-z^2\cos^2\omega_\vecz}\asymp 
r(1-z+\sin^2\omega_\vecz)^{\frac12}\gg r\sin\omega_\vecz$.
Note that the point $\vecq+\vech$ lies in $\fC$, where
\begin{align}\label{HFIRSTDEF}
\vech=(h_1,\ldots,h_d)
:=r\delta\bigl(\|\vecv'\|^2\vece_1-v_1\iota(\vecv')\bigr),
\qquad\text{with }\:
\delta
=\min\Bigl(\frac1{2\|\vecv'\|^2},\frac{1-z|\cos\omega_\vecz|}{2v_1\|\vecv'\|}
\Bigr).
\end{align}
Indeed $0<h_1<r$, and $\vecq'=(q_2,\ldots,q_d)$ 
and $\vech'=(h_2,\ldots,h_d)$ satisfy
\begin{align*}
\bigl\|r\vecz-(\vecq'+\vech')\bigr\|
=r\Bigl\|\frac{z\cos\omega_\vecz+\delta v_1\|\vecv'\|}{\|\vecv'\|}\vecv'\Bigr\|
\leq r\Bigl(z|\cos\omega_\vecz|+\sfrac12(1-z|\cos\omega_\vecz|)\Bigr)<r;
\end{align*}
hence indeed $\vecq+\vech\in\fC$.
Note also $\vecq,\vech\in\vecv^\perp$.
Hence $\fC\cap\vecv^\perp$ contains the $(d-1)$-dimensional
cone which is the relative
interior of the convex hull of $B_1\cap\vecv^\perp$ and $\vecq+\vech$.
Also $\vech$ is orthogonal to $B_1\cap\vecv^\perp$
(since $\vech\in\text{Span}\{\vece_1,\vecv\}$),
thus the height of this cone is 
$\|\vech\|=r\delta\|\vecv'\|\asymp 
r\min(1,\frac{1-z+\sin^2\omega_\vecz}{v_1})$,
where we used the fact that if
$\frac1{2\|\vecv'\|^2}<\frac{1-z|\cos\omega_\vecz|}{2v_1\|\vecv'\|}$
then %
$\|\vecv'\|\asymp1$.

If $\omega_\vecz>\frac\pi2$ then we may replace $\vech$ in %
the above argument by
\begin{align}\label{HINCYLINDER}
\vech=(h_1,\ldots,h_d)
:=r\|\vecv'\|^{-1}\bigl(\|\vecv'\|^2\vece_1-v_1\iota(\vecv')\bigr).
\end{align}
We still have $0<h_1<r$, and 
$r\vecz-(\vecq'+\vech')=r\frac{z\cos\omega_\vecz+v_1}{\|\vecv'\|}\vecv'$
has length $<r$ since $\cos\omega_\vecz<0$; %
hence $\vecq+\vech\in\fC$ just as before.
With this choice the cone has radius $r'=r\sqrt{1-z^2\cos^2\omega_\vecz}$ 
as before but height $\|\vech\|=r$.

Finally the edge ratio of the cone is, in both cases,
$\frac{r'-rz\sin\omega_\vecz}{r'}
\asymp\min(1,\frac{1-z}{\sin^2\omega_\vecz})$.
\end{proof}

In the same way we have that $\fC'\cap\vecv^\perp$ contains a 
$(d-1)$-dimensional cone of volume $\gg r^{d-1}(\sin\omega_\vecw)^d$, 
with $\bn$ in its base.
Hence by Corollary \ref{CONECYLCOR}, and using the fact that
$\Si_{d-1}$ can be covered by a finite number of fundamental regions
for $\Gamma^{(d-1)}\backslash G^{(d-1)}$, we get
\begin{align}\label{CYLINDER2PTSMAINTHMPF5}
\mu\Bigl(\Bigl\{\tM\in\Si_{d-1}\col
\Z^{d-1}\tM\cap a_1^{\frac1{d-1}}(\fC_\vecv\cup\fC'_\vecv)=\emptyset
\Bigr\}\Bigr)
\ll\Bigl(r\max\bigl(\sin\omega_\vecz,\sin\omega_\vecw\bigr)
\Bigr)^{-\frac{2d(d-2)}{d-1}}
\end{align}
whenever $a_1>Cr$.

Set $\varphi_0:=\min(\varphi,\pi-\varphi)$.
This is the distance between $\|\vecz\|^{-1}\vecz$ and 
$\|\vecw\|^{-1}\vecw$ inside $\S_1^{d-2}/\{\pm\}$
with its quotient metric from $\S_1^{d-2}$.  %
Similarly the distance between $\|\vecz\|^{-1}\vecz$ and 
$\|\vecv'\|^{-1}\vecv'$ is $\min(\omega_\vecz,\pi-\omega_\vecz)$ 
and the distance between $\|\vecv'\|^{-1}\vecv'$ and
$\|\vecw\|^{-1}\vecw$ is 
$\min(\omega_\vecw,\pi-\omega_\vecw)$.
Hence by the triangle inequality at least one of
$\min(\omega_\vecz,\pi-\omega_\vecz)$ and $\min(\omega_\vecw,\pi-\omega_\vecw)$
must be $\geq\frac12\varphi_0$, and thus
$\max(\sin\omega_\vecz,\sin\omega_\vecw)\geq\sin\frac12\varphi_0$
for all $\vecv\in\S_1^{d-1}$ with $\vecv'\neq\bn$.
Note also that if we parametrize $\vecv$ as in \eqref{VPARA} then
$|\omega-\omega_\vecz|\leq\varphi_0$,
by \eqref{CYLINDER2PTSMAINTHMPF2alt} and the triangle inequality 
in $\S_1^{d-2}$. 
Using this latter fact whenever 
$2\varphi_0\leq\omega\leq\pi-2\varphi_0$
we obtain, for all $\vecv\in\S_1^{d-1}$ with $\vecv'\neq\bn$,
\begin{align}\label{MAXSOZWINEQ}
\max(\sin\omega_\vecz,\sin\omega_\vecw)\gg
\max(\sin\omega,\varphi_0).
\end{align}

Hence
\begin{align}\label{CYLINDER2PTSMAINTHMPF6}
{\sum}_1\ll r^{-d}\int_0^\pi
\min\Bigl\{1,\Bigl(r\max\bigl(\sin\omega,\varphi_0
\bigr)\Bigr)^{-\frac{2d(d-2)}{d-1}}\Bigr\}
(\sin\omega)^{d-3}\,d\omega.
\end{align}
The integrand is invariant under
$\omega\leftrightarrow\pi-\omega$, and furthermore the part of
the integral where $\omega\in(0,\frac12\varphi_0)$
is bounded above (up to a constant which only depends on $d$)
by the part where $\omega\in(\frac12\varphi_0,\varphi_0)$.
Hence
\begin{align}\label{CYLINDER2PTSMAINTHMPF3}
{\sum}_1\ll r^{-d}\int_{\varphi_0/2}^{\pi/2}
\min\bigl\{1,(r\omega)^{-\frac{2d(d-2)}{d-1}}\bigr\}\,
\omega^{d-3}\,d\omega\ll
\min\Bigl\{r^{2-2d},r^{\frac{2d}{d-1}-3d}\varphi_0^{-d+\frac2{d-1}}\Bigr\}.
\end{align}

In the case when $\varphi$ is near $\pi$ we can do better, and in fact we have
\begin{align}\label{CYLINDER2PTSMAINTHMPF4}
{\sum}_1\ll r^{\frac{2d}{d-1}-3d}\qquad
\text{whenever }\:\varphi\geq\frac\pi2.
\end{align}
To prove this, first note that $\vecv\in S$ implies
$|\vecp\cdot\vecv|>Cr$ and therefore since $\|\vecp\|<3r$
there exists a constant $c_1>0$ which only depends on $d$
such that for all $k_1\neq0$ and all $\vecv\in S=S^{(k_1)}$ 
we have $|\varphi(\vecp,\vecv)-\frac\pi2|>2c_1$.

Now note that \eqref{CYLINDER2PTSMAINTHMPF4} follows
from \eqref{CYLINDER2PTSMAINTHMPF3} unless $r$ is large and 
$\varphi_0=\pi-\varphi$ is small.
Furthermore we may assume that both $\|\vecw\|$ and $\|\vecz\|$ are near $1$,
since otherwise $\Phi_\bn(\xi,\vecw,\vecz)=0$ 
by Proposition \ref{CYLSUPPPROP}.
In particular, since $\vecp=r(1,\vecw+\vecz)$, 
we may assume that $\varphi(\vecp,\vece_1)<c_1$ holds.
This implies $|\varphi(\vecp,\vecv)-\varphi(\vece_1,\vecv)|<c_1$
by the triangle inequality for the geodesic metric on $\S_1^{d-1}$.
It follows that $|\varphi(\vece_1,\vecv)-\frac\pi2|>c_1$
(viz.\ $|\varpi-\frac\pi2|>c_1$) 
holds for all $\vecv\in S$.
But we also have $v_1=\cos\varpi\geq0$ for $\vecv\in S$;
hence in fact $0\leq\varpi<\frac\pi2-c_1$.

First consider $\vecv\in S$ with $\omega\in[\frac\pi2,\pi)$.
Writing $p:\R^d\to\R^{d-1}$ for the projection
$(x_1,\ldots,x_d)\mapsto(x_2,\ldots,x_d)$, we note that
\begin{align*}
p(\fC\cap\vecv^\perp)=\bigl\{\vecx\in r\vecz+\scrB_r^{d-1}\col
0<-v_1^{-1}(\vecx\cdot\vecv')<r\bigr\}.
\end{align*}
Recalling \eqref{VPARA} and \eqref{OMEGAZWDEF}
we see that ${p}(\fC\cap\vecv^\perp)$ is a ``doubly cut ball'' as in
Corollary \ref{DCBCOR}, of radius $r$ and with cut parameters
$t=1+\|\vecz\|\cos\omega_\vecz$ and $t'=\min(2,t+\cot\varpi)$.
(Thus if $\cot\varpi\geq1-\|\vecz\|\cos\omega_\vecz$ then 
${p}(\fC\cap\vecv^\perp)$ is a cut ball with cut ratio $t=t_1$.)
Note that $p_{|\vecv^\perp}$ is a linear map
$\vecv^\perp\to\R^{d-1}$ which scales volume with a factor
$\cos\varpi\asymp1$; thus $\fC_\vecv$ can be mapped by
a map in $G^{(d-1)}$ to a doubly cut ball of radius $\asymp r$
and cut parameters $t$ and $t'$ as above.
Since $\omega\in[\frac\pi2,\pi)$ and
$|\omega_\vecz-\omega|\leq\varphi_0$ (as noted previously),
and we are assuming $\varphi_0$ to be small,
we may assume that we always have $\omega_\vecz>\frac13\pi$.
Thus $t<\frac32$.
We also have $t'-t\gg1$, since $\varpi<\frac\pi2-c_1$.
Hence by Corollary \ref{DCBCOR},
\begin{align}\label{CYLINDER2PTSMAINTHMPF10}
\mu\Bigl(\Bigl\{\tM\in\Si_{d-1}\col
\Z^{d-1}\tM\cap a_1^{\frac1{d-1}}\fC_\vecv=\emptyset
\Bigr\}\Bigr)\ll r^{-2d+\frac{2d}{d-1}}
\end{align}
whenever $a_1>Cr$.

On the other hand if $\omega\in(0,\frac\pi2)$ then we instead consider
$\fC'_\vecv$.
Similarly as above we find that ${p}(\fC'\cap\vecv^\perp)$ is
a doubly cut ball of radius $r$ and with cut parameters 
$t=1+\|\vecw\|\cos\omega_\vecw$ and $t'=\max(2,t+\cot\varpi)$.
Now $|\omega_\vecw-(\pi-\omega)|=
|\varphi(\vecv',\vecw)-\varphi(\vecv',-\vece_1)|\leq\varphi(\vecw,-\vece_1)
\leq\varphi_0$, by \eqref{CYLINDER2PTSMAINTHMPF2alt},
and thus we may assume that we always have $\omega_\vecw>\frac13\pi$.
We now obtain, as before,
\begin{align}\label{CYLINDER2PTSMAINTHMPF11}
\mu\Bigl(\Bigl\{\tM\in\Si_{d-1}\col
\Z^{d-1}\tM\cap a_1^{\frac1{d-1}}\fC'_\vecv=\emptyset
\Bigr\}\Bigr)\ll r^{-2d+\frac{2d}{d-1}}
\end{align}
whenever $a_1>Cr$.
Using \eqref{CYLINDER2PTSMAINTHMPF10} and \eqref{CYLINDER2PTSMAINTHMPF11}
in \eqref{CYLINDER2PTSMAINTHMPF9} we obtain 
\eqref{CYLINDER2PTSMAINTHMPF4}.

\vspace{5pt}

We next turn to $\sum_0$.
Mimicking the argument in Section \ref{PPCBOUNDCCONESEC} we get
\begin{align}\label{CYLINDER2PTSMAINTHMPF7}
{\sum}_0\ll r^{-1}\int_{Cr}^\infty\int_{\S_1^{d-1}\cap\vecp^\perp}
p_\vecy^{(d-1)}\bigl(a_1^{\frac1{d-1}}\fC_\vecv\bigr)
\, d_\vecp(\vecv)\,a_1^{-d}\,da_1
\end{align}
where $\vecy=\vecy(a_1,\vecv)=a_1^{\frac1{d-1}}\iota^{-1}(\vecp f(\vecv)^{-1})$
as in \eqref{RA1V}.
Here for any fixed $a_1,\vecv$ we have as in Section \ref{PPCBOUNDCCONESEC},
using the fact that $\fC\cap\vecv^\perp$ contains the
$(d-1)$-dimensional relatively open cone 
with base $B_1\cap\vecv^\perp$ and apex $\vecp$,
\begin{align}\label{CYLINDER2PTSMAINTHMPF12}
p_\vecy^{(d-1)}\bigl(a_1^{\frac1{d-1}}\fC_\vecv\bigr)
\ll %
\bigl(r^d(\sin\omega_\vecz)^{d-2}\bigr)^{-2+\frac2{d-1}}. %
\end{align}
But note \textit{also} that $\fC\cap\vecv^\perp$ contains the 
$(d-1)$-dimensional relatively open cone
with base $B_2\cap\vecv^\perp$ and apex $\bn$.
This cone has volume $\gg r^{d-1}(\sin\omega_\vecw)^{d-2}$.
Hence (since $a_1>Cr$) $a_1^{\frac1{d-1}}\fC_\vecv\subset\R^{d-1}$ contains an 
open cone of volume $\gg r^d(\sin\omega_\vecw)^{d-2}$
with apex $\bn$ and with $\vecy$ in its base.
Hence by Theorem \ref{CONEAPEXTHM} (and using the general symmetry relation
$p_\vecy(\vecy-\fZ)=p_\vecy(\fZ)$),
\begin{align}\label{CYLINDER2PTSMAINTHMPF13}
p_\vecy^{(d-1)}\bigl(a_1^{\frac1{d-1}}\fC_\vecv\bigr)
\ll %
\bigl(r^d(\sin\omega_\vecw)^{d-2}\bigr)^{-2+\frac2{d-1}}. %
\end{align}
Note that the map $\vecv\mapsto\|\vecv'\|^{-1}\vecv'$ is a diffeomorphism from 
$\S_1^{d-1}\cap\vecp^\perp$ onto $\S_1^{d-2}$ 
whose Jacobian determinant is uniformly bounded from below by a positive
constant which only depends on $d$, 
since $\varphi(\vecp,\vece_1)<\arctan 2<\pi$.
Hence, mimicking the discussion around
\eqref{CYLINDER2PTSMAINTHMPF5}---\eqref{CYLINDER2PTSMAINTHMPF3},
\begin{align}\label{CYLINDER2PTSMAINTHMPF8}
{\sum}_0\ll r^{-d}\int_{\varphi_0/2}^{\pi/2}
\min\Bigl\{1,
\bigl(r^d\omega^{d-2}\bigr)^{-2+\frac2{d-1}}
\Bigr\}\,\omega^{d-3}\,d\omega
\ll\begin{cases}
r^{-6}\log(2+\min(r,\varphi_0^{-1}))&\text{if }\: d=3
\\
r^{-2d}\min(1,(r^d\varphi_0^{d-2})^{-\frac{d-3}{d-1}})&\text{if }\: d\geq4.
\end{cases}
\end{align}
Let us note that this bound can be slightly improved in the case
$d=3$, $\varphi>\pi-\frac1{10}$, by using the fact that
for any fixed $a_1,\vecv$ appearing in \eqref{CYLINDER2PTSMAINTHMPF7},
$\fC\cap\vecv^\perp$ contains the relative interior of the convex hull of
$B_1\cap\vecv^\perp$ and $B_2\cap\vecv^\perp$.
If $d=3$ then this convex hull is a quadrilateral with two parallel sides
which have distance $\geq r$ from each other, and 
lengths $2r\sqrt{1-\|\vecz\|^2\cos^2\omega_\vecz}\geq2r\sin\omega_\vecz$
and $2\sqrt{1-\|\vecw\|^2\cos^2\omega_\vecw}\geq2r\sin\omega_\vecw$.
The following lemma is a simple consequence of the explicit formula
for $\Phi_\bn$ in dimension 2.

\begin{figure}
\begin{center}
\framebox{
\begin{minipage}{0.55\textwidth}
\unitlength0.1\textwidth
\begin{picture}(10.0,4.9)(0,0)
\put(2,0.4){\includegraphics[width=0.55\textwidth]{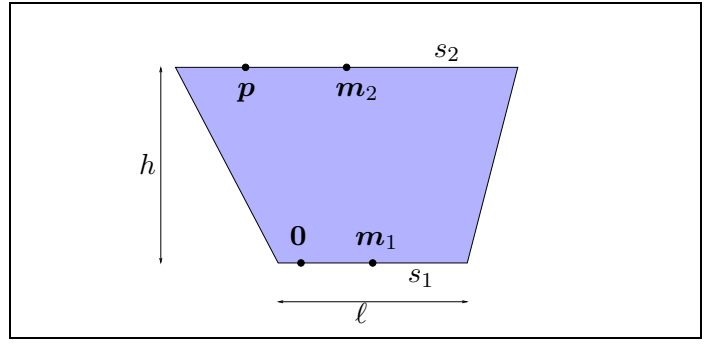}}
\put(1.7,2.4){$h$}
\put(5.0,0.1){$\ell$}
\put(5.8,0.73){$s_1$}
\put(5.0,1.3){$\vecm_1$}
\put(4.0,1.3){$\bn$}
\put(6.2,4.2){$s_2$}
\put(4.7,3.6){$\vecm_2$}
\put(3.2,3.6){$\vecp$}
\end{picture}
\end{minipage}
}
\end{center}
\caption{The quadrilateral in Lemma \ref{TRAPEZOIDLEM}}\label{TRAPEZOIDLEMFIG}
\end{figure}

\begin{lem}\label{TRAPEZOIDLEM}
Let $\fC\subset\R^2$ be the interior of a quadrilateral with two 
parallel sides $s_1,s_2$, such that $\bn\in s_1$.
Let $\vecp$ be a point on $s_2$ such that
$\vecm_1\cdot(\vecm_2-\vecp)\geq0$, where
$\vecm_1,\vecm_2$ are the respective midpoints of $s_1,s_2$.
Assume furthermore that $h\ell>2$, where $h$ is the height of $\fC$
(viz.\ the distance between the two lines containing $s_1$ and $s_2$)
and $\ell$ is the length of the shortest of the two sides $s_1$, $s_2$.
Then $p_\vecp(\fC)=0$.
\end{lem}
\begin{proof}
Because of the symmetry relation $p_\vecp(\fC)=p_\vecp(\vecp-\fC)$
we may without loss of generality assume that $\ell$ is the length of $s_1$.
Now there is a subsegment $s_2'\subset s_2$ of length $\ell$ and with 
midpoint $\vecm_2'$ such that
$\vecm_2-\vecp$ and $\vecm_2'-\vecp$ have the same direction (or one is zero),
and thus $\vecm_1\cdot(\vecm_2'-\vecp)\geq0$.
Let $\fC'$ be the parallelogram which is the interior of the
convex hull of  $s_1$ and $s_2'$.
Then $\fC'\subset\fC$ and thus $p_\vecp(\fC)\leq p_\vecp(\fC')$.
By \eqref{PPDINV}   %
we have
$p_\vecp(\fC')=\Phi_\bn(\frac12|\fC'|,w,z)$ for some
$w,z\in[-1,1]$ with $wz\leq0$. %
Also $|\fC'|=h\ell>2$.
Now $p_\vecp(\fC)=p_\vecp(\fC')=0$ follows from the fact that
$\Phi_\bn(\xi,w,z)=0$ whenever $\xi>1$ and $wz\leq0$,
by \eqref{PF2DIMGENERIC}.
\end{proof}

Now to carry through our argument for $d=3$, $\varphi>\pi-\frac1{10}$
it is convenient to make the specific choices
$\vecz=(z,0)$, $\vecw=w(\cos\varphi,\sin\varphi)$
($0<z,w<1$).\label{SPECD2COORDS}
We parametrize $\vecv$ 
in \eqref{CYLINDER2PTSMAINTHMPF7} by writing
$\tvecv':=\|\vecv'\|^{-1}\vecv'=(\cos\omega,\sin\omega)$.
Since $(-\vecv)^\perp=\vecv^\perp$ we need only consider
$0<\omega<\pi$; thus $\omega_\vecz=\omega$.
Now if $0<\omega<\varphi$ then $\omega_\vecw=\varphi-\omega$ and
the midpoints of the two line segments $B_1\cap\vecv^\perp$
and $B_2\cap\vecv^\perp$ are
\begin{align*}
\vecm_1=r\bigl(0,\vecz-z(\cos\omega_\vecz)\tvecv'\bigr)
\quad\text{and}\quad
\vecm_2=r\bigl(1,\vecz+w(\cos\omega_\vecw)\tvecv'\bigr);
\end{align*}
and thus
\begin{align*}
r^{-2}\vecm_1\cdot(\vecm_2-\vecp)
=(\vecz-z(\cos\omega_\vecz)\tvecv')\cdot
(-\vecw+w(\cos\omega_\vecw)\tvecv')
=zw\sin\omega\sin(\varphi-\omega)>0.
\end{align*}
Hence by Lemma \ref{TRAPEZOIDLEM},
we have $p_\vecy^{(2)}(a_1^{1/2}\fC_\vecv)=0$ for all
$a_1>Cr$ and all $\omega\in(0,\varphi)$ for which
$r^3\min(\sin\omega,\sin(\varphi-\omega))>C^{-1}$.
Using $p_\vecy^{(2)}(a_1^{1/2}\fC_\vecv)\leq1$ for all remaining
$\omega\in(0,\varphi)$, and
$p_\vecy^{(2)}(a_1^{1/2}\fC_\vecv)\leq r^{-3}\varphi_0^{-1}$
for all $\omega\in[\varphi,\pi)$
(cf.\ \eqref{CYLINDER2PTSMAINTHMPF12} and \eqref{CYLINDER2PTSMAINTHMPF13}
and note that $\omega_\vecw=\omega-\varphi$ when $\omega\in[\varphi,\pi)$)
we get in \eqref{CYLINDER2PTSMAINTHMPF7}:
\begin{align}\label{CYLINDER2PTSMAINTHMPF8IMPRD3} 
{\sum}_0\ll r^{-3}\bigl(r^{-3}+\varphi_0(r^{-3}\varphi_0^{-1})\bigr)\ll r^{-6}.
\end{align}

Adding the bound \eqref{CYLINDER2PTSMAINTHMPF3}
(improved as in \eqref{CYLINDER2PTSMAINTHMPF4} when $\varphi\geq\frac\pi2$)
and the bound \eqref{CYLINDER2PTSMAINTHMPF8} 
(improved as in \eqref{CYLINDER2PTSMAINTHMPF8IMPRD3} when $d=3$ and 
$\varphi\geq\frac\pi2$),
we finally obtain \eqref{CYLINDER2PTSMAINTHMRES}.
This completes the proof of Theorem \ref{CYLINDER2PTSMAINTHM}.
\hfill $\square$ $\square$ $\square$

\subsection{\texorpdfstring{A better bound on the contribution from $k_1\neq1$ for $\varphi$ small}{A better bound on the contribution from k1 not equal 1 for phi small}}

Let us fix a fundamental domain $\F_d\subset\Si_d'$ for $\Gamma\backslash G$.
Replacing $\Si_d'$ with $\F_d$ in \eqref{PHI0FIRSTSPLIT},
and arguing as before \eqref{CONEAPEXTHMPF3},
we get an \textit{equality}
\begin{align}\label{PHI0FIRSTSPLIT2}
\Phi_\bn(\xi,\vecw,\vecz)=
\sum_{\veck\in\widehat\Z^d}
\nu_\vecp\bigl(\bigl\{
M\in G_{\veck,\vecp}\cap\F_d\col \Z^dM\cap\fC=\emptyset\bigr\}\bigr).
\end{align}
To prepare for the derivation in \cite{partIV} of an asymptotic formula
for $\Phi_\bn(\xi,\vecw,\vecz)$ for $\varphi$ small and $\xi\to\infty$,
we give in this section some further bounds on the various contributions
in the sum in \eqref{PHI0FIRSTSPLIT2}, which fit naturally in the
present discussion.

Our first result says that for $\varphi$ small, the
contribution from all terms with $k_1\neq1$ in 
\eqref{PHI0FIRSTSPLIT2} (or in \eqref{PHI0FIRSTSPLIT})
is of strictly lower order of magnitude than the right hand side
of \eqref{CYLINDER2PTSMAINTHMRES}, and furthermore
if $\xi^{\frac2{d-1}}\max(1-\|\vecz\|,1-\|\vecw\|)$ 
is sufficiently large then these terms
in fact vanish! 

\begin{prop}\label{K1NEQ1BOUNDPROP}
Let $B_1$, $B_2$, $\vecp$, $\vecw$, $\vecz$ be as in
\eqref{CYLINDER2PTSMAINTHMPF1}, \eqref{CYLINDER2PTSMAINTHMPF2alt},
with $\varphi\leq\frac\pi2$,
and let $\fC$ be the the interior of the convex hull of 
$B_1$ and $B_2$, as before.
Then
the contribution from all terms with $k_1\neq1$ in \eqref{PHI0FIRSTSPLIT2} is
\begin{align}\label{K1NEQ1BOUNDPROPRES}
\ll\begin{cases}
\xi^{-2}\log(2+\min(\xi,\varphi^{-1}))&\text{if }\:d=3
\\
\xi^{-2}\min\bigl(1,(\xi\varphi^{d-2})^{-\frac{d-3}{d-1}}\bigr)
&\text{if }\:d\geq4.
\end{cases}
\end{align}
Furthermore there is a constant $c>0$ which only depends on $d$ such that
if either $1-\|\vecz\|$ or $1-\|\vecw\|$
is $\geq c\,\xi^{-\frac2{d-1}}$, then %
all terms with $k_1\neq1$ in \eqref{PHI0FIRSTSPLIT2} vanish.
\end{prop}

We stress that for the proposition to be valid
it is crucial that we assume $\F_d\subset\Si_d'$, 
viz.\ $v_1\geq0$ for all $[a_1,\vecv,\vecu,\tM]\in\F_d$.

\vspace{5pt}

We
will need the following variant of
(a part of) Corollary \ref{CONECYLCOR}.
\begin{lem}\label{CONELATTICELEM3VAR}
Assume $d\geq2$, $r\geq h>0$ and $A>0$.
Let $\fC\subset\R^d$ be a right open cone of height $h$ whose base
is a $(d-1)$-dimensional ball of radius $r$ containing $\bn$
in its relative interior.
Then
\begin{align}\label{CONELATTICELEM3VARRES}
\mu\bigl(\bigl\{M=\nn(u) \aa(a) \kk\in\Si_d\col a_1>A,\:
\Z^dM\cap\fC=\emptyset\bigr\}\bigr)
\ll A^{-d}\min\bigl(1,(Ahr^{d-2})^{\frac2d-1}\bigr).
\end{align}
In fact the left hand side of \eqref{CONELATTICELEM3VARRES} vanishes
unless the edge ratio of $\fC$ is $e\ll(Ahr^{d-2})^{-\frac2d}$.
\end{lem}
\begin{proof}
After applying an appropriate rotation $M_0\in\SO(d)$
we may assume $\fC$ is the interior of the convex hull of
$B$ and $\vecp$, where
\begin{align*}
B=t\vece_2+\iota(\scrB_r^{d-1})\quad\text{and}\quad
\vecp=h\vece_1+t\vece_2,\quad\text{for some }\:0\leq t<r.
\end{align*}
(We keep $M_0\in\SO(d)$ in order that the map
$M\mapsto MM_0$ leaves ``$a_1$'' invariant;
this is the reason why we require $\fC$ to be a right cone from start.)
Mimicking the proof of Proposition~\ref{GENPRINCBOUNDPROP} we obtain
that the left hand side of \eqref{CONELATTICELEM3VARRES} is
\begin{align*}%
\ll\int_{\S_1^{d-1}}\int_A^\infty 
p^{(d-1)}\Bigl(a_1^{\frac1{d-1}}\fC_\vecv\Bigr)\,\frac{da_1}{a_1^{d+1}}
\,d\vecv.
\end{align*}
This is of course $\ll\int_A^\infty a_1^{-d-1}\,da_1\ll A^{-d}$,
and this suffices to prove \eqref{CONELATTICELEM3VARRES} when $d=2$.
Next, if $d\geq3$ then
by a small modification of the proof of Lemma \ref{CONEINTERSECTLEM}
we see that for any $\vecv$ as in \eqref{VPARA} with $\varpi,\omega\in(0,\pi)$,
$\fC_\vecv$ contains a $(d-1)$-dimensional cone of volume
$\gg \frac h{\sqrt{r^2+h^2}}r^{d-1}(\sin\omega)^d
\gg hr^{d-2}(\sin\omega)^d$, the base of which contains $\bn$.
Hence we get, using \eqref{ATHREYAMARGULISRES} in dimension $d-1$,
\begin{align*}
\ll\int_{A}^\infty\int_0^{\pi/2}
\min\bigl(1,A^{-1}h^{-1}r^{2-d}\omega^{-d}\bigr)\omega^{d-3}\,d\omega
\,\frac{da_1}{a_1^{d+1}},
\end{align*}
and this leads to the bound in \eqref{CONELATTICELEM3VARRES}.
(Note that we do not get any better bound by using
Corollary \ref{CONECYLCOR} in place of \eqref{ATHREYAMARGULISRES}.)

We next prove the statement about vanishing.
For $d=2$ one notes that $\fC_\vecv$ is a line segment of length
$\gg eh$ for all $\vecv\in\S_1^{d-1}\setminus\{\pm\vece_1\}$,
and hence if $Aeh$ is sufficiently large then
$p^{(1)}(a_1\fC_\vecv)=0$ for all
$a_1>A$ and all $\vecv\in\S_1^{d-1}\setminus\{\pm\vece_1\}$,
and hence the left hand side of \eqref{CONELATTICELEM3VARRES} vanishes.
Next assume $d\geq3$.
Note that for any $\vecv\in\S_1^{d-1}\setminus\{\pm\vece_1\}$,
as in \eqref{VPARA} with $\varpi,\omega\in(0,\pi)$,
the intersection $\fC\cap\vecv^\perp$ contains a $(d-1)$-dimensional
cone of radius $r'\gg r(e+\sin^2\omega)^{\frac12}$,
height $h'\geq\frac{h}{\sqrt{r^2+h^2}}(r-t|\cos\omega|)
\asymp h(e+\sin^2\omega)$,
and edge ratio $e'\asymp\min(1,\frac e{\sin^2\omega})$.
Thus if $a_1>A$ then
\begin{align*}
e'\bigl|a_1^{\frac1{d-1}}\fC_\vecv\bigr|^{\frac2{d-1}}
\gg e'\bigl(Ahr^{d-2}(e+\sin^2\omega)^{\frac d2}\bigr)^{\frac2{d-1}}
\asymp r^{\frac{2(d-2)}{d-1}}h^{\frac2{d-1}}A^{\frac2{d-1}}
e\max(e,\sin^2\omega)^{\frac1{d-1}}
\\
\geq (e^{\frac d2}Ahr^{d-2})^{\frac2{d-1}}.
\end{align*}
Hence if $e^{\frac d2}Ahr^{d-2}$ is large then
also $e'\bigl|a_1^{\frac1{d-1}}\fC_\vecv\bigr|^{\frac2{d-1}}$ is
large, uniformly over all $a_1\geq A$ and all
$\vecv\in\S_1^{d-1}\setminus\{\pm\vece_1\}$,
and by Corollary \ref{CONECYLCOR} (in dimension $d-1$) this
forces $p^{(d-1)}(a_1^{\frac1{d-1}}\fC_\vecv)=0$ for all these $a_1,\vecv$,
which means that the left hand side of \eqref{CONELATTICELEM3VARRES} vanishes.
\end{proof}

\begin{proof}[Proof of Proposition \ref{K1NEQ1BOUNDPROP}]
Let us first consider the terms with $k_1\geq2$ in \eqref{PHI0FIRSTSPLIT2}.
We now refine the argument around 
\eqref{CYLINDER2PTSMAINTHMPF5}--\eqref{CYLINDER2PTSMAINTHMPF6}.
Given $M=[\vecv,\tM]_{\veck,\vecp}\in G_{\veck,\vecp}$,
by applying \eqref{LATTICEINPARAM} with $n=\lfloor\frac{k_1}2\rfloor$
we see that there is some $\veca\in\R^{d-1}$ such that
\begin{align*}
\Bigl(\lfloor\sfrac{k_1}2\rfloor,\vecm\Bigr)M
=\lfloor\sfrac{k_1}2\rfloor a_1\vecv+a_1^{-\frac1{d-1}}
\bigl(0,\veca+\vecm\tM\bigr)f(\vecv),
\qquad\forall\vecm\in\Z^{d-1}.
\end{align*}
Hence if $\Z^dM\cap\fC=\emptyset$ then in particular $\fC$ has empty
intersection with $\lfloor\frac{k_1}2\rfloor a_1\vecv+a_1^{-\frac1{d-1}}
\iota(\veca+\Z^{d-1}\tM)f(\vecv)$.
The latter is a $(d-1)$-dimensional lattice inside the hyperplane
$\lfloor\frac{k_1}2\rfloor a_1\vecv+\vecv^\perp$,
and because of $a_1=k_1^{-1}(\vecp\cdot\vecv)$
this hyperplane contains the point $t\vecp$ where
$t:=k_1^{-1}\lfloor\frac{k_1}2\rfloor$.
Note that $\frac13\leq t\leq\frac12$ since $k_1\geq2$.
Now there is an absolute constant $c>0$ such that
$\|t\vecw-(1-t)\vecz\|<1-c$ for all 
$\vecz,\vecw\in\scrB_1^{d-1}$ with $\varphi(\vecz,\vecw)\leq\frac\pi2$
and all $\frac13\leq t\leq\frac12$;
hence, if we take $c\leq\frac13$,
we necessarily have $t\vecp+\scrB_{cr}^d\subset\fC$.
Hence if $\Z^dM\cap\fC=\emptyset$ then
$a_1^{-\frac1{d-1}}\Z^{d-1}\tM$ must have empty intersection
with a certain ball of radius $cr$,
and thus using $a_1>Cr$ and Lemma \ref{A1LARGELEM} 
(and $M\in\F_d\subset\Si_d\Rightarrow\tM\in\Si_{d-1}$),
we conclude that $\ta_1\geq C'r^{\frac d{d-1}}$,
where $C'$ is a positive constant which only depends on $d$.
Hence the contribution from $k_1\geq2$ in \eqref{PHI0FIRSTSPLIT2} is,
by mimicking the argument in Section \ref{PPCBOUNDCCONESEC} leading up to
\eqref{SUM1BOUNDFIRSTTIMEEVER}:
\begin{align}\label{CYLINDER2PTSMAINTHMPF9new}
\ll\sum_{2\leq k_1<2/C}k_1^{d-1}\int_{S}
\mu\Bigl(\Bigl\{\tM\in\Si_{d-1}\col
\ta_1\geq C'r^{\frac d{d-1}},\:
\Z^{d-1}\tM\cap a_1^{\frac1{d-1}}(\fC_\vecv\cup\fC'_\vecv)=\emptyset
\Bigr\}\Bigr)
\,\frac{d\vecv}{(k_1r)^d},
\end{align}
where $S=S^{(k_1)}$ is as in \eqref{SK1DEF},
and $a_1=k_1^{-1}(\vecp\cdot\vecv)$ as before
(thus $a_1>Cr$ for all $\vecv\in S^{(k_1)}$).

Now by Lemma \ref{CONVENIENTCONELEM},
$\fC\cap\vecv^\perp$ contains a right relatively open $(d-1)$-dimensional
cone with $\bn$ in its base,
of radius $r'$, height $h'$ and edge ratio $e'$ satisfying
\begin{align}\label{RHEBOUNDS}
r'\gg r(1-z+\sin^2\omega_\vecz)^{\frac12},\qquad
h'\gg r(1-z+\sin^2\omega_\vecz),\qquad
e'\asymp\min\bigl(1,\sfrac{1-z}{\sin^2\omega_\vecz}\bigr),
\end{align}
where we write $z=\|\vecz\|$.
We may also assume $h'\leq r'$, 
so that Lemma \ref{CONELATTICELEM3VAR} applies;
for if $h'>r'$ then we may shrink the cone by decreasing $h'$ until $h'=r'$;
the bounds \eqref{RHEBOUNDS} remain true.
Now $a_1^{\frac1{d-1}}\fC_\vecv$ has radius
$a_1^{\frac1{d-1}}r'$ and height $a_1^{\frac1{d-1}}h'$, and by 
Lemma \ref{CONELATTICELEM3VAR},
applied with $A=C'r^{\frac d{d-1}}$ and with $d-1$ in place of $d$,
it follows that (for $\vecv\in S\setminus\{\vece_1\}$)
the integrand in \eqref{CYLINDER2PTSMAINTHMPF9new} vanishes
whenever the product
\begin{align*}
\min\bigl(1,\sfrac{1-z}{\sin^2\omega_\vecz}\bigr)^{\frac{d-1}2}
r^{\frac d{d-1}+\frac{(d-2)d}{d-1}}(1-z+\sin^2\omega_\vecz)^{\frac{d-1}2}
\end{align*}
is larger than a certain constant which only depends on $d$.
The last expression is seen to be
$\asymp r^d(1-z)^{\frac{d-1}2}=\xi(1-\|\vecz\|)^{\frac{d-1}2}$,
independently of $\omega_\vecz$;
hence we conclude that \eqref{CYLINDER2PTSMAINTHMPF9new} vanishes
whenever $\xi(1-\|\vecz\|)^{\frac{d-1}2}$ is sufficiently large.
Similarly, using $\fC_\vecv'$ in place of $\fC_\vecv$,
one proves that \eqref{CYLINDER2PTSMAINTHMPF9new} vanishes
whenever $\xi(1-\|\vecw\|)^{\frac{d-1}2}$ is sufficiently large.
Furthermore, Lemma \ref{CONELATTICELEM3VAR} also implies
(using just $r'\gg r\sin\omega_\vecz$ and $h'\gg r\sin^2\omega_\vecz$,
and the corresponding result for $\fC'\cap\vecv^\perp$) that
for general $\vecz,\vecw$ (with $\varphi\leq\frac\pi2$),
\eqref{CYLINDER2PTSMAINTHMPF9new} is
\begin{align}
\ll r^{-d}\int_{\S_1^{d-1}}  %
r^{-d}\min\Bigl\{1,
\Bigl(r^d\max(\sin\omega_\vecz,\sin\omega_\vecw)^{d-1}\Bigr)^{\frac2{d-1}-1}
\Bigr\}\,d\vecv.
\end{align}
Now by \eqref{MAXSOZWINEQ} we have
$\max(\sin\omega_\vecz,\sin\omega_\vecw)\gg\sin\omega$;
thus the above is
\begin{align}\label{K1GEQ2BOUND}
\ll r^{-d}\int_0^{\pi/2} r^{-d}\min\Bigl\{1,
(r^d\omega^{d-1})^{\frac2{d-1}-1}\Bigr\}\,\omega^{d-3}\,d\omega
\ll r^{-3d+\frac{2d}{d-1}}\int_0^{\pi/2}\,d\omega\ll %
\xi^{-3+\frac2{d-1}}.
\end{align}
Hence we have now proved both the desired vanishing and the desired bound
for all the terms with $k_1\geq2$ in \eqref{PHI0FIRSTSPLIT2}.

\vspace{5pt}

We next consider the terms with $k_1\leq-1$. %
We will again refine the argument around 
\eqref{CYLINDER2PTSMAINTHMPF5}--\eqref{CYLINDER2PTSMAINTHMPF6}.
Note that if $\vecv\in S^{(k_1)}$ for some $k_1\leq-1$ then
$\vecp\cdot\vecv<0$ and $v_1>0$; hence
$(\vecz+\vecw)\cdot\vecv'<0$, viz.\ 
$\varphi(\vecz+\vecw,\vecv')>\frac\pi2$.
By Proposition \ref{CYLSUPPPROP}
we may assume $\vecz,\vecw$ to be close to $1$ without loss of generality;
hence since $\varphi=\varphi(\vecz,\vecw)\leq\frac\pi2$ there exists an
absolute constant $c>0$ such that
$\varphi(\vecz+\vecw,\vecz)\leq\frac\pi2-c$
and $\varphi(\vecz+\vecw,\vecw)\leq\frac\pi2-c$.
It follows that 
\begin{align*}
\sfrac\pi2<\varphi(\vecz+\vecw,\vecv')
\leq\varphi(\vecz+\vecw,\vecz)+\varphi(\vecz,\vecv')
\leq\sfrac\pi2-c+\omega_\vecz,
\end{align*}
viz.\ $\omega_\vecz>c$. Similarly $\omega_\vecw>c$.
Hence by Lemma \ref{CONVENIENTCONELEM}, $\fC\cap\vecv^\perp$ contains
a relatively open 
$(d-1)$-dimensional cone of volume $\gg r^{d-1}(\sin\omega_\vecz)^{d-2}$
with $\bn$ in its base, and $\fC'\cap\vecv^\perp$ contains
a relatively open 
$(d-1)$-dimensional cone of volume $\gg r^{d-1}(\sin\omega_\vecw)^{d-2}$
with $\bn$ in its base.
Hence by Corollary \ref{CONECYLCOR} we have
\begin{align}\label{OMEGAGTPIHALFBD}
\mu\Bigl(\Bigl\{\tM\in\Si_{d-1}\col
\Z^{d-1}\tM\cap a_1^{\frac1{d-1}}(\fC_\vecv\cup\fC'_\vecv)=\emptyset
\Bigr\}\Bigr)
\ll\Bigl(r^d\max\bigl(\sin\omega_\vecz,\sin\omega_\vecw\bigr)^{d-2}\Bigr)^{-\frac{2(d-2)}{d-1}}
\end{align}
whenever $a_1>Cr$.
Using this improvement of \eqref{CYLINDER2PTSMAINTHMPF5}
and the same type of argument as in 
\eqref{CYLINDER2PTSMAINTHMPF6}, \eqref{CYLINDER2PTSMAINTHMPF3},
we obtain that the contribution from all terms with $k_1\leq-1$ in
\eqref{PHI0FIRSTSPLIT2} %
is bounded by
exactly the same integral as in \eqref{CYLINDER2PTSMAINTHMPF8},
i.e.\ this contribution is
\begin{align}\label{OMEGAGTPIHALFBD2}
\ll\begin{cases}
r^{-6}\log(2+\min(r,\varphi^{-1}))&\text{if }\: d=3
\\
r^{-2d}\min(1,(r^d\varphi^{d-2})^{-\frac{d-3}{d-1}})&\text{if }\: d\geq4,
\end{cases}
\end{align}
which agrees with \eqref{K1NEQ1BOUNDPROPRES}.
To prove the statement about vanishing we note that 
the $(d-1)$-dimensional cone inside $\fC\cap\vecv^\perp$
provided by Lemma \ref{CONVENIENTCONELEM} in fact has 
volume $\gg r^{d-1}(1-z+\sin^2\omega_\vecz)^{\frac{d-2}2}$;
also the edge ratio is $\asymp\min(1,\frac{1-z}{\sin^2\omega_\vecz})$,
and hence from Corollary \ref{CONECYLCOR} it follows that
for any $k_1\leq-1$ the integrand in \eqref{CYLINDER2PTSMAINTHMPF9}
vanishes
if the product
\begin{align}\label{VANPROD}
\min\bigl(1,\sfrac{1-z}{\sin^2\omega_\vecz}\bigr)^{\frac{d-1}2}
r^d(1-z+\sin^2\omega_\vecz)^{\frac{d-2}2}
\end{align}
is larger than a certain constant which only depends on $d$.
This product is 
$\gg\xi(1-z)^{\frac{d-1}2}$.
Hence all terms with $k_1\leq-1$ in \eqref{PHI0FIRSTSPLIT2} vanish
whenever $\xi(1-\|\vecz\|)^{\frac{d-1}2}$ is sufficiently large.
Similarly, using $\fC_\vecv'$ in place of $\fC_\vecv$, we see that this
vanishing statement also holds
whenever $\xi(1-\|\vecw\|)^{\frac{d-1}2}$ is sufficiently large.

\vspace{5pt}

It now only remains to consider the terms with $k_1=0$.
The contribution from these terms has already been proved to satisfy
the bound \eqref{K1NEQ1BOUNDPROPRES}; cf.\ \eqref{CYLINDER2PTSMAINTHMPF8}.
To prove the desired vanishing we note that the 
$(d-1)$-dimensional cone inside $\fC\cap\vecv^\perp$
used in the proof of \eqref{CYLINDER2PTSMAINTHMPF8}
in fact has volume $\gg r^{d-1}(1-z+\sin^2\omega_\vecz)^{\frac{d-2}2}$,
since the base $B_1\cap\vecv^\perp$ has radius
$\gg r(1-z+\sin^2\omega_\vecz)^{\frac12}$ just as in the proof of
Lemma \ref{CONVENIENTCONELEM}.
Also the edge ratio is $\asymp\min(1,\frac{1-z}{\sin^2\omega_\vecz})$,
and hence by Corollary~\ref{CONECYLCOR} it follows that 
the integrand in \eqref{CYLINDER2PTSMAINTHMPF7} 
vanishes if the product in \eqref{VANPROD} is sufficiently large.
Hence we have $\sum_0=0$ whenever $\xi(1-\|\vecz\|)^{\frac{d-1}2}$ is
sufficiently large; and similarly we also have
$\sum_0=0$ whenever $\xi(1-\|\vecw\|)^{\frac{d-1}2}$ is sufficiently large.
\end{proof}
\vspace{5pt}

We next give a bound which shows that 
when considering the contribution for
$k_1=1$ in \eqref{PHI0FIRSTSPLIT2}, we may restrict the range of
$\vecv$ in $M=[a_1,\vecv,\vecu,\tM]$ somewhat, at the cost of an
error satisfying the same bound as in Proposition~\ref{K1NEQ1BOUNDPROP}.

\begin{prop}\label{RESTRRANGELEM}
Take notation %
as in Section \ref{CYLBDSEC}, assume $\varphi\leq\frac\pi2$,
and keep $k_1=1$, so that $a_1=\vecp\cdot\vecv$.
Let $C'>1$ be an arbitrary constant, and set
\begin{align*}
\Delta=\Bigl\{\vecv\in\S_1^{d-1}\col\vecp\cdot\vecv>Cr,\:
0\leq v_1\leq C'(\varphi+\omega)^2\Bigr\}.
\end{align*}
Then
\begin{align}\notag
\int_{\Delta}
\mu\Bigl(\Bigl\{\tM\in\Si_{d-1}\col
\Z^{d-1}\tM\cap a_1^{\frac1{d-1}}(\fC_\vecv\cup\fC'_\vecv)=\emptyset
\Bigr\}\Bigr)\,\frac{d\vecv}{|\vecp\cdot\vecv|^d}\hspace{60pt}
\\\label{RESTRRANGELEMRES}
\ll\begin{cases}
\xi^{-2}\log(2+\min(\xi,\varphi^{-1}))&\text{if }\:d=3
\\
\xi^{-2}\min\bigl(1,(\xi\varphi^{d-2})^{-\frac{d-3}{d-1}}\bigr)
&\text{if }\:d\geq4,
\end{cases}
\end{align}
where the implied constant depends only on $d$ and $C'$.
\end{prop}
The proof depends on the following lemma.
\begin{lem}\label{CONVENIENTCONELEMvar}
For every $\vecv\in\Delta^\circ$ with $\omega_\vecz\geq\frac12\varphi$,
$\fC\cap\vecv$ contains a $(d-1)$-dimensional cone of volume
$\gg r^{d-1}(\sin\omega_\vecz)^{d-2}$ with $\bn$ in its base.
(The implied constant depends only on $d$ and $C'$.)
\end{lem}
\begin{proof}
We may assume $\omega_\vecz<\frac1{10}$, since the claim otherwise
follows from Lemma \ref{CONVENIENTCONELEM}.
Now $\vecv\in\Delta^\circ$ and $\omega_\vecz\geq\frac12\varphi$
imply $0<v_1<C'(\varphi+\omega)^2
\leq C'(2\omega_\vecz+\omega_\vecz+\varphi)^2\leq25C'\omega_\vecz^2$.
Now to construct our cone we rotate temporarily, 
as in the proof of Lemma \ref{CONVENIENTCONELEM}, to the situation
where $\vecz=z\vece_1$ for some $0<z<1$;
thus $\omega_\vecz=\varphi(\vecv',\vece_1)$.
Fix $0<c<1$ so small that $\cos\alpha+25C'c\alpha^2<1$ 
for all $\alpha\in(0,\frac1{10}]$.
Now $B_1\cap\vecv^\perp$ is a $(d-2)$-dimensional ball with center
$\vecq$ as in \eqref{QCENTERDEF} and radius $r'\gg r\sin\omega_\vecz$.
Instead of \eqref{HFIRSTDEF} we set
\begin{align*}
\vech=(h_1,\ldots,h_d)
:=cr\|\vecv'\|^{-1}\bigl(\|\vecv'\|^2\vece_1-v_1\iota(\vecv')\bigr).
\end{align*}
Indeed this gives $0<h_1<r$ and
\begin{align*}
\bigl\|r\vecz-(\vecq'+\vech')\bigr\|
=r\Bigl\|\frac{z\cos\omega_\vecz+cv_1}{\|\vecv'\|}\vecv'\Bigr\|
=r(z\cos\omega_\vecz+cv_1)
<r(\cos\omega_\vecz+25C'c\omega_\vecz^2)
\leq r,
\end{align*}
by our choice of $c$; hence $\vecq+\vech\in\fC$.
The desired conclusion now follows as in the proof of 
Lemma \ref{CONVENIENTCONELEM}.
\end{proof}
\begin{proof}[Proof of Proposition \ref{RESTRRANGELEM}]
For every $\vecv\in\Delta^\circ$ it follows from 
Lemma \ref{CONVENIENTCONELEMvar} and its analogue for
$\vecw$ that at least one of $\fC\cap\vecv$ or $\fC'\cap\vecv$ contains
a $(d-1)$-dimensional cone of volume 
$\gg r^{d-1}\max(\sin\omega_\vecz,\sin\omega_\vecw)^{d-2}$ with
$\bn$ in its base.
Indeed this is direct if $\omega_\vecz,\omega_\vecw\geq\frac12\varphi$;
on the other hand if $\omega_\vecz<\frac12\varphi$ then by
the triangle inequality in $\S_1^{d-1}$ we have
$\frac12\varphi<\omega_\vecw<\frac32\varphi\leq\frac34\pi
\leq\pi-\frac12\varphi$
and thus $\fC'\cap\vecv$ contains a $(d-1)$-dimensional cone of
volume $\gg r^{d-1}(\sin\omega_\vecw)^{d-2}
\geq r^{d-1}\max(\sin\omega_\vecz,\sin\omega_\vecw)^{d-2}$;
similarly the statement also holds if $\omega_\vecw<\frac12\varphi$.

Hence \eqref{OMEGAGTPIHALFBD} holds for all $\vecv\in\Delta^\circ$,
and integrating over $\vecv$ we obtain the stated bound.
\end{proof}

\subsection{\texorpdfstring{Lower bounds on $\Phi_\bn(\xi,\vecw,\vecz)$}{Lower bounds on Phi0(xi,w,z)}}

In this section we prove some lower bounds on $\Phi_\bn(\xi,\vecw,\vecz)$.
Our first proposition says that the bound in Theorem \ref{CYLINDER2PTSMAINTHM} 
is sharp in a natural sense when $d=3$.
Note that for $d=3$, Theorem \ref{CYLINDER2PTSMAINTHM} states that
$\Phi_\bn(\xi,\vecw,\vecz)\ll\min(\xi^{-\frac43},\xi^{-2}\varphi^{-2})$, 
uniformly over $\varphi\in[0,\pi]$.

\begin{prop}\label{CYLD3OPTPROP}
There is an absolute constant $c>0$ such that, for $d=3$,
\begin{align}\label{CYLD3OPTPROPRES}
\Phi_\bn(\xi,\vecw,\vecz)\gg
\min\bigl(1,\xi^{-\frac43},\xi^{-2}\varphi^{-2}\bigr)
\end{align}
holds whenever 
$\|\vecw\|,\|\vecz\|\geq 1-c\cdot s_3(\xi,\varphi)$.
\end{prop}
The next proposition says that for general $d\geq3$
the bound in Theorem \ref{CYLINDER2PTSMAINTHM} is sharp
if either $\varphi\ll\xi^{-\frac1d}$ or $\pi-\varphi\ll\xi^{-\frac1{d-2}}$.
(To see that the restrictions on $\|\vecw\|,\|\vecz\|$ below
are natural, note that for $\xi$ large we have
$s_d(\xi,\varphi)\asymp\xi^{-\frac2d}$ if $\varphi\ll\xi^{-\frac1d}$,
and $s_d(\xi,\varphi)\asymp\xi^{-\frac2{d-2}}$
if $\pi-\varphi\ll\xi^{-\frac1{d-2}}$.)
\begin{prop}\label{CYLINDER2PTSLOWBDPROP2}
Let $d\geq3$. There is a constant $c>0$ which only depends on $d$
such that, 
if $\varphi\leq c\xi^{-\frac1d}$ and
$\|\vecw\|,\|\vecz\|\geq1-c\xi^{-\frac2d}$
then
\begin{align}\label{CYLINDER2PTSLOWBDPROP2RES1}
\Phi_\bn(\xi,\vecw,\vecz)\gg\min\bigl(1,\xi^{-2+\frac2d}\bigr)
\end{align}
whereas if $\pi-\varphi\leq c\xi^{-\frac1{d-2}}$ and 
$\|\vecw\|,\|\vecz\|\geq1-c\xi^{-\frac2{d-2}}$ then
\begin{align}\label{CYLINDER2PTSLOWBDPROP2RES2}
\Phi_\bn(\xi,\vecw,\vecz)\gg\min\bigl(1,\xi^{-2}\bigr).
\end{align}
\end{prop}

When proving these two propositions we will again
let $\fC$ and $\vecp$ be as 
around \eqref{CYLINDER2PTSMAINTHMPF1} (wherein $r=\xi^{1/d}>0$),
with $\vecz,\vecw$ as in \eqref{CYLINDER2PTSMAINTHMPF2alt}.
We start by giving some auxiliary lemmas.

\begin{lem}\label{CONVENIENTCYLINDERLEM}
For any $\vecv\in\S_1^{d-1}$ with $\frac9{10}<v_1<1$, the intersection
$\fC\cap\vecv^\perp$ is contained in a right
$(d-1)$-dimensional cylinder with $\bn$ in one of its bases,
of height $\asymp r(1-\|\vecz\|+\omega_\vecz^2)$, radius
$\asymp r(1-\|\vecz\|+\sin^2\omega_\vecz)^{\frac12}$
and edge ratio $\asymp\min(1,\frac{1-\|\vecz\|}{\sin^2\omega_\vecz})$.
\end{lem}
\begin{proof}
This is similar to the proof of Lemma \ref{CONVENIENTCONELEM}.
\end{proof}

\begin{lem}\label{CYLINDER2PTSLOWBDPROP2RES1PF1LEM}
\begin{align*}%
p_\vecp(\fC)\gg r^{-d}
\int_{\{\vecv\in\S_1^{d-1}\col v_1>\frac{99}{100}\}}
\mu\Bigl(\Bigl\{\tM\in\Si_{d-1}\col
\ta_1\leq(\sfrac23r)^{\frac d{d-1}},\:
\Z^{d-1}\tM\cap a_1^{\frac1{d-1}}(\fC_\vecv\cup\fC_\vecv')
=\emptyset\Bigr\}\Bigr)
\,d\vecv,
\end{align*}
where $a_1=\vecp\cdot\vecv$. %
\end{lem}
\begin{proof}
Arguing as in Section \ref{PPCFROMBELOWSEC} 
we see that \eqref{PPCFROMBELOWSECSTAT1} holds for our $\fC$.
Now for any $\vecv\in\S_1^{d-1}$ with $\frac{99}{100}<v_1<1$
we have $\|\vecv'\|<\frac3{20}$ and thus
\begin{align*}
|a_1-r|=|\vecp\cdot\vecv-r|\leq r(1-v_1)+\|\vecw+\vecz\|\cdot\|\vecv'\|
<\sfrac13r,
\end{align*}
i.e.\ $\frac23r<a_1<\frac43r$.
Now the lemma will follow by arguing in the same way as in 
Section \ref{PPCFROMBELOWSEC}  (up to \eqref{PPCBOUNDFROMBELOW}),
if we can only prove that 
$(na_1\vecv+\vecv^\perp)\cap\fC=\emptyset$ holds
for all $n\in\Z\setminus\{0,1\}$ and all
$\vecv\in\S_1^{d-1}$ with $\frac{99}{100}<v_1<1$.
But note that for any such $\vecv$ and any $\vecx\in\fC$,
we have, using $\|\vecv'\|<\frac3{20}$, $0<x_1<r$, $\|(x_2,\ldots,x_d)\|<2r$
and $\frac23r<a_1<\frac43r$:
\begin{align*}
-a_1<-\sfrac3{10}r<-2r\cdot\sfrac3{20}
<\vecx\cdot\vecv
<r+2r\cdot\sfrac3{20}=\sfrac{13}{10}r<2a_1.
\end{align*}
Hence $(na_1\vecv+\vecv^\perp)\cap\fC=\emptyset$ for all 
$n\in\Z\setminus\{0,1\}$, and we are done.
\end{proof}

\begin{lem}\label{PHI0MONOTONELEM}
For any fixed $\vecw,\vecz\in\scrB_1^{d-1}$ the function
$\Phi_\bn(\xi,\vecw,\vecz)$ is %
decreasing 
in $\xi$.
\end{lem}
\begin{proof}
Let $0<\xi<\xi'$ be given, 
and set 
\begin{align*}
&\fC=\bigl\{(x_1,\ldots,x_d)\in\R^d\col
0<x_1<\xi,\:\bigl\|(x_2,\ldots,x_d)-\vecz\bigr\|<1\bigr\};
&&
\vecp=(\xi,\vecz+\vecw);
\\
&\fC'=\bigl\{(x_1,\ldots,x_d)\in\R^d\col
0<x_1<\xi',\:\bigl\|(x_2,\ldots,x_d)-\vecz\bigr\|<1\bigr\};
&&
\vecp'=(\xi',\vecz+\vecw),
\end{align*}
so that 
$\Phi_\bn(\xi,\vecw,\vecz)=p_\vecp(\fC)$ and
$\Phi_\bn(\xi',\vecw,\vecz)=p_{\vecp'}(\fC')$.
Furthermore set $\alpha=\frac{\xi'}{\xi}>1$ and
\begin{align*}
T=\begin{pmatrix}\alpha &
{\xi}^{-1}(1-\alpha^{-\frac1{d-1}})(\vecw+\vecz)
\\
\trans\bn & \alpha^{-\frac1{d-1}}1_{d-1}
\end{pmatrix}\in G.
\end{align*}
We now claim $\fC T\subset\fC'$.
Indeed, take an arbitrary point $\vecx=(x_1,\ldots,x_d)\in\fC$ and set 
$\vecy=(y_1,\ldots,y_d)=\vecx T$.
Then $y_1=\alpha x_1\in(0,\xi')$, and
\begin{align*}
\bigl\|(y_2,\ldots,y_d)-\vecz\bigr\|
&=\bigl\| x_1\xi^{-1}(1-\alpha^{-\frac1{d-1}})(\vecw+\vecz)
+\alpha^{-\frac1{d-1}}(x_2,\ldots,x_d)-\vecz\bigr\|
\\
&=\bigl\|(1-\alpha^{-\frac1{d-1}})( x_1\xi^{-1}\vecw-(1- x_1\xi^{-1})\vecz)
+\alpha^{-\frac1{d-1}}((x_2,\ldots,x_d)-\vecz)\bigr\|
\\
&\hspace{200pt}<(1-\alpha^{-\frac1{d-1}})+\alpha^{-\frac1{d-1}}=1.
\end{align*}
Hence $\vecy\in\fC'$ and the claim is proved.
Noticing also $\vecp T=\vecp'$,
it now follows that 
$p_\vecp(\fC)=p_{\vecp'}(\fC T)\geq p_{\vecp'}(\fC')$,
i.e.\ $\Phi_\bn(\xi,\vecw,\vecz)\geq \Phi_\bn(\xi',\vecw,\vecz)$.
\end{proof}

\begin{proof}[Proof of \eqref{CYLINDER2PTSLOWBDPROP2RES1} in Proposition \ref{CYLINDER2PTSLOWBDPROP2}]
We assume from start $0<c<1$, and consider some
arbitrary $\xi,\vecw,\vecz$ with $\varphi\leq c\xi^{-\frac1d}$ and
$\|\vecw\|,\|\vecz\|\geq1-c\xi^{-\frac2d}$.
Now for every $\vecv\in\S_1^{d-1}$ with
$\frac{99}{100}<v_1<1$ and
$\omega=\varphi(\vecv',\vece_1)<cr^{-1}$
we have
$\omega_\vecz,\omega_\vecw<2cr^{-1}$, 
and hence by Lemma \ref{CONVENIENTCYLINDERLEM},
$|\fC_\vecv|\ll %
c^{\frac d2}r^{-1}$ and similarly
$|\fC_\vecv'|\ll c^{\frac d2}r^{-1}$.
Also recall from the proof of Lemma \ref{CYLINDER2PTSLOWBDPROP2RES1PF1LEM} 
that $a_1=\vecp\cdot\vecv<\frac43r$.
Hence if $c$ is sufficiently small then
$|a_1^{\frac1{d-1}}(\fC_\vecv\cup\fC_\vecv')|<\frac12$
holds for all $\vecv\in\S_1^{d-1}$ with
$\frac{99}{100}<v_1<1$ and $\omega<cr^{-1}$.
Considering the contribution from these $\vecv$ in 
Lemma \ref{CYLINDER2PTSLOWBDPROP2RES1PF1LEM}
we obtain, if $r=\xi^{1/d}$ is sufficiently large,
\begin{align*}
\Phi_\bn(\xi,\vecw,\vecz)\gg r^{-d}\int_0^{cr^{-1}}\omega^{d-3}\,d\omega
\gg r^{-2d+2}=\xi^{-2+\frac2d}.
\end{align*}
Note also that Remark \ref{PPCVOLSMALLREM} applies to our $\fC$,
thus $\Phi_\bn(\xi,\vecw,\vecz)>\frac1{10}$ 
for all $\vecw,\vecz\in\scrB_1^{d-1}$ and all
$0<\xi<2^{-d}v_{d-1}^{-1}$,
where $v_{d-1}=|\scrB_1^{d-1}|=\pi^{\frac{d-1}2}\Gamma(\frac{d+1}2)^{-1}$.
Now \eqref{CYLINDER2PTSLOWBDPROP2RES1} follows for all $\xi>0$
by using the monotonicity in Lemma \ref{PHI0MONOTONELEM}.
\end{proof}

We next give the proof 
of \eqref{CYLINDER2PTSLOWBDPROP2RES2} in 
Proposition \ref{CYLINDER2PTSLOWBDPROP2} in the case $d\geq4$.
The remaining case $d=3$ will be treated below in the proof of 
Proposition \ref{CYLD3OPTPROP}.

\begin{proof}[Proof of \eqref{CYLINDER2PTSLOWBDPROP2RES2} in Proposition \ref{CYLINDER2PTSLOWBDPROP2} when $d\geq4$]
We have
\begin{align*}
&\Phi_\bn(\xi,\vecw,\vecz)\gg{\sum}_0
=\sum_{\veck'\in\widehat\Z^{d-1}}\nu_\vecp\bigl(\bigl\{
M\in G_{(0,\veck'),\vecp}\cap\Si_d\col
\Z^dM\cap\fC=\emptyset\bigr\}\bigr)
\\
&\asymp \sum_{\veck'\in\widehat\Z^{d-1}}
r^{-1}\int_{\R_{>0}}\int_{\S_1^{d-1}\cap\vecp^\perp}\int_{\R^{d-1}}
\nu_\vecy\bigl(\bigl\{\tM\in G_{\veck',\vecy}\col
M\in\Si_d,\:
\Z^dM\cap\fC=\emptyset\bigr\}\bigr)
\,d\vecu\,d_\vecp(\vecv)\,a_1^{-d}\,da_1,
\end{align*}
where in the last step we used Lemma \ref{GKYMEASLEM2},
and in the last line
$\vecy=\vecy(a_1,\vecv)$ is as in \eqref{RA1V},
and we write $M=[a_1,\vecv,\vecu,\tM]$.
Note that $\fC\subset\scrB_{3r}^d$, and thus if $a_1>3r$ then 
$\fC\cap(na_1\vecv+\vecv^\perp)=\emptyset$
for all $n\in\Z\setminus\{0\}$.
Hence, using \eqref{LATTICEINPARAM} and \eqref{SIDSUBSSIDM1},
we see that the above expression is
\begin{align*}
\gg \sum_{\veck'\in\widehat\Z^{d-1}}
r^{-d-1}\int_{3r}^{4r}\int_{\S_1^{d-1}\cap\vecp^\perp}
\nu_\vecy\Bigl(\Bigl\{\tM\in G_{\veck',\vecy}\cap\F_{d-1}\col
\ta_1\leq\sfrac2{\sqrt3}(3r)^{\frac d{d-1}},\:  
\hspace{70pt}
\\
\Z^{d-1}\tM\cap a_1^{\frac1{d-1}}\fC_\vecv=\emptyset\Bigr\}\Bigr)
\,d_\vecp(\vecv)\,da_1,
\end{align*}
where $\F_{d-1}\subset\Si_{d-1}$ is some fixed fundamental domain for
$\Gamma^{(d-1)}\backslash G^{(d-1)}$.
Using \eqref{X1E1P} in dimension $d-1$ we get,
so long as $r$ is larger than some constant which only depends on $d$, 
\begin{align}\label{CYLINDER2PTSLOWBDPROP2RES2PF1}
\geq r^{-d-1}\int_{\S_1^{d-1}\cap\vecp^\perp}\int_{3r}^{4r}
\Bigl(p_\vecy(a_1^{\frac1{d-1}}\fC_\vecv)-\sfrac1{100}\Bigr)
\,da_1\,d_\vecp(\vecv).
\end{align}
Now for every $\vecv\in\S_1^{d-1}\cap\vecp^\perp$ with
$\omega<cr^{-\frac d{d-2}}$
we have $\omega_\vecz<2cr^{-\frac d{d-2}}$
and $\pi-\omega_\vecw<2cr^{-\frac d{d-2}}$,
because of
\eqref{CYLINDER2PTSMAINTHMPF2alt} and $\pi-\varphi\leq cr^{-\frac d{d-2}}$.
Furthermore $\|\vecw\|,\|\vecz\|\geq 1-cr^{-\frac{2d}{d-2}}$ and
hence as in the proof of Lemma \ref{CONVENIENTCONELEM}, the
$(d-2)$-dimensional balls $B_1\cap\vecv$ and $B_2\cap\vecv$
both have radii
$\ll c^{\frac12}r^{-\frac{2}{d-2}}$.
Since $\fC\cap\vecv$ does not intersect the central axis of $\fC$
(provided $c$ is sufficiently small)
it follows that
$|\fC_\vecv|\ll c^{\frac{d-2}2}r^{-1}$.
Hence if $c$ is sufficiently small then
$|a_1^{\frac1{d-1}}\fC_\vecv|<\frac12$ for all 
$\vecv\in\S_1^{d-1}\cap\vecp^\perp$ with $\omega<cr^{-\frac d{d-2}}$
and all $a_1\in(3r,4r)$,
and since $d-1\geq3$ and $\R\vecy\cap a_1^{\frac1{d-1}}\fC_\vecv=(0,1)\vecy$,
the argument in Remark \ref{PPCVOLSMALLREM}
applies to give $p_\vecy(a_1^{\frac1{d-1}}\fC_\vecv)>\frac1{10}$.
Considering the contribution from these $\vecv$ in 
\eqref{CYLINDER2PTSLOWBDPROP2RES2PF1} we get
\begin{align*}
\Phi_\bn(\xi,\vecw,\vecz)\gg r^{-d}\int_0^{cr^{-\frac d{d-2}}}
\omega^{d-3}\,d\omega
\gg r^{-2d}=\xi^{-2}.
\end{align*}
This has been proved for all sufficiently large $\xi$;
the case of smaller $\xi$ is treated as in the proof of
\eqref{CYLINDER2PTSLOWBDPROP2RES1}.
\end{proof}

\begin{proof}[Proof of Proposition \ref{CYLD3OPTPROP}]
We will prove that there exists an absolute constant $0<c<1$
such that if $\xi$ is sufficiently large
and if $\|\vecw\|,\|\vecz\|\geq1-c\cdot s_3(\xi,\varphi)$,
then \eqref{CYLD3OPTPROPRES} holds.
This suffices, since the case of smaller $\xi$ can then be treated as in
the proof of \eqref{CYLINDER2PTSLOWBDPROP2RES1}.
We will successively impose conditions on $c$ being sufficiently
small.

Let us fix the choice
$\vecz=(z,0)$ and $\vecw=w(\cos\varphi,\sin\varphi)$
in \eqref{CYLINDER2PTSMAINTHMPF1}, with $\varphi\in[0,\pi]$ and $z,w\in(0,1)$.
We write
\begin{align*}
\vecv=(\cos\varpi,\sin\varpi\cos\omega,\sin\varpi\sin\omega),
\end{align*}
where we will keep $\varpi\in(0,\frac1{10})$ (thus $\frac{99}{100}<v_1<1$)
and $\omega\in(0,\pi)$.
Hence $\omega_\vecz=\omega$ and 
$\omega_\vecw=|\varphi-\omega|$.
Let us also write $t=\max(1-w,1-z)$.
Thus we are assuming $t\leq c\cdot s_3(\xi,\varphi)$.
As a variant of Lemma \ref{CONVENIENTCYLINDERLEM} one proves that,
for any $\vecv\in\S_1^{d-1}$
with $\varpi\in(0,\frac1{10})$ and $\omega\in(0,\pi)$,
the union $\fC_\vecv\cup\fC_\vecv'$
is contained in some right 2-dimensional cylinder
(viz.\ a rectangle) with $\bn$ in one of its bases,
which has height $\asymp r(t+\omega_\vecz^2+\omega_\vecw^2)$,
radius $\asymp r(t^{\frac12}+\sin\omega_\vecz+\sin\omega_\vecw)$,
and edge ratio $\asymp\min(1,\max(\frac{1-z}{\sin^2\omega_\vecz},
\frac{1-w}{\sin^2\omega_\vecw}))$ if $\omega\geq\varphi$,
otherwise edge ratio $1$.

Let us first assume $0\leq\varphi\leq\frac\pi2$.
By 
\eqref{CYLINDER2PTSLOWBDPROP2RES1} in Proposition \ref{CYLINDER2PTSLOWBDPROP2}
there is an absolute constant $c_1>0$ such that,
provided $c$ is sufficiently small,
$\Phi_\bn(\xi,\vecw,\vecz)\gg\xi^{-\frac43}$ holds
whenever $\varphi\leq c_1\xi^{-\frac13}$ 
(since $t\leq c\cdot s_3(\xi,\varphi)\leq c\xi^{-\frac23}$).
Hence we may now assume
$c_1\xi^{-\frac13}<\varphi\leq\frac\pi2$.
Then
\begin{align}\label{CYLD3OPTPROPPF1}
t\leq c\cdot s_3(\xi,\varphi)\leq c(\varphi\xi)^{-1}<cc_1^{-3}\varphi^{2}.
\end{align}
Let us consider the contribution 
in Lemma \ref{CYLINDER2PTSLOWBDPROP2RES1PF1LEM}
from $\vecv$ with $\varpi\in(0,\frac1{10})$
and $\omega\in(\frac43\varphi,\frac53\varphi)$.
For any such $\vecv$ we have 
$\sin\omega_\vecz\asymp\omega_\vecz\asymp\sin\omega_\vecw\asymp
\omega_\vecw\asymp\varphi$,
and thus (using also \eqref{CYLD3OPTPROPPF1})
$\fC_\vecv\cup\fC_\vecv'$ is contained in a rectangle
of base $\asymp r\varphi$,
height $\asymp r\varphi^2$ and edge ratio $\asymp t\varphi^{-2}$.
Recall from the proof of Lemma \ref{CYLINDER2PTSLOWBDPROP2RES1PF1LEM} that
$a_1=\vecp\cdot\vecv\asymp r$; hence
$|a_1^{\frac12}(\fC_\vecv\cup\fC_\vecv')|\asymp\xi\varphi^3>c_1^3$.
Now Corollary \ref{CONECYLCOR} gives,
using \eqref{CYLD3OPTPROPPF1} and provided that $c$ is sufficiently small:
$p^{(2)}(a_1^{\frac12}(\fC_\vecv\cup\fC_\vecv'))\asymp(\xi\varphi^3)^{-1}$.
On the other hand 
$\mu(\{\tM\in\Si_2\col\ta_1>(\frac23r)^{\frac32}\})\asymp r^{-3}=\xi^{-1}$
(cf.\ the proof of Lemma \ref{TRIVPCBOUNDLEM}).
Hence without keeping explicit track of implied constants, our usual
subtraction argument cannot be carried through unless $\varphi$
is sufficiently small.
As a simple way to remedy this, note that we may
fix an absolute constant $c_2>1$ such that
$\mu(\{\tM\in\Si_2\col\ta_1>c_2(\frac23r)^{\frac32}\})
<\frac12p^{(2)}(a_1^{\frac12}(\fC_\vecv\cup\fC_\vecv'))$ holds for all
 $\xi,\vecw,\vecz,\vecv$ which we are currently considering.
Let
\begin{align}\label{CYLD3OPTPROPPF3}
\Si_3(K):=\Bigl\{\nn(u) \aa(a) \kk \col u\in \mathcal{F}_N,\:
0<a_2\leq Ka_1,\: 0<a_3\leq\sfrac{2}{\sqrt 3}a_2,
\: \kk \in \SO(d) \Bigr\}
\end{align}
(thus $\Si_3=\Si_3(\frac2{\sqrt3})$, cf.\ \eqref{SIDEF});
then for \textit{any} $K>0$ the set $\Si_3(K)$ is
contained in a finite union of 
fundamental regions for $X_1$ (\cite{Borel}). %
Using this fact with $K=c_2$ 
in the proof of Lemma \ref{CYLINDER2PTSLOWBDPROP2RES1PF1LEM}
it follows that
\begin{align}\notag
p_\vecp(\fC) \gg \xi^{-1}\int_{\substack{\vecv\in\S_1^2\\ 
(v_1>\frac{99}{100})}}
\mu\Bigl(\Bigl\{\tM\in\Si_2\col
\ta_1\leq c_2(\sfrac23r)^{\frac32},\:
\Z^2\tM\cap a_1^{\frac12}(\fC_\vecv\cup\fC_\vecv')
=\emptyset\Bigr\}\Bigr)\,d\vecv
\hspace{50pt}
\\\label{CYLD3OPTPROPPF4}
\geq\xi^{-1}\int_0^{\frac1{10}}\int_{\frac43\varphi}^{\frac53\varphi}
\sfrac12\,p^{(2)}\Bigl(a_1^{\frac12}(\fC_\vecv\cup\fC_\vecv')\Bigr)
\,d\omega\,(\sin\varpi)\,d\varpi
\gg\xi^{-1}\int_{\frac43\varphi}^{\frac53\varphi}(\xi\varphi^3)^{-1}
\,d\omega\asymp\xi^{-2}\varphi^{-2},
\end{align}
as desired.

Next assume $\frac\pi2<\varphi\leq\pi$.
Recall that we write $\varphi_0=\pi-\varphi$.
First note that if $c$ is sufficiently small, then
there is an absolute constant $c_3\in(0,1)$ such that if
$\varphi_0\leq c_3\xi^{-1}$, 
then for any
$\vecv$ with $\varpi\in(0,\frac1{10})$ 
and $\omega\in(\varphi-c_3\xi^{-1},\varphi)$,
$a_1^{\frac12}(\fC_\vecv\cup\fC_\vecv')$ is contained in a 
rectangle of area $\leq\frac12$
(the proof of this makes use of
$t\leq c\cdot s_3(\xi,\varphi)=c\xi^{-2}$).
It follows that if $\xi$ is larger than a certain absolute constant
and if $\varphi_0\leq c_3\xi^{-1}$, then
\begin{align*}
p_\vecp(\fC)\gg\xi^{-1}\int_{\varphi-c_3\xi^{-1}}^\varphi\sfrac13\,d\omega
\gg\xi^{-2}.
\end{align*}
Hence we may now assume %
$\frac\pi2<\varphi<\pi-c_3\xi^{-1}$.
Then
\begin{align}\label{CYLD3OPTPROPPF2}
t\leq c\cdot s_3(\xi,\varphi)<cc_3^{-1}\varphi_0\xi^{-1}
<cc_3^{-2}\varphi_0^2.
\end{align}
Now consider the set of 
$\vecv$ with $\varpi\in(0,\frac1{10})$
and $\omega\in(\pi-\frac12\varphi_0,\pi-\frac13\varphi_0)$.
For any such $\vecv$ we have
$\omega_\vecz\asymp1$ and $\sin\omega_\vecz\asymp\sin\omega_\vecw
\asymp\omega_\vecw\asymp\varphi_0$,
and thus (using also \eqref{CYLD3OPTPROPPF2})
$\fC_\vecv\cup\fC_\vecv'$ is contained in a rectangle
of base $\asymp r\varphi_0$,
height $\asymp r$ and edge ratio $\asymp t\varphi_0^{-2}$.
Hence $|a_1^{\frac12}(\fC_\vecv\cup\fC_\vecv')|\asymp\xi\varphi_0>c_3$,
and thus Corollary \ref{CONECYLCOR} gives,
using \eqref{CYLD3OPTPROPPF2} and assuming that $c$ is sufficiently small:
$p^{(2)}(a_1^{\frac12}(\fC_\vecv\cup\fC_\vecv'))\asymp(\xi\varphi_0)^{-1}$.
Repeating the argument from \eqref{CYLD3OPTPROPPF3}, \eqref{CYLD3OPTPROPPF4}
we now obtain
\begin{align*}
p_\vecp(\fC)\gg\xi^{-1}\int_{\pi-\frac12\varphi_0}^{\pi-\frac13\varphi_0}
(\xi\varphi_0)^{-1}\,d\omega
\asymp\xi^{-2},
\end{align*}
as desired.
\end{proof}


\begin{thebibliography}{99}

\bibitem{athreyamargulis} 
J. S. Athreya and G. A. Margulis,
Logarithm laws for unipotent flows, I,
J.\ Mod.\ Dyn.\ \textbf{3} (2009), 359--378. 
%

%
%

%
%

\bibitem{Boca00}
F.P. Boca, C. Cobeli and A. Zaharescu, Distribution of lattice points visible from the origin.  Comm. Math. Phys.  \textbf{213}  (2000), 433--470.

\bibitem{Boca06}
F.P. Boca and A. Zaharescu, On the correlations of directions in the Euclidean plane.  Trans. Amer. Math. Soc.  \textbf{358 } (2006), 1797--1825 

%
%

%
%

%
%

%
%

%
%

\bibitem{bonnesen87}
T.\ Bonnesen and W.\ Fenchel,
\textit{Theory of convex bodies},
Translated from the German and edited by L.\ Boron, C.\ Christenson 
and B.\ Smith,
BCS Associates,
Moscow, ID,
1987.

\bibitem{Borel}
A. Borel, \textit{Introduction aux groupes arithm\'etiques},
Hermann, Paris, 1969.

\bibitem{Bourgain98}
J. Bourgain, F. Golse and B. Wennberg, On the distribution of free path lengths for the periodic Lorentz gas.  Comm. Math. Phys.  \textbf{190}  (1998), 491--508.

%
%

%
%
%
%

%
%

%
%

%
%

%
%
%
%
%

%
%

\bibitem{Daleyverejones}
D. J. Daley and D. Vere-Jones, 
\textit{An introduction to the theory of point processes, Vol.\ II},
2nd ed.,
Probability and its Applications,
Springer-Verlag, New York, 2008.

%
%
%

\bibitem{DRS}
W. Duke, Z. Rudnick, P. Sarnak,
Density of Integer Points on Affine Homogeneous Varieties,
Duke Math.\ J.\ \textbf{71} (1993), 143--179.

\bibitem{Elkies04}
N.D. Elkies and C.T. McMullen, Gaps in $\sqrt{n}\bmod 1$ and ergodic theory.  Duke Math. J.  \textbf{123}  (2004), 95--139.

%
%
%
%

%
%
%
%

\bibitem{Golse00}
F. Golse and B. Wennberg, On the distribution of free path lengths for the periodic Lorentz gas. II.  M2AN Math. Model. Numer. Anal.  \textbf{34}  (2000),  no. 6, 1151--1163.

%
%

%
%

%
%
%

%
%

%
%
%

\bibitem{John48}
F. John, Extremum problems with inequalities as subsidiary conditions,
\textit{Courant Anniversary Volume}, Interscience, New York, 1948,
pp.\ 187--204.
%

\bibitem{Kallenberg1984}
O. Kallenberg, An informal guide to the theory of conditioning in point processes,
Internat.\ Statist.\ Rev.\ \textbf{52} (1984), 151--164.

\bibitem{Kallenberg}
O. Kallenberg, \textit{Random measures},
4th ed., Akademie-Verlag, Berlin, 1986.

\bibitem{Marklof00}
J. Marklof,
The $n$-point correlations between values of a linear form,
Ergod.\ Th.\ \& Dynam.\ Sys.\ \textbf{20} (2000), 1127--1172.

%
%

\bibitem{partI}
J. Marklof and A. Str\"ombergsson, The distribution of free path lengths in the periodic Lorentz gas and related lattice point problems, 
Annals of Math. \textbf{172} (2010), 1949--2033.

\bibitem{partII}
J. Marklof and A. Str\"ombergsson, %
The Boltzmann-Grad limit of the periodic Lorentz gas, arXiv:0801.0612; to appear in the Annals of Mathematics.

\bibitem{partIII}
J. Marklof and A. Str\"ombergsson, Kinetic transport in the two-dimensional periodic Lorentz gas, arXiv:0804.0566; to appear in Nonlinearity.

\bibitem{partIV}
J. Marklof and A. Str\"ombergsson, The periodic Lorentz gas in the Boltzmann-Grad limit: Asymptotic estimates, 
to appear in GAFA.

%
%

%
%

%
%
%
%

\bibitem{Rogers58}
C. A. Rogers, 
Lattice covering of space: {T}he {M}inkowski-{H}lawka theorem,
Proc.\ London Math.\ Soc.\ \textbf{8} (1958), 447--465.

\bibitem{Schmidt58}
W. M. Schmidt,
The measure of the set of admissible lattices,
Proc.\ Amer.\ Math.\ Soc.\ \textbf{9} (1958), 390--403.

\bibitem{Schmidt59}
W. M. Schmidt,
Masstheorie in der {G}eometrie der {Z}ahlen,
Acta Math.\ \textbf{102} (1959), 159--224.

\bibitem{schneider}
R. Schneider, \textit{Convex bodies: the {B}runn-{M}inkowski theory},
Cambridge University Press, Cambridge, 1993.

%
%

\bibitem{Shimura}
G. Shimura,
\textit{Introduction to the Arithmetic Theory of Automorphic Forms},
Iwanami Shoten and Princeton University Press, 1971.
%

%
%
%
%

\bibitem{siegel45}
C. L. Siegel,
A mean value theorem in geometry of numbers,
Ann.\ of Math.\ \textbf{46} (1945), 340--347.
%

\bibitem{siegel}
C. L. Siegel,
Lectures on the Geometry of Numbers,
Springer-Verlag, Berlin-Heidelberg-New York, 1989.

%
%
%

\bibitem{dSwKjM95}
D.\ Stoyan, W.\ S.\ Kendall and J.\ Mecke, 
\textit{Stochastic geometry and its applications},
John Wiley \& Sons,
Chichester, 1995.

\bibitem{SV}
A.\ Str\"ombergsson, A.\ Venkatesh,
Small solutions to linear congruences and Hecke equidistribution, 
Acta Arith., \textbf{118} (2005), 41-78.

%
%
%

%
%
%
%

%
%
%
\end{thebibliography}
\end{document}